\documentclass[12pt,a4paper]{amsart} 
\usepackage[utf8]{inputenc}
\usepackage[english]{babel}

\usepackage{amsmath}
\usepackage{amsfonts}
\usepackage{amssymb}
\usepackage{color}

\usepackage{amsthm}

\newtheorem{theorem}{Theorem} [section]
\newtheorem{proposition}[theorem]{Proposition}
 
\newtheorem{lemma}[theorem]{Lemma}

\newtheorem{open}{Open Question}

\newcounter{defno}

\usepackage{graphics}
\usepackage{epsfig}

\usepackage{url}
\usepackage{hyperref}

\usepackage[a4paper,top=2.5cm,bottom=2.8cm,
left=3.2cm,right=3.2cm]{geometry}
\markleft{\hfill \textsc{Relationships between Central Quadrilateral and Reference Quadrilateral} \hfill}
\newcommand{\degrees}{^\circ}
\newcommand{\msg}[1]{\textcolor{red}{#1}}
\newcommand{\no}[1]{\textcolor{red}{#1}}

\long\def\void#1{}

\begin{document}
International Journal of  Computer Discovered Mathematics (IJCDM) \\
ISSN 2367-7775 \copyright IJCDM \\
Volume 10, 2025, pp. 349--382  \\
web: \url{http://www.journal-1.eu/} \\
Received 13 July 2025. Published on-line 21 Aug. 2025 \\ 

\copyright The Author(s) This article is published 
with open access.\footnote{This article is distributed under the terms of the Creative Commons Attribution License which permits any use, distribution, and reproduction in any medium, provided the original author(s) and the source are credited.} \\
\bigskip
\bigskip

\begin{center}
	{\Large \textbf{More Shapes of Central Quadrilaterals}} \\
	\medskip
	\bigskip
        \bigskip

	\textsc{Stanley Rabinowitz$^a$ and Ercole Suppa$^b$} \\

	$^a$ 545 Elm St Unit 1,  Milford, New Hampshire 03055, USA \\
	e-mail: \href{mailto:stan.rabinowitz@comcast.net}{stan.rabinowitz@comcast.net}\footnote{Corresponding author} \\
	web: \url{http://www.StanleyRabinowitz.com/} \\
	
	$^b$ Via B. Croce 54, 64100 Teramo, Italia \\
	e-mail: \href{mailto:ercolesuppa@gmail.com}{ercolesuppa@gmail.com} \\
	web: \url{https://www.esuppa.it} \\

\bigskip

\end{center}
\bigskip
\bigskip

\textbf{Abstract.}
Let $E$ be a point in the plane of a convex quadrilateral $ABCD$.
The lines from $E$ to the vertices of the quadrilateral form four triangles.
If we locate a triangle center in each of these triangles, the four triangle
centers form another quadrilateral called a central quadrilateral.
For each of various shaped quadrilaterals, and each of 1000 different triangle
centers, and for various choices for $E$, we examine the shape of the central quadrilateral.
Using a computer, we determine when the central quadrilateral
has a special shape, such as being a rhombus or a cyclic quadrilateral.
A typical result is the following.
Let $E$ be the centroid of equidiagonal quadrilateral $ABCD$.
Let $F$, $G$, $H$, and $I$ be the $X_{591}$-points of triangles $\triangle ABE$, $\triangle BCE$, $\triangle CDE$,
and $\triangle DAE$, respectively. Then $FGHI$ is an orthodiagonal quadrilateral.



\newcommand{\ru}{\rule[-7pt]{0pt}{20pt}}


\bigskip
\bigskip
%
\section{Introduction}
\label{section:introduction}

In this study, $ABCD$ always represents a convex quadrilateral known as the \emph{reference quadrilateral}.
A point $E$ in the plane of the quadrilateral (not on the boundary) is chosen and will be called the \emph{radiator}.
The radiator can be an arbitrary point or it can be a notable point associated with the quadrilateral.
Lines are drawn from the radiator to the vertices of the reference quadrilateral forming four triangles
with the sides of the quadrilateral as shown in Figure \ref{fig:centerPointTriangles}.
These triangles will be called the \emph{radial triangles}.

\begin{figure}[h!t]
\centering
\includegraphics[width=0.4\linewidth]{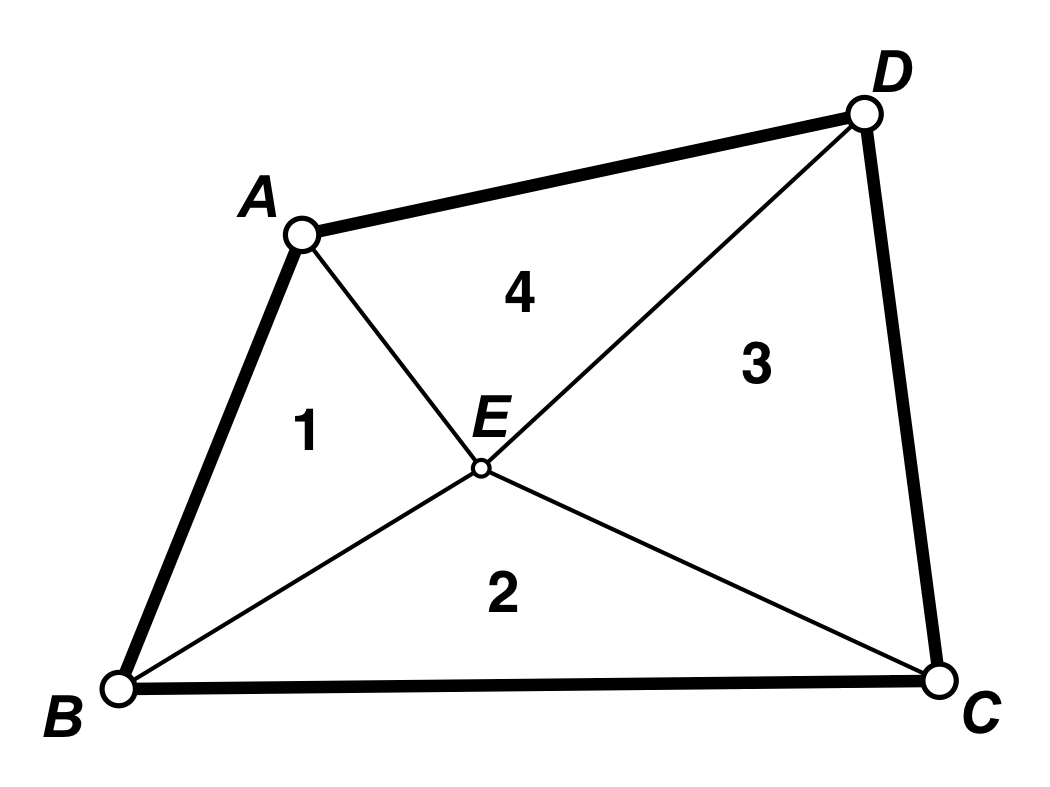}
\caption{Radial Triangles}
\label{fig:centerPointTriangles}
\end{figure}

In the figure, the radial triangles have been numbered in a counterclockwise order starting with side $AB$:
$\triangle ABE$, $\triangle BCE$, $\triangle CDE$, $\triangle DAE$.
Triangle centers (such as the incenter, centroid, or circumcenter) are selected in each triangle.
The same type of triangle center is used with each radial triangle.
In order, the names of these points are $F$, $G$, $H$, and $I$ as shown in Figure \ref{fig:centralQuadrilateral}.
These four centers form a quadrilateral $FGHI$ that will be called the \emph{central quadrilateral} (of quadrilateral $ABCD$
with respect to $E$). Quadrilateral $FGHI$ need not be convex.

\begin{figure}[h!t]
\centering
\includegraphics[width=0.4\linewidth]{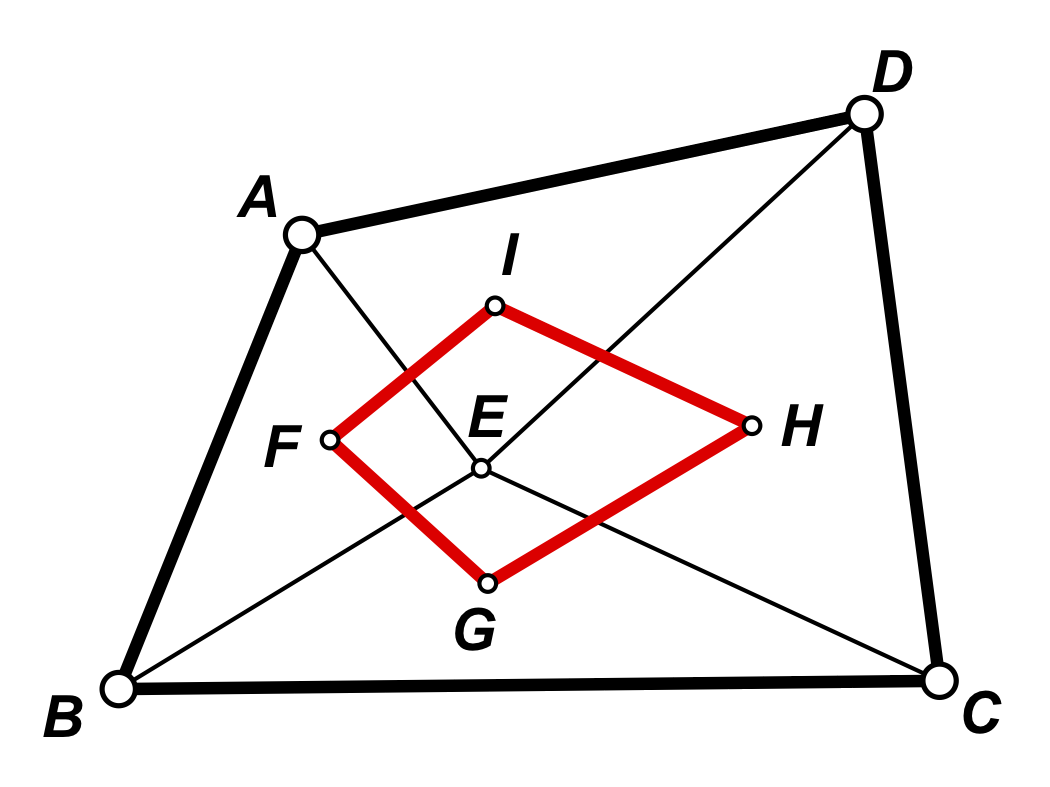}
\caption{Central Quadrilateral}
\label{fig:centralQuadrilateral}
\end{figure}

The purpose of this paper is to determine when a central quadrilateral has a special shape,
such as being a rhombus or a cyclic quadrilateral.

\newpage

\section{Types of Quadrilaterals Studied}
\label{section:quadrilaterals}

We are only interested in reference quadrilaterals that have a certain amount of symmetry.
For example, we excluded bilateral quadrilaterals (those with two equal sides),
bisect-diagonal quadrilaterals (where one diagonal bisects another), right kites,
right trapezoids, and golden rectangles.
The types of quadrilaterals we studied are shown in Table \ref{table:quadrilaterals}.
The sides of the quadrilateral, in order, have lengths $a$, $b$, $c$, and $d$.
The diagonals have lengths $p$ and $q$.
The measures of the angles of the quadrilateral, in order, are $A$, $B$, $C$, and $D$.

\begin{table}[ht!]
\caption{}
\label{table:quadrilaterals}
\begin{center}
\footnotesize
\begin{tabular}{|l|l|l|}\hline
\multicolumn{3}{|c|}{\textbf{\color{blue}\Large \strut Types of Quadrilaterals Considered}}\\ \hline
\textbf{Quadrilateral Type}&\textbf{Geometric Definition}&\textbf{Algebraic Condition}\\ \hline
general&convex&none\\ \hline
cyclic&has a circumcircle&$A+C=B+D$\\ \hline
tangential&has an incircle&$a+c=b+d$\\ \hline
extangential&has an excircle&$a+b=c+d$\\ \hline
parallelogram&opposite sides parallel&$a=c$, $b=d$\\ \hline
equalProdOpp&product of opposite sides equal&$ac=bd$\\ \hline
equalProdAdj&product of adjacent sides equal&$ab=cd$\\ \hline
orthodiagonal&diagonals are perpendicular&$a^2+c^2=b^2+d^2$\\ \hline
equidiagonal&diagonals have the same length&$p=q$\\ \hline
Pythagorean&equal sum of squares, adjacent sides&$a^2+b^2=c^2+d^2$\\ \hline
kite&two pair adjacent equal sides&$a=b$, $c=d$\\ \hline
trapezoid&one pair of opposite sides parallel&$A+B=C+D$\\ \hline
rhombus&equilateral&$a=b=c=d$\\ \hline
rectangle&equiangular&$A=B=C=D$\\ \hline
Hjelmslev&two opposite right angles&$A=C=90^\circ$\\ \hline
isosceles trapezoid&trapezoid with two equal sides&$A=B$, $C=D$\\ \hline
APquad&sides in arithmetic progression&$d-c=c-b=b-a$\\ \hline
\end{tabular}
\end{center}
\end{table}

The following combinations of entries in the above list were also considered:
bicentric quadrilaterals (cyclic and tangential), exbicentric quadrilaterals (cyclic and extangential),
bicentric trapezoids, cyclic orthodiagonal quadrilaterals, equidiagonal kites,
equidiagonal orthodiagonal quadrilaterals, equidiagonal orthodiagonal trapezoids,
harmonic quadrilaterals (cyclic and equalProdOpp), orthodiagonal trapezoids, tangential trapezoids,
and squares (equiangular rhombi).

So, in addition to the general convex quadrilateral, a total of 28 types of quadrilaterals
were considered in this study.

A graph of the types of quadrilaterals considered is shown in Figure \ref{fig:quadShapes}.
An arrow from A to B means that any quadrilateral of type B is also of type A.
For example: all squares are rectangles and all kites are orthodiagonal.
If a directed path leads from a quadrilateral of type A to a quadrilateral of type B, then we will
say that A is an \emph{ancestor} of B. For example, an equidiagonal quadrilateral is an ancestor of a rectangle.
In other words, all rectangles are equidiagonal.

\begin{figure}[h!t]
\centering
\scalebox{1}[1.5]{\includegraphics[width=1\linewidth]{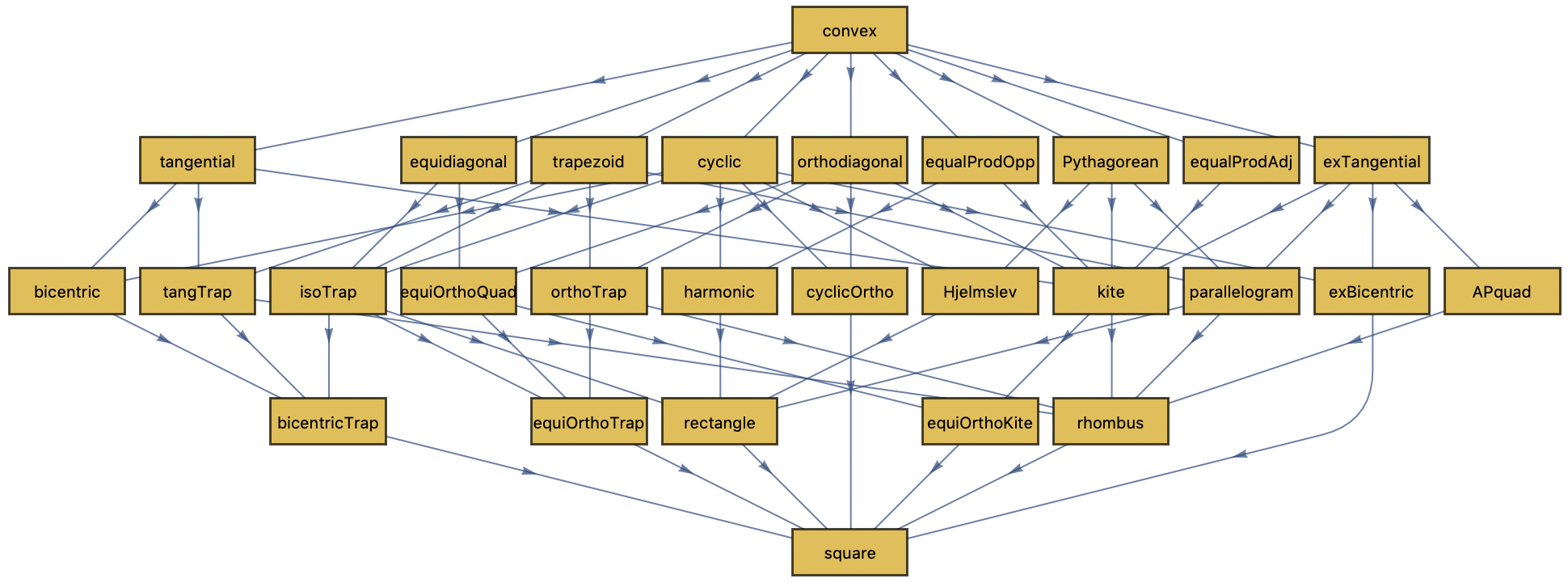}}
\caption{Quadrilateral Shapes}
\label{fig:quadShapes}
\end{figure}

Unless otherwise specified, when we give a theorem about a quadrilateral, we will omit an entry
for a particular shape quadrilateral if the property is known to be true for an ancestor of that quadrilateral.

We do not include results where the central quadrilateral degenerates to a line segment or a point.

\newpage

\section{Methodology}
\label{section:methodology}

In this study, we locate triangle centers in the four radial
triangles. We use Clark Kimberling's definition of a triangle center \cite{KimberlingA}.
More details can be found in section 3 of \cite{relationships}.

We used a computer program called GeometricExplorer to examine the shape of the central quadrilateral.
Starting with each type of quadrilateral listed in
Figure~\ref{fig:quadShapes} for the reference quadrilateral, we picked various choices
for point $E$, the radiator. The types of radiators studied are shown in Table \ref{table:radiators}.

\begin{table}[ht!]
\caption{}
\label{table:radiators}
\begin{center}
\begin{tabular}{|l|l|}
\hline
\multicolumn{2}{|c|}{\Large \strut \textbf{\color{blue}Points Used as Radiators}}\\
\hline
\textbf{name}&\textbf{description}\\ \hline
arbitrary point&any point in the plane of $ABCD$\\ \hline
diagonal point&intersection of the diagonals (QG--P1)\\ \hline
Poncelet point&(QA--P2)\\ \hline
Steiner point&(QA--P3)\\ \hline
vertex centroid&(QA-P1)\\ \hline
area centroid&(QG-P4)\\ \hline
\void{
quasi circumcenter$\dagger$&(QG-P5)\\ \hline
quasi orthocenter$\dagger$&(QG-P6)\\ \hline
quasi nine-point center$\dagger$&(QG-P7)\\ \hline
quasi incenter$\dagger$&\\ \hline
Kirikami center$\dagger$&(QG-P15)\\ \hline
Miquel point$\dagger$&(QL-P1)\\ \hline
\multicolumn{2}{|l|}{$\dagger$ \small This point might be omitted from the final paper.}\\ \hline
}
\end{tabular}
\end{center}
\end{table}

A code in parentheses represents the name for the point as listed in
the Encyclopedia of Quadri-Figures~\cite{EQF}.

For each $n$ from 1 to 1000, we placed center $X_n$
in each of the radial triangles of the reference quadrilateral.
The program then analyzes the central quadrilateral formed by these four centers
and reports if the central quadrilateral has a special shape.
Points at infinity were omitted.
GeometricExplorer uses numerical coordinates (to 15 digits of precision) for locating
all the points. This does not constitute a proof that the result is correct,
but gives us compelling evidence for the validity of the result.

If a theorem in this paper is accompanied by a figure, this means that the
figure was drawn using either Geometer's Sketchpad or GeoGebra.
In either case, we used the drawing program to dynamically vary the points
in the figure. Noticing that the result remains true as the points vary offers
further evidence that the theorem is true.

To prove the results that we have discovered, we use geometric methods, when possible.
If we could not find a purely geometrical proof, we turned to analytic methods using
barycentric coordinates and performing exact symbolic computation using Mathematica.
All proofs can be found in the Mathematica notebooks included in the supplementary material
associated with the paper.

If our only ``proof'' of a particular relationship is by using numerical calculations (and not using exact computation),
then we have colored the center \no{red} in the table of relationships.

\newpage


\section{Results Using an Arbitrary Point}
\label{section:arbitraryPoint}

In this configuration, the radiator, $E$, is any point in the plane of
the reference quadrilateral $ABCD$, not on the boundary.

Our computer analysis found only one special shape  associated with all quadrilaterals
when $E$ is an arbitrary point in the plane. We examined all the types of
quadrilaterals listed in Table 1 and all triangle centers from $X_1$ to $X_{1000}$.
The special shape occurs only when the chosen center is $X_2$, the centroid.
The result is shown below.

\bigskip
\begin{center}
\begin{tabular}{|l|p{2.5in}|}
\hline
\multicolumn{2}{|c|}{\textbf{\color{blue}\large \strut Central Quadrilateral of a General Quadrilateral}}\\ \hline
\textbf{Shape of central quadrilateral}&\textbf{center}\\ \hline
\ru parallelogram&2\\ \hline
\end{tabular}
\end{center}


\begin{theorem}
\label{thm:genArbX2}
Let $E$ be an arbitrary point in the plane of convex quadrilateral $ABCD$.
Let $F$, $G$, $H$, and $I$ be the centroids of triangles $\triangle ABE$, $\triangle BCE$, $\triangle CDE$,
and $\triangle DAE$, respectively (Figure~\ref{fig:genArbX2}).
Then $FGHI$ is a parallelogram.
The sides of the parallelogram are parallel to the diagonals of $ABCD$.
\end{theorem}

\begin{figure}[h!t]
\centering
\includegraphics[width=0.35\linewidth]{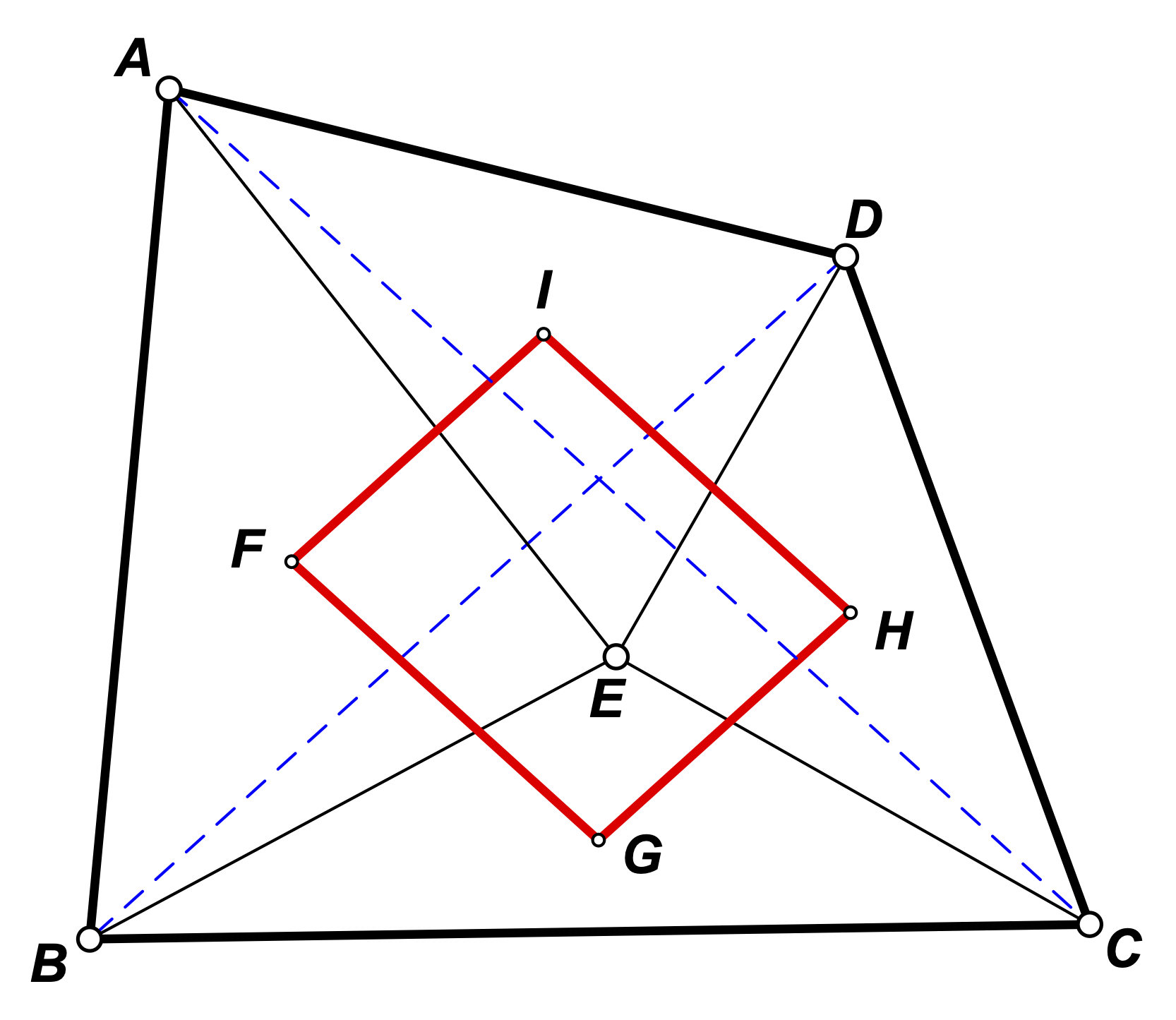}
\caption{General quadrilateral: centroids $\implies$ parallelogram}
\label{fig:genArbX2}
\end{figure}

A geometric proof is straightforward. We start with a lemma.

\begin{lemma}
\label{lemma:genArbX2Lemma}
Let $E$ be an arbitrary point in the plane of $\triangle ABC$.
Let $F$ be the centroid of triangle $\triangle ABE$ and let $G$ be the centroid of $\triangle ACE$(Figure~\ref{fig:genArbX2Lemma}).
Then $FG\parallel BC$ and $FG=BC/3$.
\end{lemma}

\begin{figure}[h!t]
\centering
\includegraphics[width=0.3\linewidth]{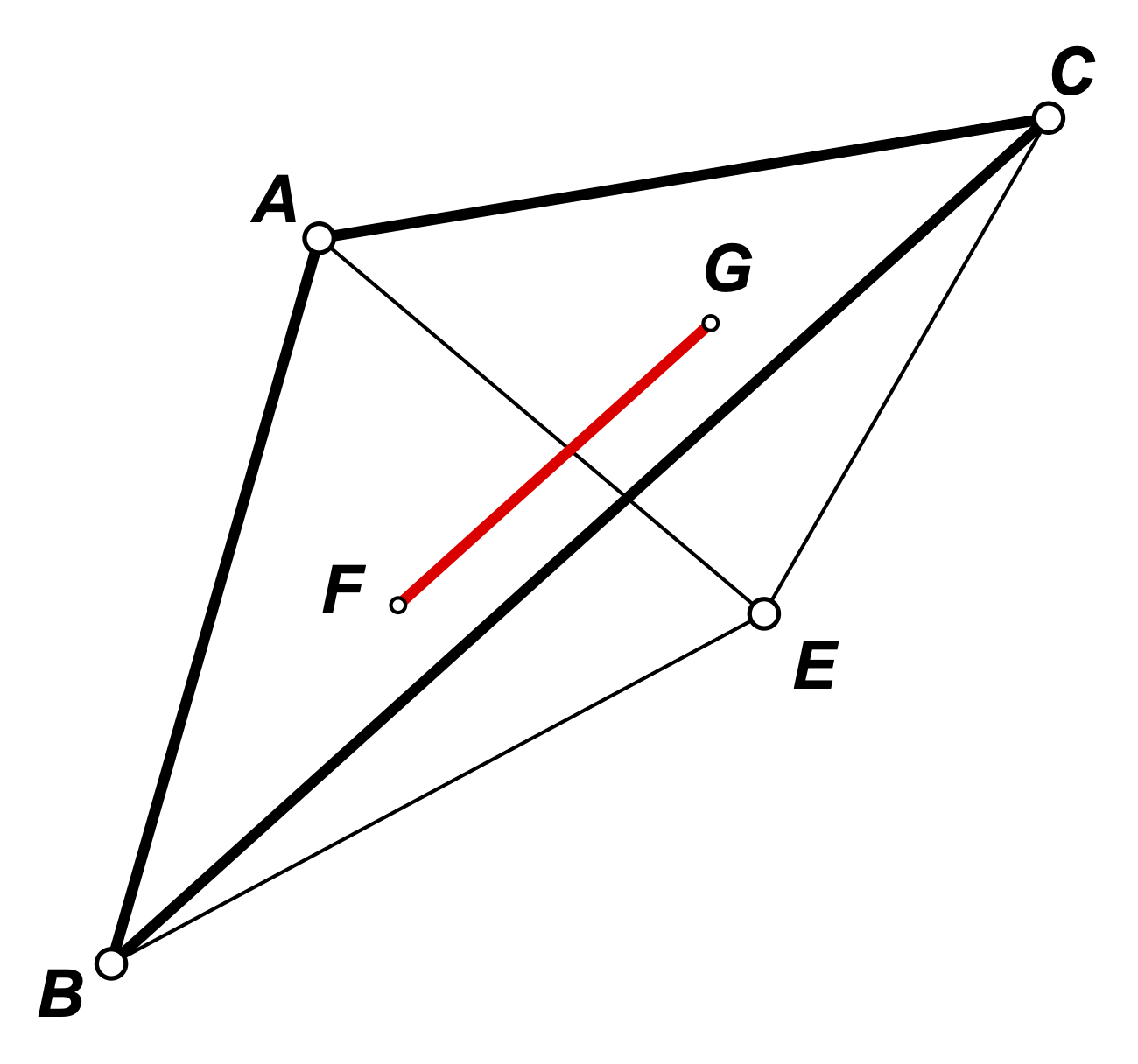}
\caption{}
\label{fig:genArbX2Lemma}
\end{figure}

\newpage

\begin{proof}
Let $AP$ and $AQ$ be the medians of triangles $AEB$ and $AEC$, respectively (Figure~\ref{fig:genArbX2Proof}).
Then $AF/FP=2$ and $AG/GQ=2$ which implies $FG\parallel PQ$ and $FG=\frac23 PQ$.
Since $P$ and $Q$ are the midpoints of $EB$ and $EC$, respectively, we have $BP/PE=1$ and $CQ/QE=1$
which implies that $PQ\parallel BC$ and $PQ=BC/2$.
Thus, $FG\parallel BC$ and $FG=\frac23 PQ=\frac23 (\frac12 BC)=\frac13 BC$.

\begin{figure}[h!t]
\centering
\includegraphics[width=0.3\linewidth]{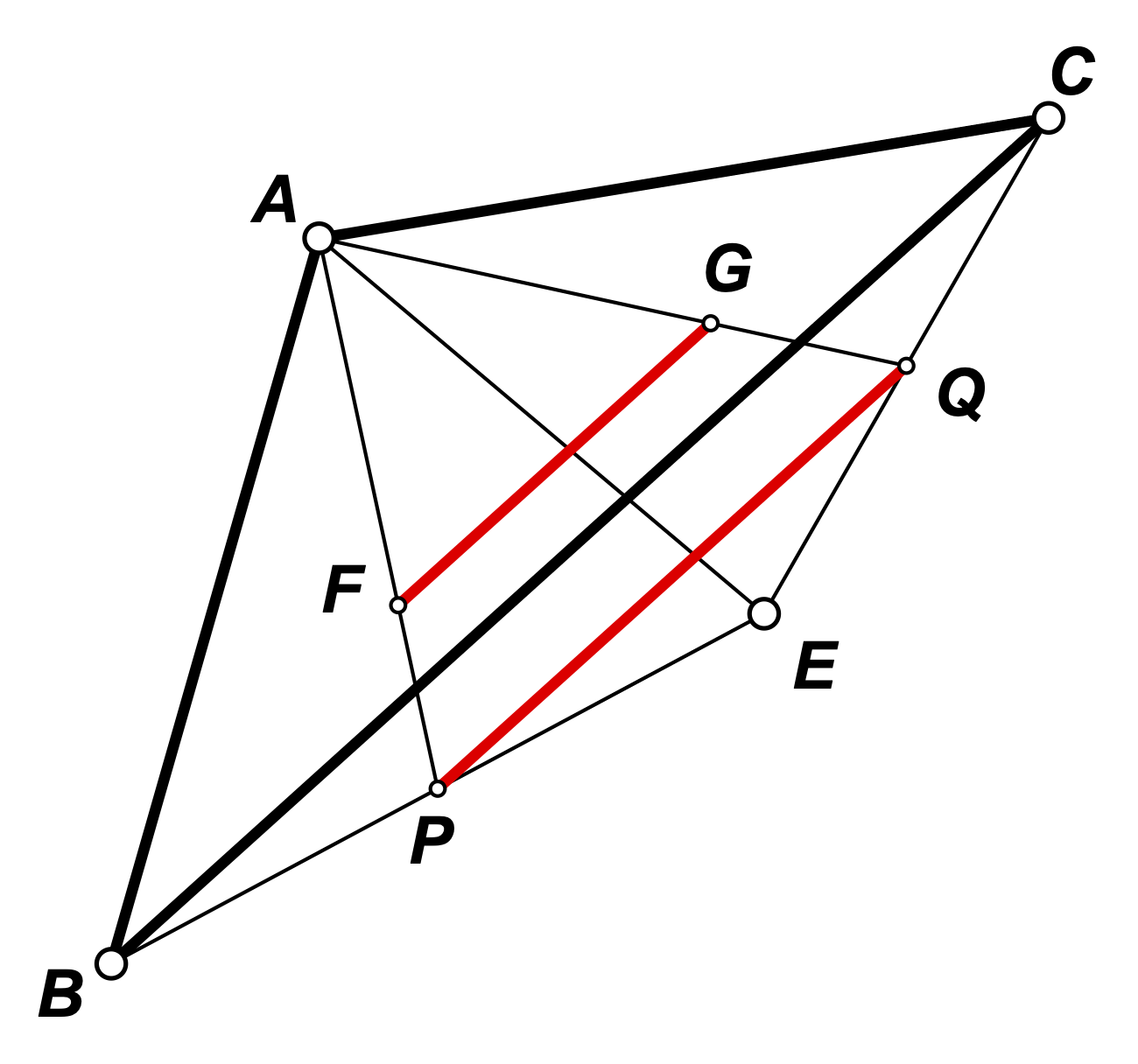}
\caption{}
\label{fig:genArbX2Proof}
\end{figure}
\end{proof}

Now on to the proof of Theorem~\ref{thm:genArbX2}. Refer back to Figure~\ref{fig:genArbX2}.

\begin{proof}
By Lemma~\ref{lemma:genArbX2Lemma}, $FI\parallel BD$.
Similarly, $GH\parallel BD$. Thus, $FI\parallel GH$.
In the same way, $FG\parallel IH$. Hence, $FGHI$ is a parallelogram.
\end{proof}

Our computer study found special shapes associated with equidiagonal and orthodiagonal quadrilaterals.
The results are shown in the following three tables.

\bigskip
\begin{center}
\begin{tabular}{|l|p{2.5in}|}
\hline
\multicolumn{2}{|c|}{\textbf{\color{blue}\large \strut Central Quadrilateral of an Equidiagonal Quadrilateral}}\\ \hline
\textbf{Shape of central quadrilateral}&\textbf{center}\\ \hline
\ru rhombus&2\\ \hline
\end{tabular}
\end{center}

\begin{theorem}
\label{thm:equiArbX2}
Let $E$ be an arbitrary point in the plane of equidiagonal quadrilateral $ABCD$.
Let $F$, $G$, $H$, and $I$ be the centroids of triangles $\triangle ABE$, $\triangle BCE$, $\triangle CDE$,
and $\triangle DAE$, respectively (Figure~\ref{fig:genArbX2}).
Then $FGHI$ is a rhombus.
\end{theorem}

\begin{proof}
By Lemma~\ref{lemma:genArbX2Lemma}, $FI=\frac13 BD$ and $FG=\frac13 AC=\frac13 BD$, so $FI=FG$.
But a parallelogram with two equal adjacent sides is a rhombus.
\end{proof}

\bigskip
\begin{center}
\begin{tabular}{|l|p{2.5in}|}
\hline
\multicolumn{2}{|c|}{\textbf{\color{blue}\large \strut Central Quadrilateral of an Orthodiagonal Quadrilateral}}\\ \hline
\textbf{Shape of central quadrilateral}&\textbf{center}\\ \hline
\ru rectangle&2\\ \hline
\end{tabular}
\end{center}

\begin{theorem}
\label{thm:orthoArbX2}
Let $E$ be an arbitrary point in the plane of orthodiagonal quadrilateral $ABCD$.
Let $F$, $G$, $H$, and $I$ be the centroids of triangles $\triangle ABE$, $\triangle BCE$, $\triangle CDE$,
and $\triangle DAE$, respectively (Figure~\ref{fig:genArbX2}).
Then $FGHI$ is a rectangle.
\end{theorem}

\begin{proof}
By Lemma~\ref{lemma:genArbX2Lemma}, $FI\parallel BD$ and $FG\parallel AC$. Since $BD\perp AC$,
we can conclude that $FI\perp FG$.
But a parallelogram with two perpendicular adjacent sides is a rectangle.
\end{proof}

\bigskip
\begin{center}
\begin{tabular}{|l|p{2.5in}|}
\hline
\multicolumn{2}{|c|}{\textbf{\color{blue}\large \strut Central Quadr of an Equidiagonal Orthodiagonal Quadr}}\\ \hline
\textbf{Shape of central quadrilateral}&\textbf{center}\\ \hline
\ru square&2\\ \hline
\end{tabular}
\end{center}

\begin{theorem}
\label{thm:equiOrthoArbX2}
Let $E$ be an arbitrary point in the plane of equidiagonal orthodiagonal quadrilateral $ABCD$.
Let $F$, $G$, $H$, and $I$ be the centroids of triangles $\triangle ABE$, $\triangle BCE$, $\triangle CDE$,
and $\triangle DAE$, respectively (Figure~\ref{fig:genArbX2}).
Then $FGHI$ is a square.
\end{theorem}

\begin{proof}
By Theorem~\ref{thm:equiArbX2}, $FGHI$ is a rhombus.
By Theorem~\ref{thm:orthoArbX2}, $FGHI$ is a rectangle.
But a figure that is both a rhombus and a rectangle must be a square.
\end{proof}

\begin{open}
Is the centroid the only triangle center for which the central quadrilateral is a parallelogram?
\end{open}

Our computer study found several interesting results for the central quadrilateral of a rectangle.
These are shown in the following table.

The symbol $\mathbb{S}$ denotes the set of all triangle centers that lie on the Euler line of the reference triangle and have constant Shinagawa coefficients.
Shinagawa coefficients are defined in \cite{ETC}.
The first few $n$ for which $X_n$ has constant Shinagawa coefficients are $n=$2, 3, 4, 5, 20, 140, 376, 381, 382, 546-550, 631, and 632.

\bigskip
\begin{center}
\begin{tabular}{|l|p{3.3in}|}
\hline
\multicolumn{2}{|c|}{\textbf{\color{blue}\large \strut Central Quadrilaterals of Rectangles}}\\ \hline
\textbf{Shape of central quad}&\textbf{centers}\\ \hline
\ru orthodiagonal&$\mathbb{S}$\\ \hline
\end{tabular}
\end{center}

\newpage

\begin{theorem}
\label{thm:rectArbXShinagawa}
Let $E$ be an arbitrary point in the plane of rectangle $ABCD$.
Let $X$ be a triangle center with constant Shinagawa coefficients.
Let $F$, $G$, $H$, and $I$ be the $X$-points of triangles $\triangle ABE$, $\triangle BCE$, $\triangle CDE$,
and $\triangle DAE$, respectively.
Then $FGHI$ is orthodiagonal. The diagonals of $FGHI$ are parallel to the sides of $ABCD$.
Figure~\ref{fig:rectArbX20} shows the case when $X$ is the de Longchamps Point ($X_{20}$).
Figure~\ref{fig:rectArbX381} shows the case when $X$ is the $X_{381}$ point.
\end{theorem}

\begin{figure}[h!t]
\centering
\includegraphics[width=0.7\linewidth]{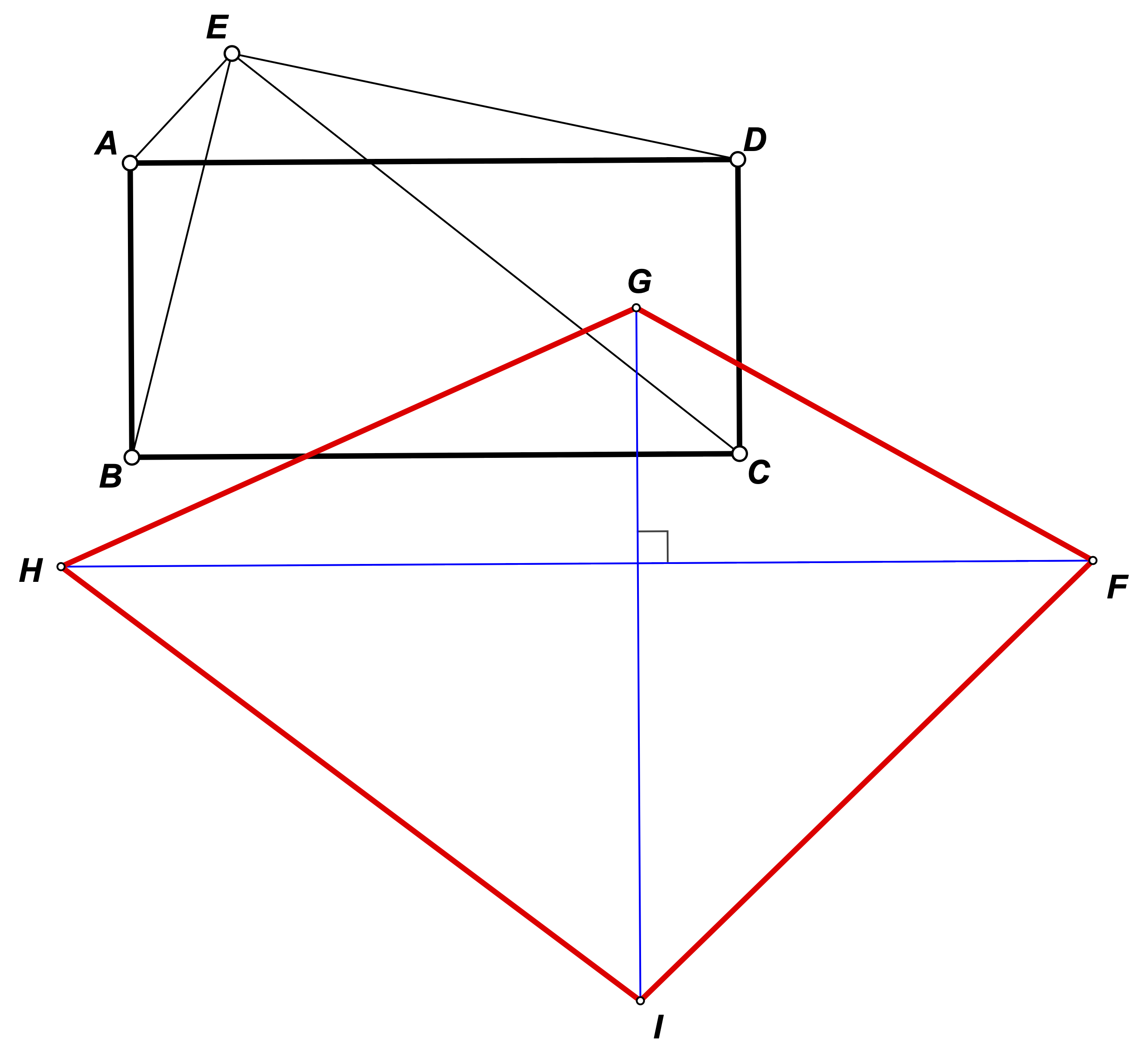}
\caption{rectangle, $X_{20}$-points $\implies$ orthodiagonal}
\label{fig:rectArbX20}
\end{figure}

\begin{figure}[h!t]
\centering
\includegraphics[width=0.7\linewidth]{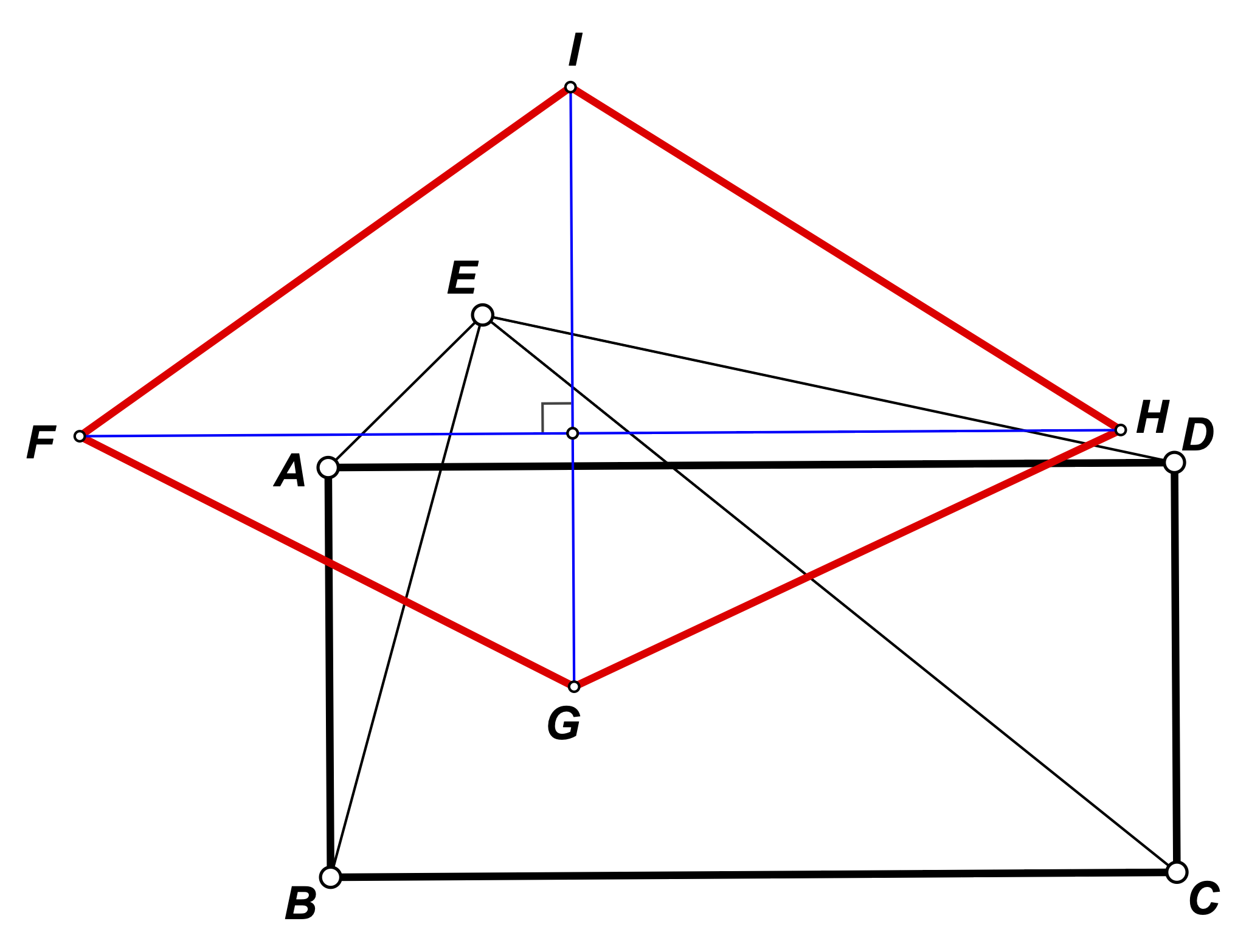}
\caption{rectangle, $X_{381}$-points $\implies$ orthodiagonal}
\label{fig:rectArbX381}
\end{figure}

\newpage

Our computer study found several interesting results for the central quadrilateral of a square.
These are shown in the following table.
Results that are true for rectangles or equidiagonal orthodiagonal quadrilaterals are omitted.

\begin{center}
\begin{tabular}{|l|p{3.3in}|}
\hline
\multicolumn{2}{|c|}{\textbf{\color{blue}\large \strut Central Quadrilaterals of Squares}}\\ \hline
\textbf{Shape of central quad}&\textbf{centers}\\ \hline
\ru square&2\\ \hline
\ru cyclic&99, 925\\ \hline
\ru equidiagonal orthodiagonal&372, 373, 640\\ \hline
\end{tabular}
\end{center}

\begin{theorem}
\label{thm:sqArbX2}
Let $E$ be an arbitrary point in the plane of square $ABCD$.
Let $F$, $G$, $H$, and $I$ be the centroids of triangles $\triangle ABE$, $\triangle BCE$, $\triangle CDE$,
and $\triangle DAE$, respectively (Figure~\ref{fig:sqArbX2}).
Then $FGHI$ is a square.
\end{theorem}

\begin{figure}[h!t]
\centering
\includegraphics[width=0.3\linewidth]{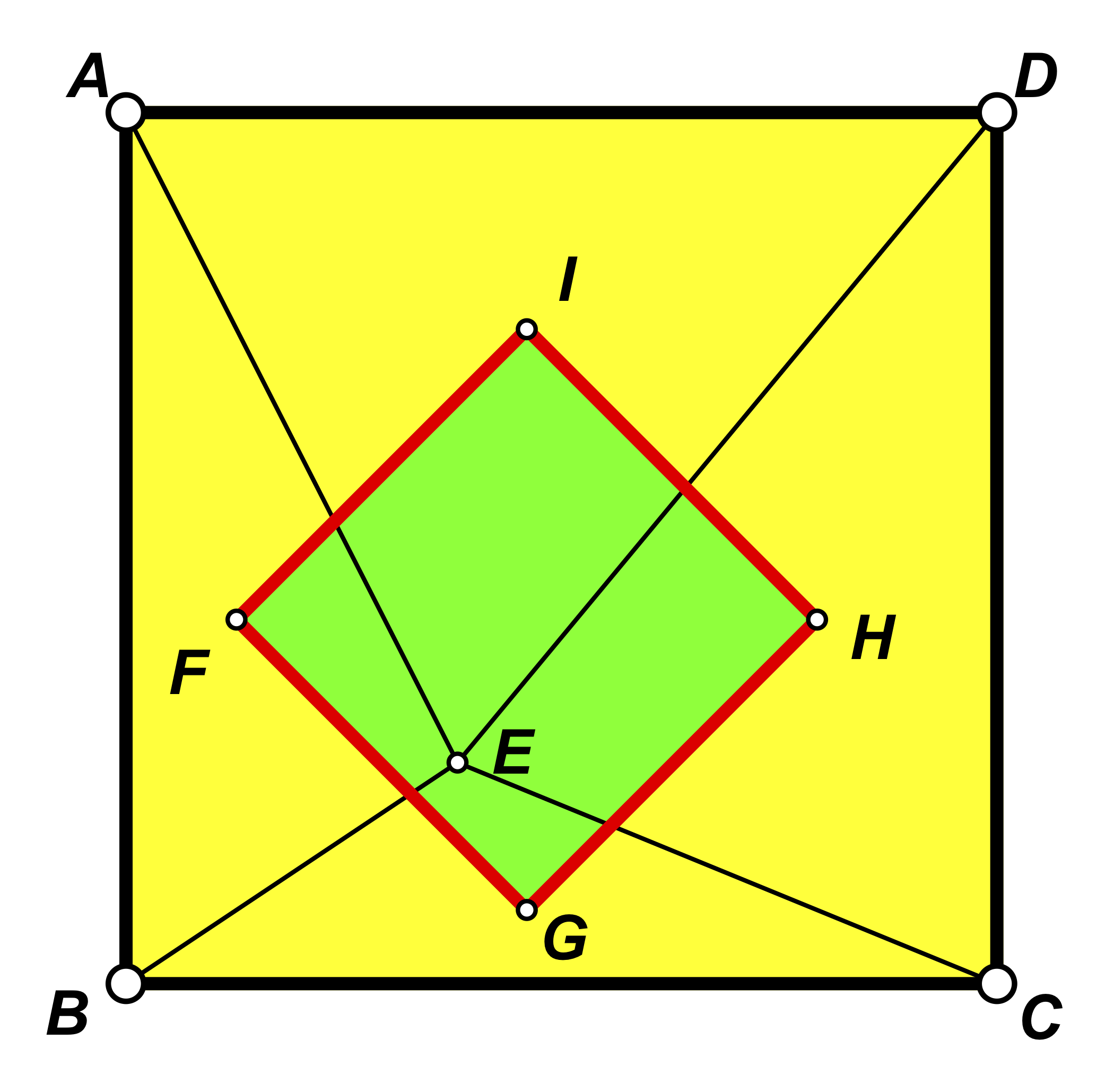}
\caption{square, $X_{2}$-points $\implies$ square}
\label{fig:sqArbX2}
\end{figure}

\begin{proof}
This is Theorem 6.2 in \cite{relationships}.
\end{proof}


\begin{theorem}
\label{thm:sqArbX99}
Let $E$ be an arbitrary point in the plane of square $ABCD$.
Let $n$ be 99 or 925.
Let $F$, $G$, $H$, and $I$ be the $X_{n}$-points of triangles $\triangle ABE$, $\triangle BCE$, $\triangle CDE$,
and $\triangle DAE$, respectively.
Then $FGHI$ is cyclic and $E$ lies on the circle $FGHI$.
(Figure~\ref{fig:sqArbX99} shows the case when $n=99$.)
\end{theorem}

\begin{figure}[h!t]
\centering
\includegraphics[width=0.49\linewidth]{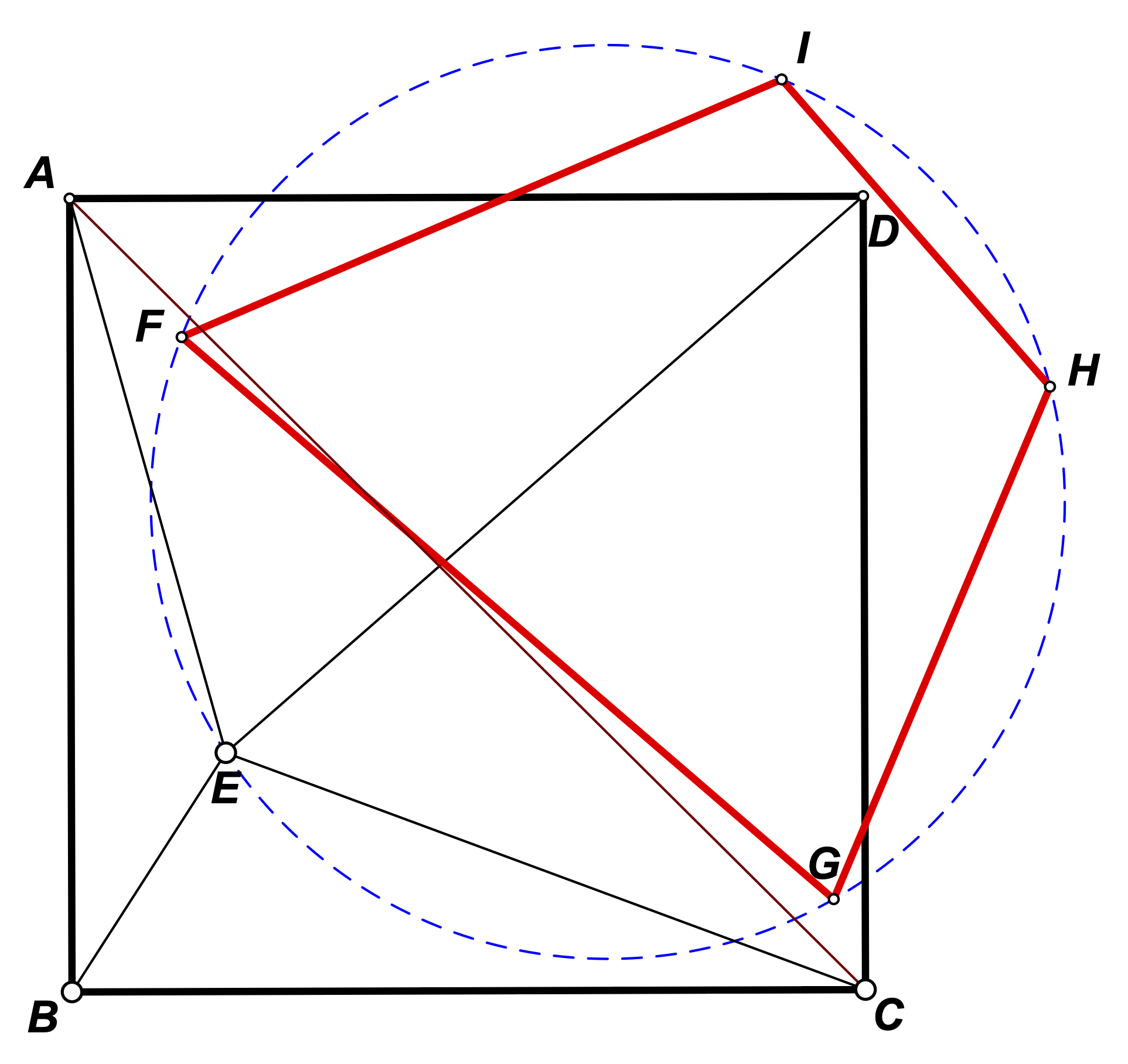}
\caption{square, $X_{99}$-points $\implies$ cyclic}
\label{fig:sqArbX99}
\end{figure}

\begin{theorem}
\label{thm:sqArbX372}
Let $E$ be an arbitrary point in the plane of square $ABCD$.
Let $n$ be 372 or 640.
Let $F$, $G$, $H$, and $I$ be the $X_{n}$-points of triangles $\triangle ABE$, $\triangle BCE$, $\triangle CDE$,
and $\triangle DAE$, respectively (Figure~\ref{fig:sqArbX372} shows the case when $n=372$).
Then $FGHI$ is an equidiagonal orthodiagonal quadrilateral.
The diagonals of $FGHI$ are parallel to the sides of $ABCD$.
\end{theorem}

\begin{figure}[h!t]
\centering
\includegraphics[width=0.4\linewidth]{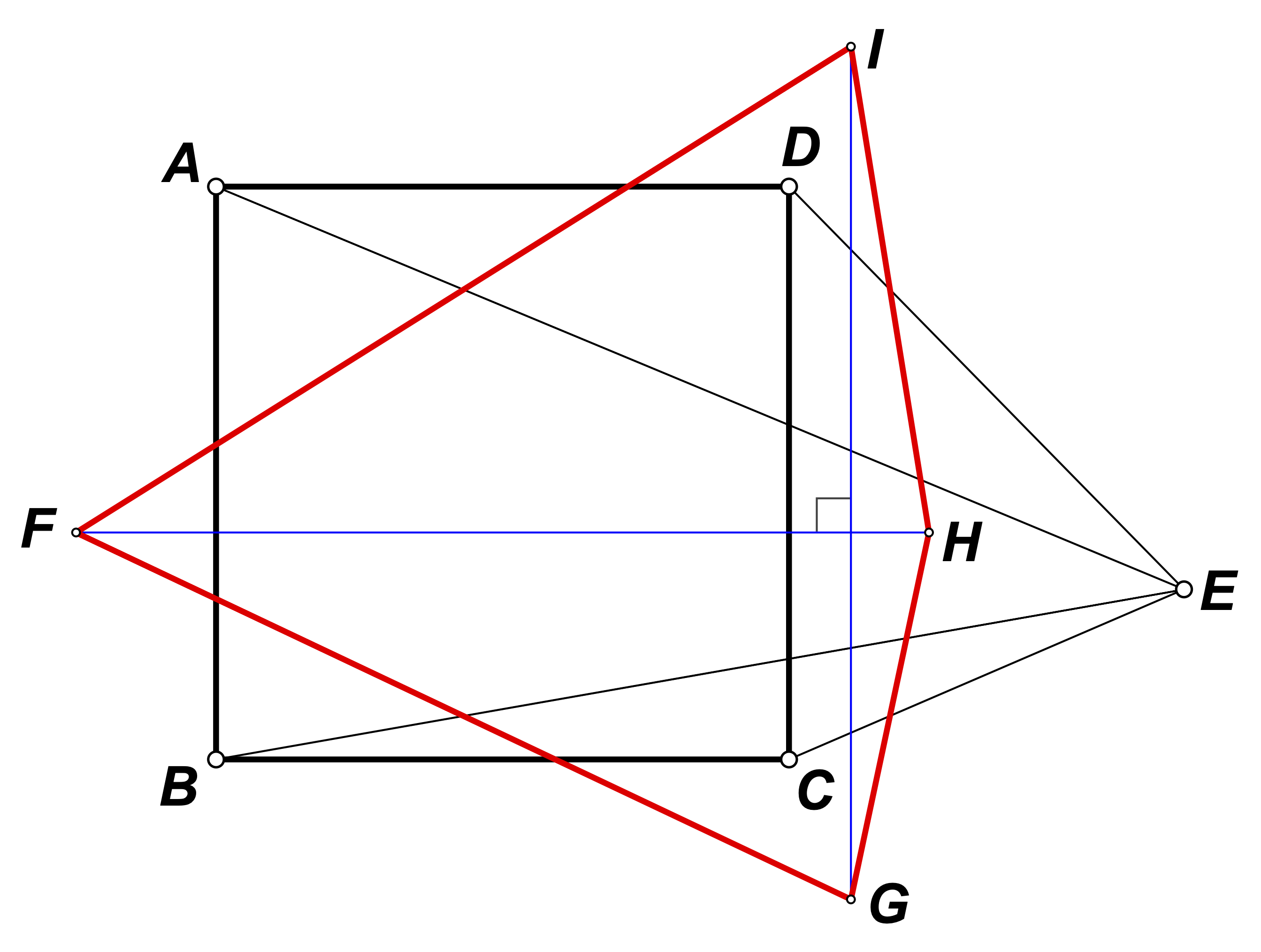}
\caption{square, $X_{372}$-points $\implies$ equi-ortho}
\label{fig:sqArbX372}
\end{figure}


\begin{theorem}
\label{thm:sqArbX373}
Let $E$ be an arbitrary point in the plane of square $ABCD$.
Let $F$, $G$, $H$, and $I$ be the $X_{373}$-points of triangles $\triangle ABE$, $\triangle BCE$, $\triangle CDE$,
and $\triangle DAE$, respectively (Figure~\ref{fig:sqArbX373}).
Then $FGHI$ is an equidiagonal orthodiagonal quadrilateral.
(Note: The diagonals of $FGHI$ are not necessarily parallel to the sides of $ABCD$.)
\end{theorem}

\begin{figure}[h!t]
\centering
\includegraphics[width=0.35\linewidth]{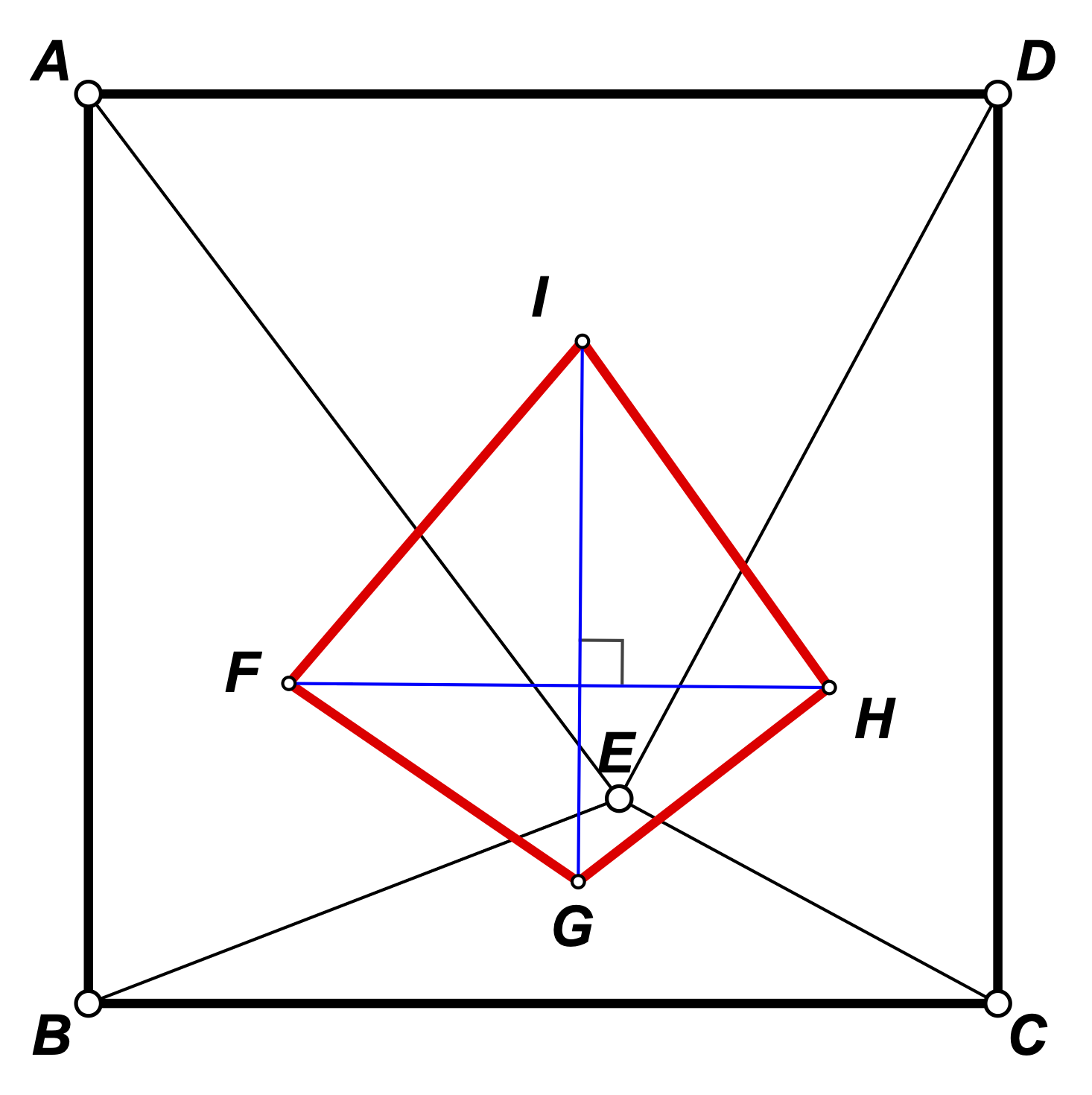}
\caption{square, $X_{373}$-points $\implies$ equi-ortho}
\label{fig:sqArbX373}
\end{figure}


\section{Point $E$ Restricted to a Line}
\label{section:misc}

In the previous section, point $E$ could be any point in the plane.
If point $E$ is restricted to be located on certain lines associated with the quadrilateral,
then some interesting results are obtained. They are shown in the following two tables.

\bigskip
\begin{center}
\begin{tabular}{|l|p{3.3in}|}
\hline
\multicolumn{2}{|c|}{\textbf{\color{blue}\large \strut Central Quadrilaterals of Kites}}\\ \hline
\textbf{Shape of central quad}&\textbf{centers}\\ \hline
\ru isosceles trapezoid&all\\ \hline
\end{tabular}
\end{center}

\newpage

\begin{theorem}
\label{thm:kiCen}
Let $E$ be any point on diagonal $AC$ of kite $ABCD$ (with $AB=AD$).
Let $X$ be any triangle center.
Let $F$, $G$, $H$, and $I$ be the $X$-points of triangles $\triangle ABE$, $\triangle BCE$, $\triangle CDE$,
and $\triangle DAE$, respectively (Figure~\ref{fig:kiCen}).
Then $FGHI$ is an isosceles trapezoid.
\end{theorem}

\begin{figure}[h!t]
\centering
\includegraphics[width=0.35\linewidth]{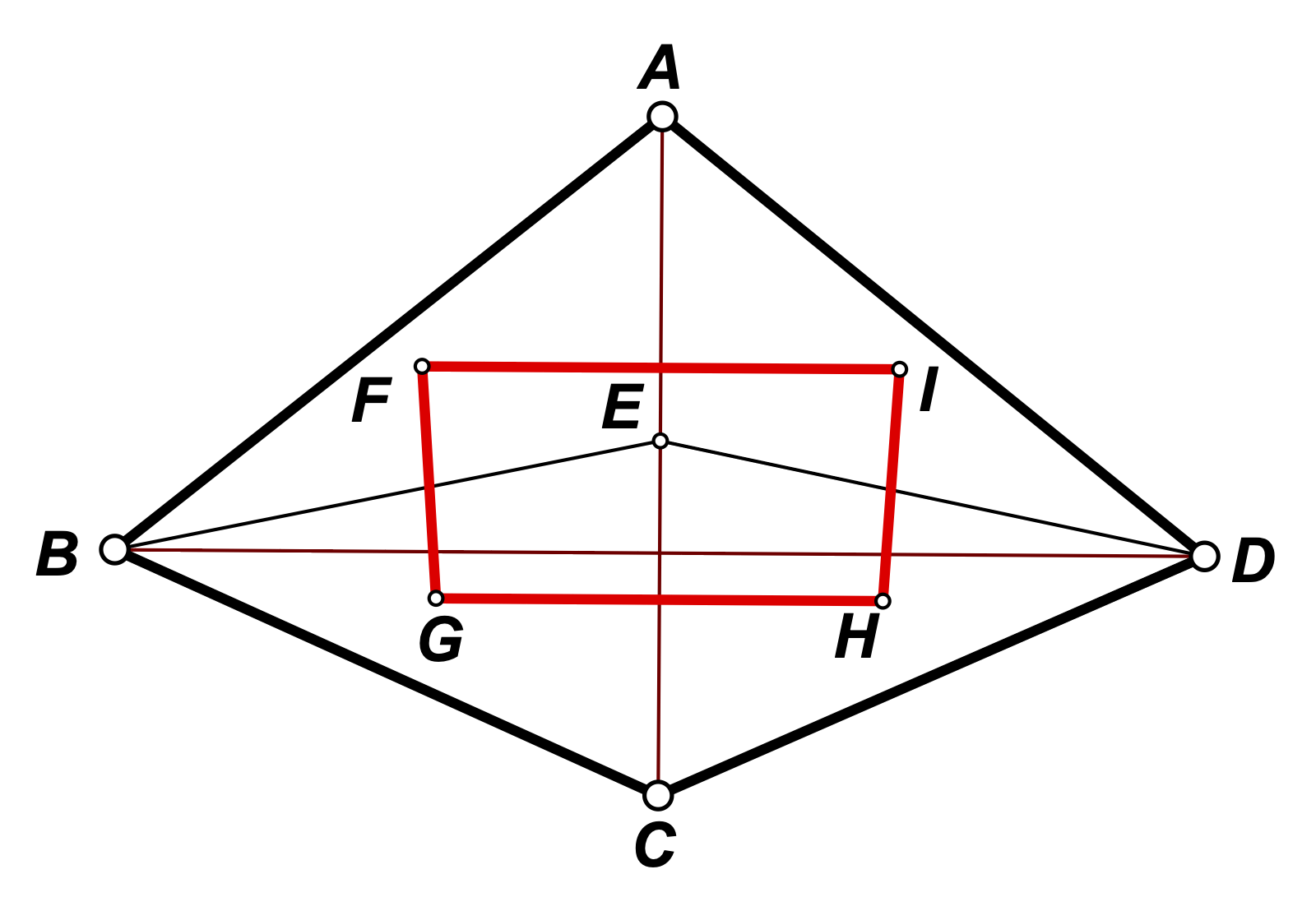}
\caption{kite, $X$-points $\implies$ isosceles trapezoid}
\label{fig:kiCen}
\end{figure}

\begin{proof}
Since $\triangle AED$ is the reflection of $\triangle AEB$ about $AE$, 
then $I$ is the reflection of $F$ about $AE$. 
In particular, $FI\perp AE$. Similarly, $GH\perp AE$. 
Therefore $FI\parallel GH$, so $FGHI$ is a trapezoid.
Furthermore, $IH$ is the reflection of $FG$ about $AC$, so $IH=FG$.
Hence, $FGHI$ is an isosceles trapezoid.
\end{proof}

\smallskip
\begin{center}
\begin{tabular}{|l|p{3.3in}|}
\hline
\multicolumn{2}{|c|}{\textbf{\color{blue}\large \strut Central Quadrilaterals of Isosceles Trapezoids}}\\ \hline
\textbf{Shape of central quad}&\textbf{centers}\\ \hline
\ru kite&all\\ \hline
\end{tabular}
\end{center}

\begin{theorem}
\label{thm:itCen}
Let $E$ be any point on the perpendicular bisector of side $BC$
of isosceles trapezoid $ABCD$ (with $AD\parallel BC$ and $AB=CD$).
Let $X$ be any triangle center.
Let $F$, $G$, $H$, and $I$ be the $X$-points of triangles $\triangle ABE$, $\triangle BCE$, $\triangle CDE$,
and $\triangle DAE$, respectively (Figure~\ref{fig:itCen}).
Then $FGHI$ is a kite.
\end{theorem}

\begin{figure}[h!t]
\centering
\includegraphics[width=0.35\linewidth]{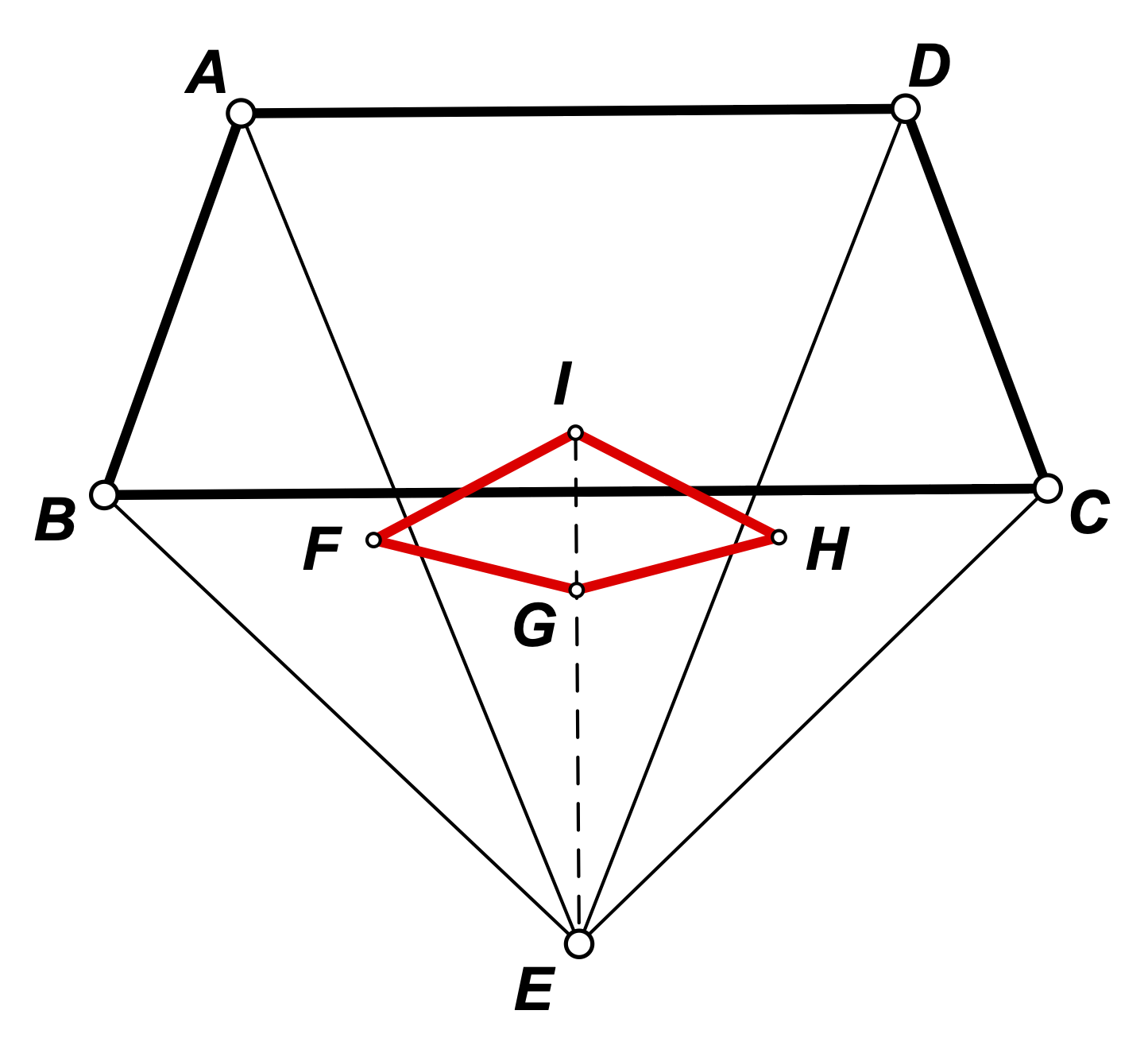}
\caption{isosceles trapezoid, $X$-points $\implies$ kite}
\label{fig:itCen}
\end{figure}

\begin{proof}
Let $m$ be the perpendicular bisector of $BC$.
Since $\triangle DCE$ is the reflection of $\triangle ABE$ about $m$,
then $H$ is the reflection of $F$ about $m$. 
Hence $FH\perp m$. 
Since $\triangle BEC$ is isosceles, $G$ lies on $m$.
Since $\triangle AED$ is isosceles, $I$ lies on $m$.
Therefore, $IG$ coincides with $m$, and hence $IG\perp FH$.
Moreover, $IG$ bisects $FH$ because $H$ is the reflection of $F$ about $m$.
Therefore, $FGHI$ is a kite.
\end{proof}


\newpage

\section{Results Using the Vertex Centroid}
\label{section:vertexCentroid}

A \textit{bimedian} of a quadrilateral is the line segment
joining the midpoints of two opposite sides.

The \textit{centroid} (or vertex centroid) of a quadrilateral is the point of intersection
of the bimedians (Figure~\ref{fig:gpCentroid}). The centroid bisects each bimedian.

\begin{figure}[h!t]
\centering
\includegraphics[width=0.4\linewidth]{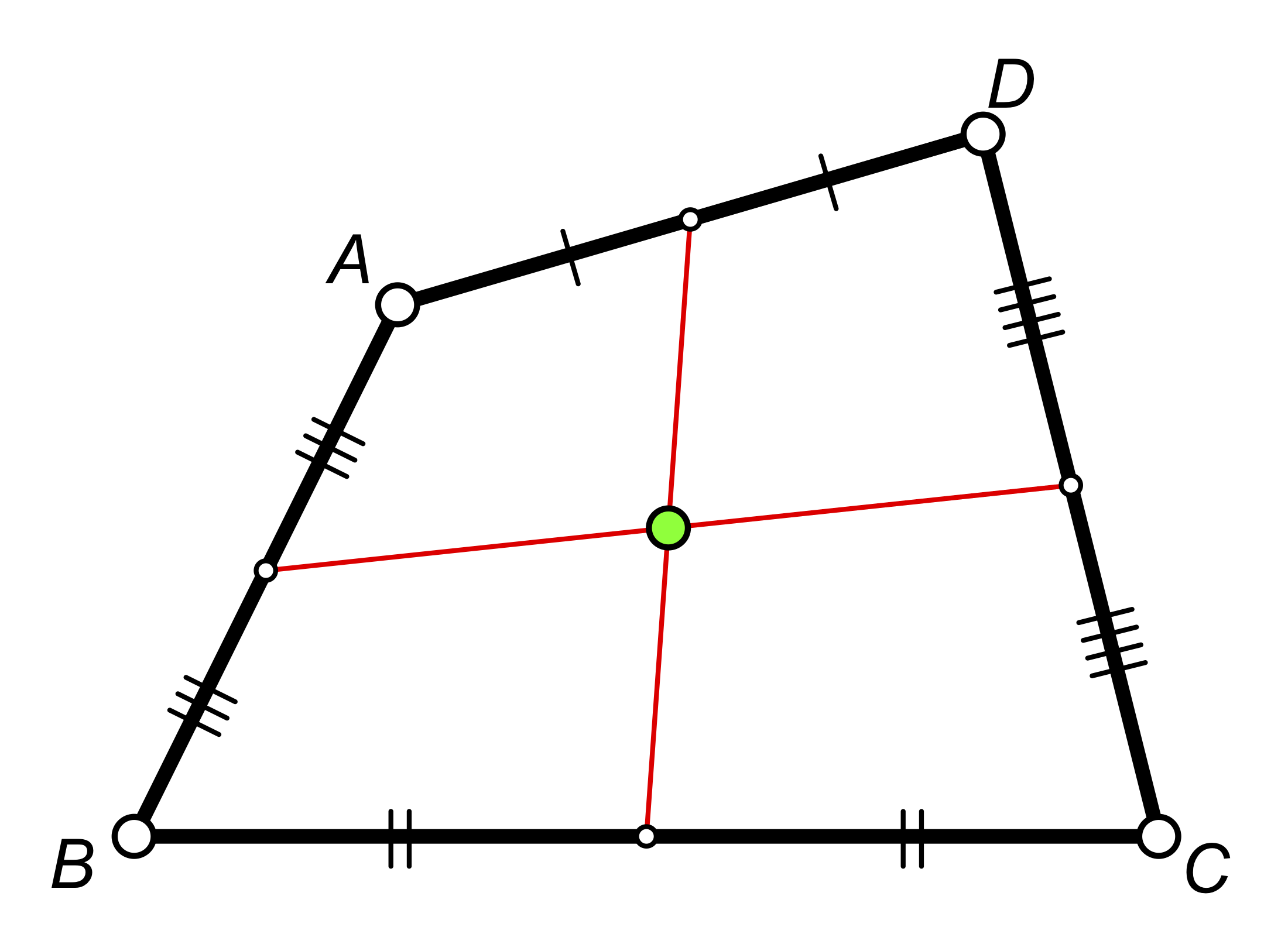}
\caption{Centroid of a quadrilateral}
\label{fig:gpCentroid}
\end{figure}

In this section, we study the case where point $E$ is the centroid of the quadrilateral.
Results that are true when point $E$ is arbitrary are omitted.

Our computer study found a unique shape for the central quadrilateral of an equidiagonal
quadrilateral. The result is shown in the following table.

\bigskip
\begin{center}
\begin{tabular}{|l|p{3.3in}|}
\hline
\multicolumn{2}{|c|}{\textbf{\color{blue}\large \strut Central Quadrilaterals of Equidiagonal Quads}}\\ \hline
\textbf{Shape of central quad}&\textbf{centers}\\ \hline
\ru orthodiagonal&591\\ \hline
\end{tabular}
\end{center}

\begin{theorem}
\label{thm:equiCenX591}
Let $E$ be the vertex centroid of equidiagonal quadrilateral $ABCD$.
Let $F$, $G$, $H$, and $I$ be the $X_{591}$-points of triangles $\triangle ABE$, $\triangle BCE$, $\triangle CDE$,
and $\triangle DAE$, respectively (Figure~\ref{fig:equiCenX591}).
Then $FGHI$ is an orthodiagonal quadrilateral.
\end{theorem}

\begin{figure}[h!t]
\centering
\includegraphics[width=0.5\linewidth]{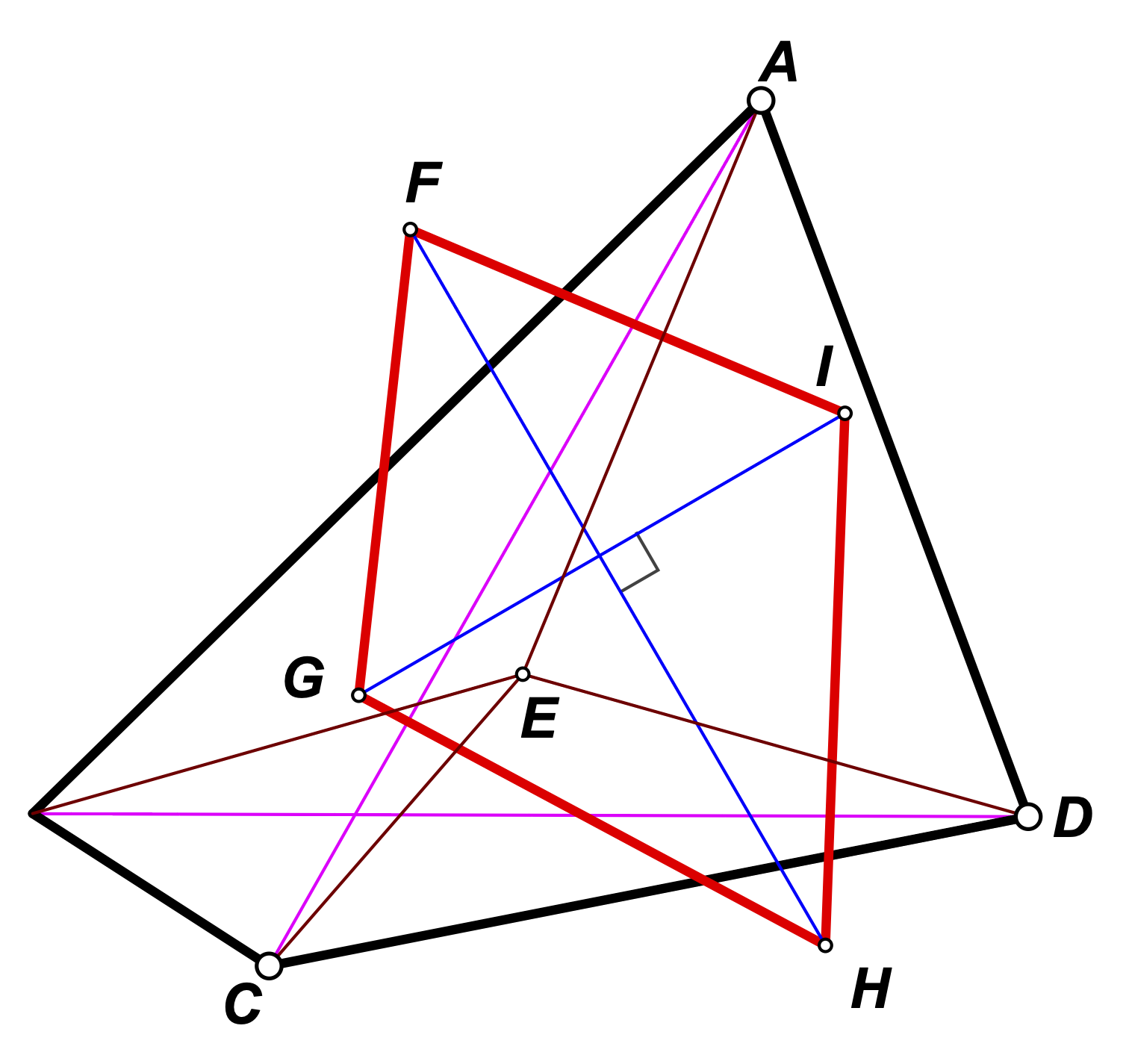}
\caption{equidiagonal quadrilateral, $X_{591}$-points $\implies$ orthodiagonal}
\label{fig:equiCenX591}
\end{figure}

\newpage

Our computer study found additional results if quadrilateral $ABCD$ is also
orthodiagonal. The results are shown in the following table.

\bigskip
\begin{center}
\begin{tabular}{|l|p{3.3in}|}
\hline
\multicolumn{2}{|c|}{\textbf{\color{blue}\large \strut Central Quads of Equidiagonal Orthodiagonal Quads}}\\ \hline
\textbf{Shape of central quad}&\textbf{centers}\\ \hline
\ru parallelogram&491, 615\\ \hline
\end{tabular}
\end{center}

\begin{theorem}
\label{thm:eoCenX491}
Let $E$ be the vertex centroid of equidiagonal orthodiagonal quadrilateral $ABCD$.
Let $n$ be 491 or 615.
Let $F$, $G$, $H$, and $I$ be the $X_{n}$-points of triangles $\triangle ABE$, $\triangle BCE$, $\triangle CDE$,
and $\triangle DAE$, respectively.
Then $FGHI$ is a parallelogram.
(Figure~\ref{fig:eoCenX491} shows the case when $n=491$.)
\end{theorem}

\begin{figure}[h!t]
\centering
\includegraphics[width=0.35\linewidth]{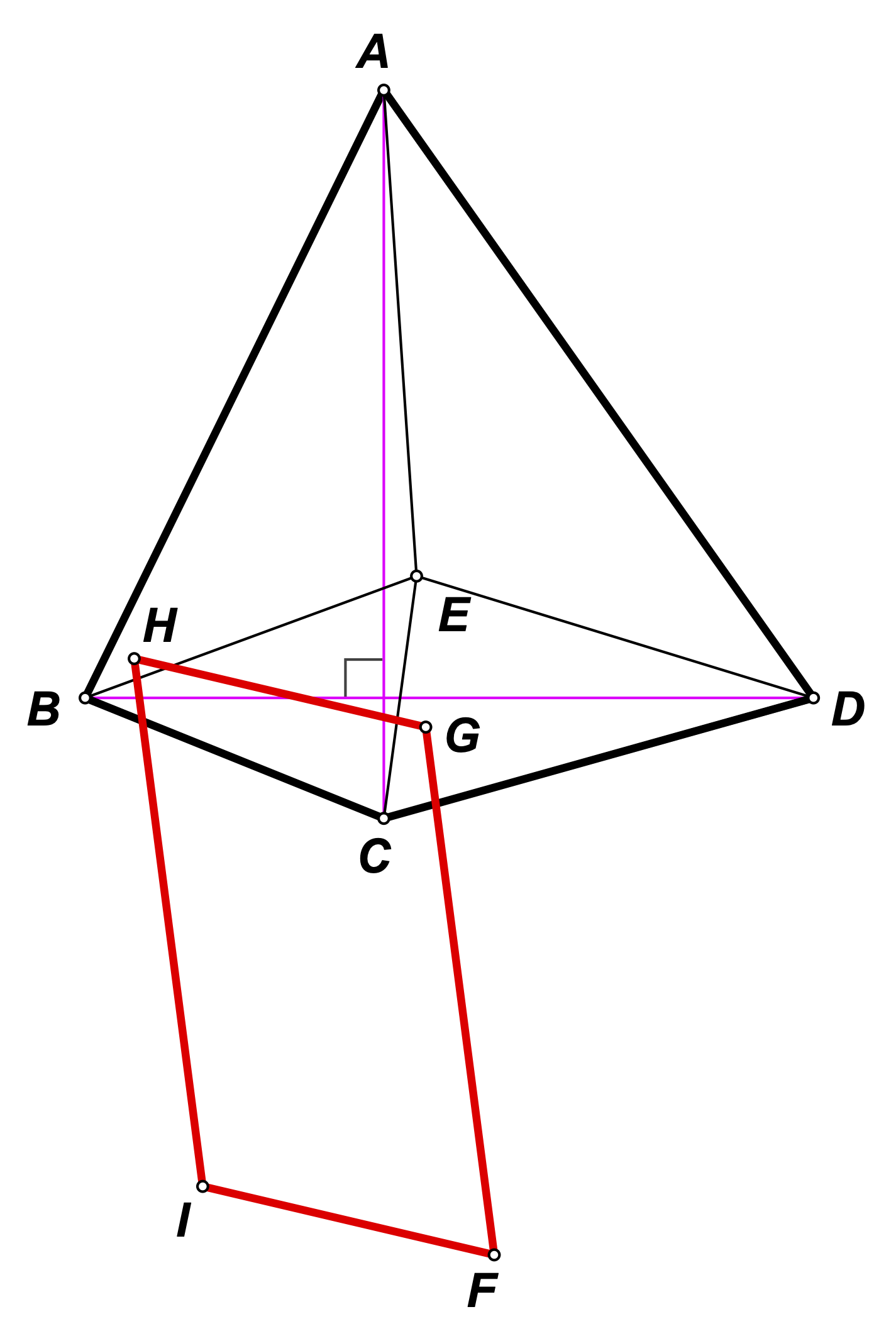}
\caption{equidiagonal orthodiagonal quad, $X_{491}$-points $\implies$ parallelogram}
\label{fig:eoCenX491}
\end{figure}


\newpage

\section{Results Using the Steiner Point}
\label{section:SteinerPoint}

A \emph{midray circle} of a quadrilateral
is the circle through the midpoints of the line segments joining one vertex
of the quadrilateral to the other vertices.

\void{
\begin{figure}[h!t]
\centering
\includegraphics[width=0.35\linewidth]{spMidrayCircle.png}
\caption{Midray circle of quadrilateral $ABCD$ relative to vertex $B$}
\label{fig:spMidrayCircle}
\end{figure}
}

The \emph{Steiner point} (sometimes called the Gergonne-Steiner point) of a quadrilateral
is the common point of the midray circles of the quadrilateral.

\begin{figure}[h!t]
\centering
\includegraphics[width=0.4\linewidth]{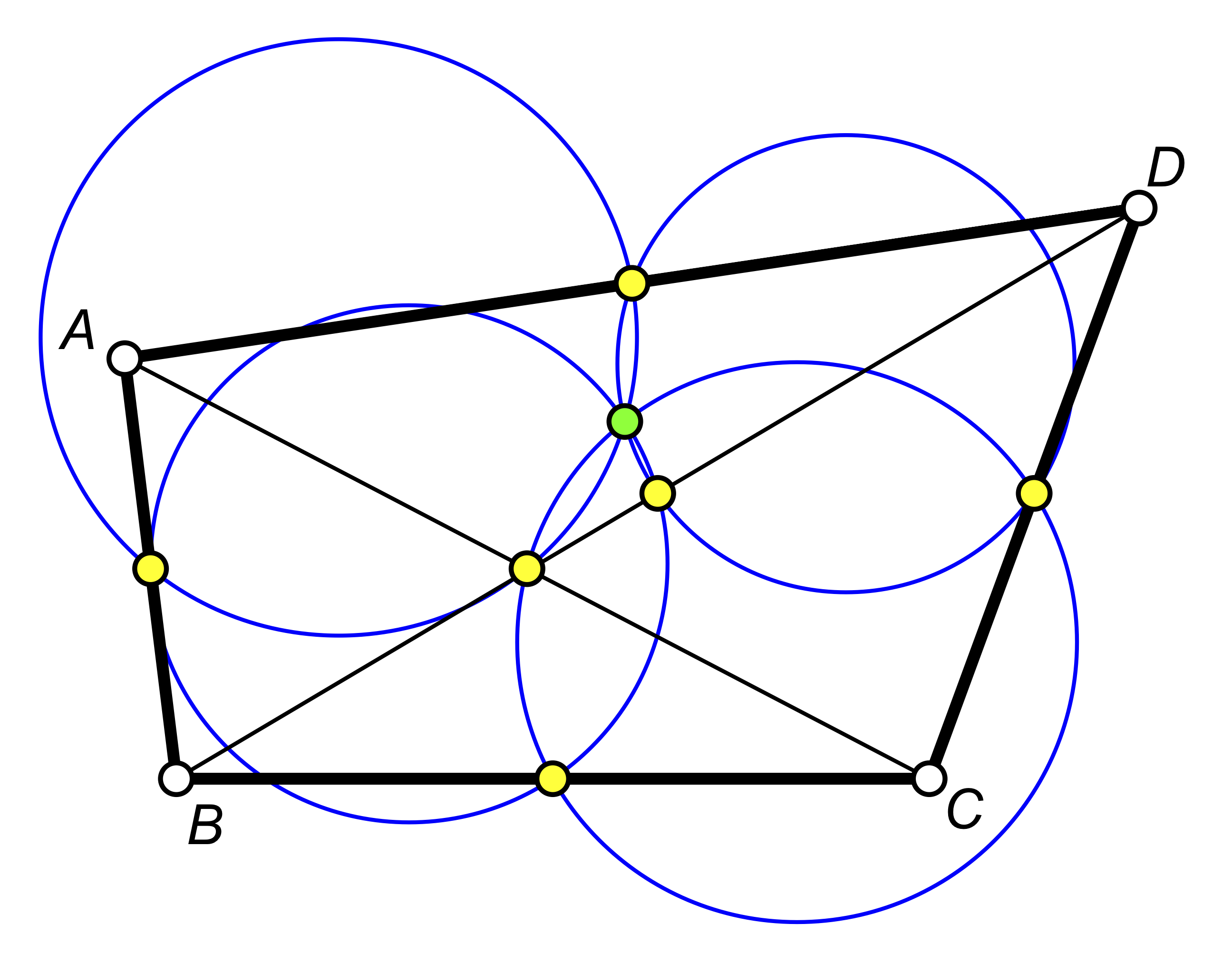}
\caption{Steiner point of quadrilateral $ABCD$}
\label{fig:spSteinerPoint}
\end{figure}

Figure \ref{fig:spSteinerPoint} shows the Steiner point of quadrilateral $ABCD$.
The yellow points represent the midpoints of the sides and diagonals of the quadrilateral.
The blue circles are the midray circles.
The common point of the four circles is the Steiner point (shown in green).

In this section, we study the case where point $E$ is the Steiner point of the quadrilateral.
Results that are true when point $E$ is arbitrary are omitted.
Before stating the results, we give some lemmas and definitions.

\begin{proposition}
\label{prop:circumcenter}
The Steiner point of a cyclic quadrilateral is the circumcenter of the quadrilateral.
\end{proposition}

\begin{proof}
This is Proposition 10.2 of \cite{relationships}.
\end{proof}

\void{
The symbol $\mathbb{C}$ denotes the set of all triangle centers that lie on the circumcircle of the reference triangle.
}

The following lemma comes from \cite{MathWorld-Circumcircle}.

\begin{lemma}
\label{lemma:circumcircle}
The center $X_n$ of $\triangle ABC$ lies on the
circumcircle of $\triangle ABC$ for the following values of $n$:
\end{lemma}
74, 98--112, 476, 477, 675, 681, 689, 691, 697, 699, 701, 703, 705, 707, 709, 711, 713, 715, 717, 719, 721, 723, 725, 727, 729, 731, 733, 735, 737, 739, 741, 743, 745, 747, 753, 755, 759, 761, 767, 769, 773, 777, 779, 781, 783, 785, 787, 789, 791, 793, 795, 797, 803, 805, 807, 809, 813, 815, 817, 819, 825, 827, 831, 833, 835, 839--843, 898, 901, 907, 915, 917, 919, 925, 927, 929--935, 953, 972.

The following two lemmas comes from \cite{relationships}.

\begin{lemma}
\label{lemma:dpIsoscelesTriangleA}
Let $ABC$ be an isosceles triangle with $AB=AC$.
Then the center $X_n$ coincides with $A$ for the following values of $n$:
\end{lemma}
59, 99, 100, 101, 107, 108, 109, 110, 112, 162, 163, 190, 249, 250, \
476, 643, 644, 645, 646, 648, 651, 653, 655, 658, 660, 662, 664, 666, \
668, 670, 677, 681, 685, 687, 689, 691, 692, 765, 769, 771, 773, 777, \
779, 781, 783, 785, 787, 789, 791, 793, 795, 797, 799, 803, 805, 807, \
809, 811, 813, 815, 817, 819, 823, 825, 827, 831, 833, 835, 839, 874, \
877, 880, 883, 886, 889, 892, 898, 901, 906, 907, 919, 925, 927, 929, \
930, 931, 932, 933, 934, 935.

\begin{lemma}
\label{lemma:dpIsoscelesTriangleMidpoint}
Let $ABC$ be an isosceles triangle with $AB=AC$.
Let $M$ be the midpoint of $BC$.
Then the center $X_n$ coincides with $M$ for the following values of $n$:
\end{lemma}
11, 115, 116, 122--125, 127, 130, 134--137, 139, \
244--247, 338, 339, 865-868.

The symbol $\mathbb{M}$ denotes these points.

\begin{lemma}
\label{lemma:dpIsoscelesTriangleAntipode}
Let $ABC$ be an isosceles triangle with $AB=AC$.
Let $P$ be the antipode of point $A$ with respect to the circumcircle of $\triangle ABC$.
Then the center $X_n$ coincides with $P$ for the following values of $n$:
\end{lemma}
74, 98, 102--106, 111, 477, 675, 697, 699, 
701, 703, 705, 707, 709, 711, 713, 715, 717, 719, 721, 723, 725, 727, 729, 731, 733, 735, 737, 739, 
741, 743, 745, 747, 753, 755, 759, 761, 767, 840--843, 915, 917, 953, 972.

\begin{proof}
This follows from Lemmas \ref{lemma:circumcircle} and \ref{lemma:dpIsoscelesTriangleA}.
\end{proof}

The symbol $\mathbb{T}$ denotes these points.

Our computer study found several results for the central quadrilateral of cyclic quadrilaterals.
They are shown in the following table.

\bigskip
\begin{center}
\begin{tabular}{|l|p{3.3in}|}
\hline
\multicolumn{2}{|c|}{\textbf{\color{blue}\large \strut Central Quadrilaterals of Cyclic Quadrilaterals}}\\ \hline
\textbf{Shape of central quad}&\textbf{centers}\\ \hline
\ru parallelogram&$\mathbb{M}$, 148--150, 290, 402, 620, 671, 903\\ \hline
\ru tangential&$\mathbb{T}$, 3, 399\\ \hline
\end{tabular}
\end{center}


\begin{theorem}
\label{thm:cqSteinX3}
Let $E$ be the Steiner point of cyclic quadrilateral $ABCD$.
Let $n$ be 3 or 399.
Let $F$, $G$, $H$, and $I$ be the $X_n$-points of triangles $\triangle ABE$, $\triangle BCE$, $\triangle CDE$,
and $\triangle DAE$, respectively. (Figure~\ref{fig:cqSteinX3} shows the case when $n=3$.)
Then $FGHI$ is tangential with incenter $E$.
The incircle of $FGHI$ is concentric with the circumcircle of $ABCD$ and
the inradius of $FGHI$ is half the circumradius of $ABCD$.
\end{theorem}

\begin{figure}[h!t]
\centering
\includegraphics[width=0.3\linewidth]{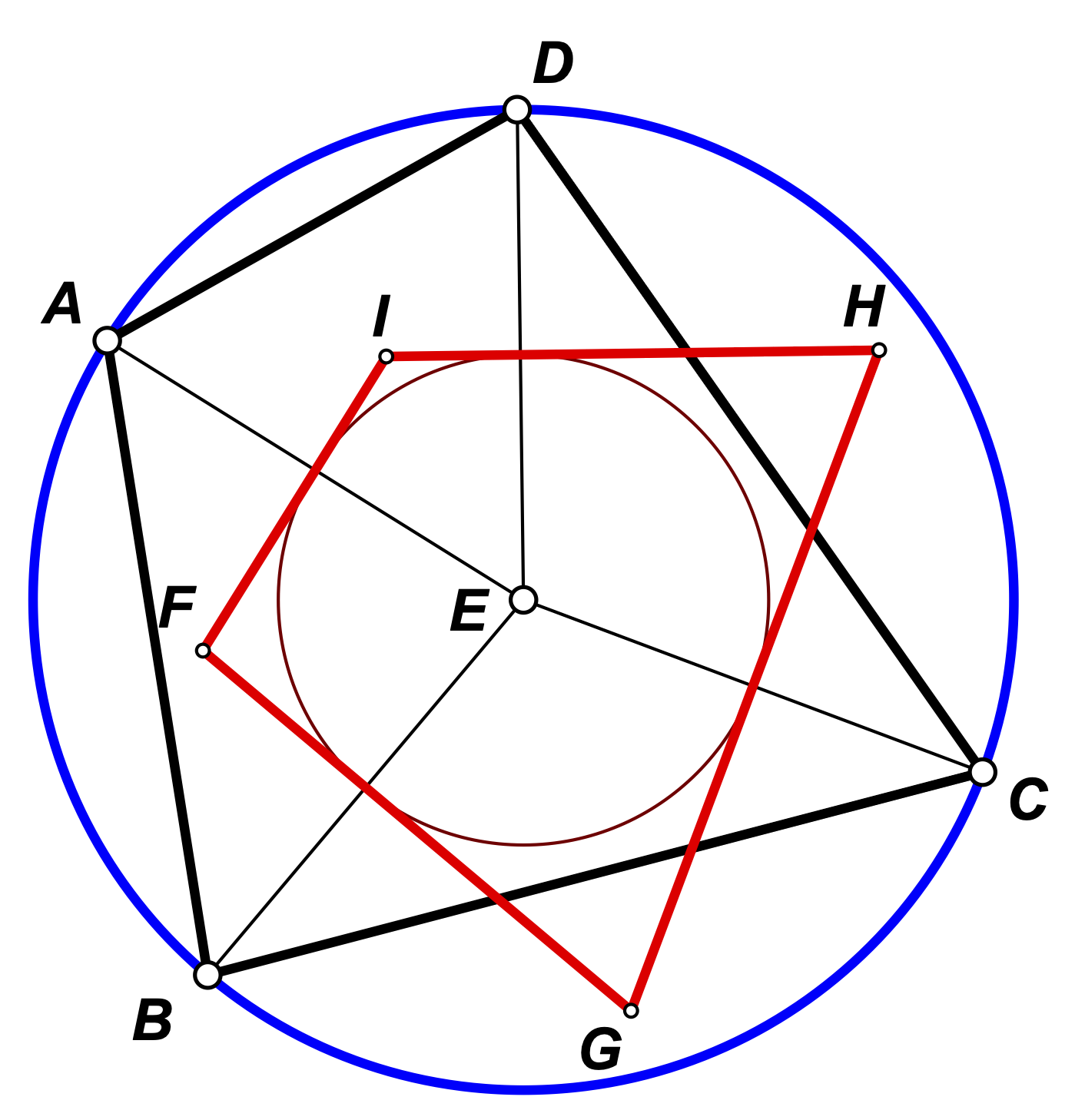}
\caption{cyclic, $X_3$-points $\implies$ tangential}
\label{fig:cqSteinX3}
\end{figure}

\newpage
For the case $n=3$, we can give a purely geometrical proof.

\begin{proof}
By Proposition~\ref{prop:circumcenter}, $E$ is the circumcenter of quadrilateral $ABCD$.
The points $F$, $G$, $H$, $I$ are the circumcenters of triangles $\triangle ABE$, $\triangle BCE$, $\triangle CDE$,
and $\triangle DAE$, respectively.
Line $AE$ is the radical axis of of circles $(AEB)$ and $(AED)$.
Let $X=AE\cap FI$.
Clearly we have $AE\perp FI$ and $EX=AE/2$.
Define $Y$, $Z$, and $W$ similarly (Figure~\ref{fig:cqSteinX3proof}).
\begin{figure}[h!t]
\centering
\includegraphics[width=0.4\linewidth]{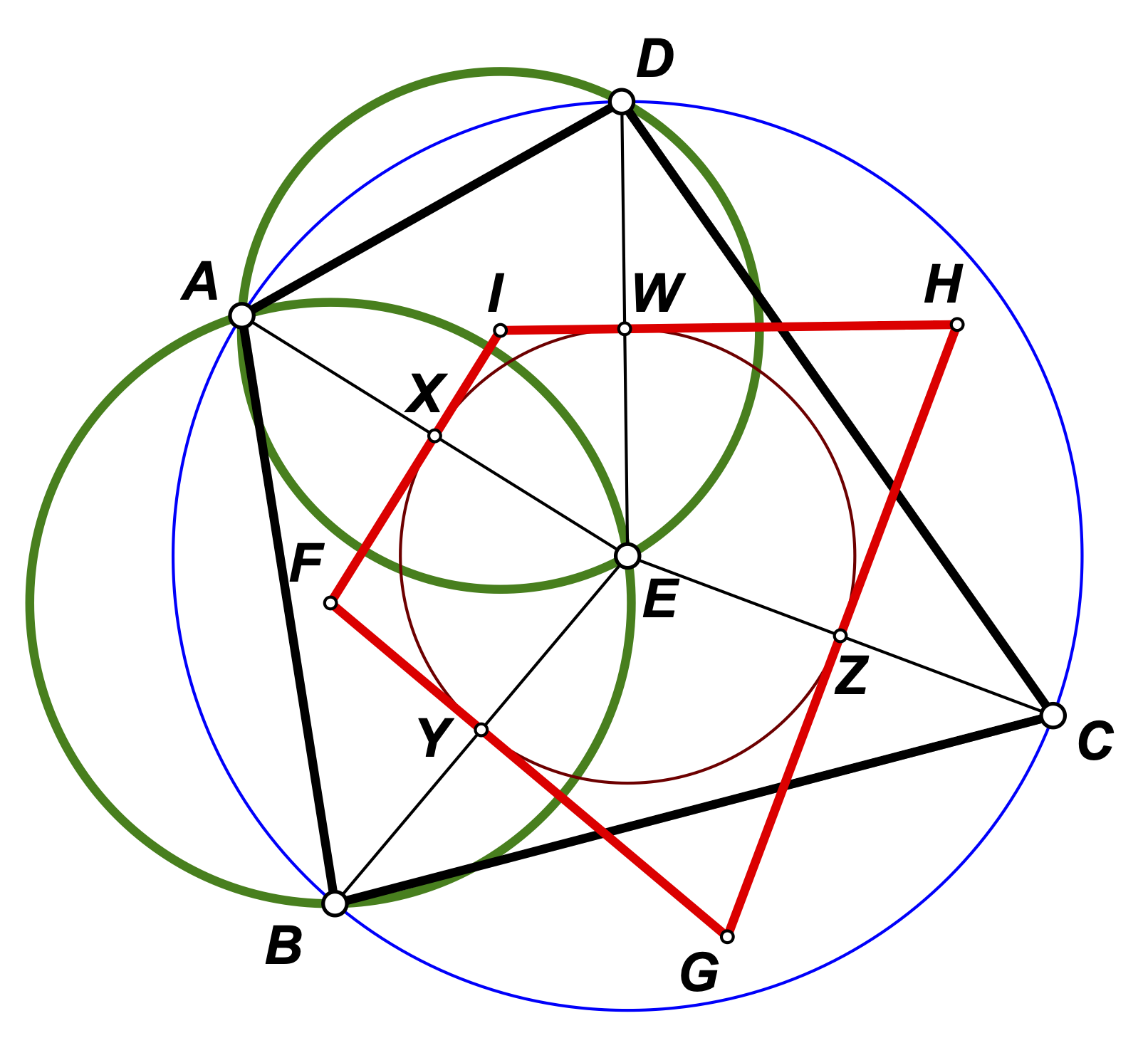}
\caption{case when $n=3$}
\label{fig:cqSteinX3proof}
\end{figure}
Similarly we have $BE\perp FG$, $CE\perp GH$, and $DE\perp HI$,
and $EY=EB/2$, and $EZ=EC/2$,  and $EW=ED/2$.
Since $EA=EB=EC=ED$, we have $EX=EY=EZ=EW$, so $E$ is the incenter
of quadrilateral $FGHI$.
Therefore $FGHI$ is a tangential quadrilateral.
\end{proof}

\begin{theorem}
\label{thm:cqSteinT}
Let $E$ be the Steiner point of cyclic quadrilateral $ABCD$.
Let $X_n\in\mathbb{T}$.
Let $F$, $G$, $H$, and $I$ be the $X_n$-points of triangles $\triangle ABE$, $\triangle BCE$, $\triangle CDE$,
and $\triangle DAE$, respectively (Figure~\ref{fig:cqSteinX74}).
Then $FGHI$ is a tangential quadrilateral with incenter $E$.
The incircle of $FGHI$ coincides with the circumcircle of $ABCD$.
\end{theorem}

\begin{figure}[h!t]
\centering
\includegraphics[width=0.5\linewidth]{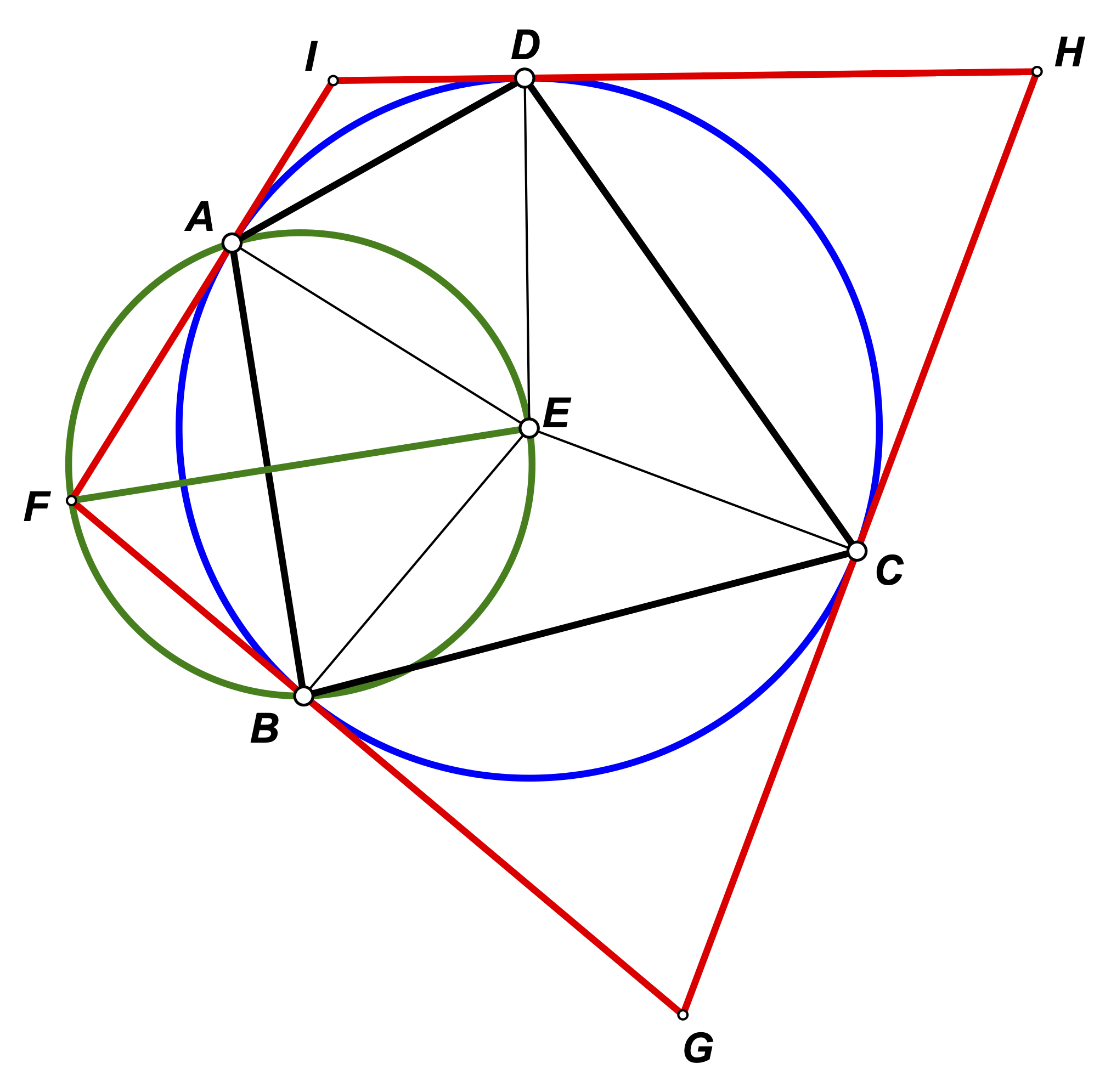}
\caption{cyclic, $X_n\in\mathbb{T}$ $\implies$ $FGHI$ tangential}
\label{fig:cqSteinX74}
\end{figure}

\begin{proof}
By Proposition~\ref{prop:circumcenter}, $E$ is the circumcenter of quadrilateral $ABCD$.
Since $EA=EB$, $\triangle EAB$ is isosceles with vertex $E$, so $F$ is the antipode of $E$
with respect to the circumcircle of $\triangle EAB$ by Lemma~\ref{lemma:dpIsoscelesTriangleAntipode}.
Therefore, $\angle FAE=90\degrees$.
Similarly, $\angle IAE=90\degrees$.
Hence $I$, $A$, and $F$ are collinear and $EA\perp FI$.
Similarly, $EB\perp FG$, $EC\perp GH$, and $ED\perp HI$.
Since $EA=EB=EC=ED$, it follows that $FGHI$ is a tangential quadrilateral
and the incenter of $FGHI$ coincides with $E$.
\end{proof}

\begin{theorem}
\label{thm:cqStein}
Let $E$ be the Steiner point of cyclic quadrilateral $ABCD$.
Let $n$ be 11, 115, 116, 122-125, 127, 130, 134--137, 139, 148--150, 244--247, 290, 338, 339, 402, 620, 671, 865--868, or 903.
Let $F$, $G$, $H$, and $I$ be the $X_n$-points of triangles $\triangle ABE$, $\triangle BCE$, $\triangle CDE$,
and $\triangle DAE$, respectively.
Then $FGHI$ is a parallelogram.
\end{theorem}

\begin{proof}
If $n$ is 11, 115, 116, 122--125, 127, 130, 134--137, 139, 244--247, 338, 339, 865--868,
the result is true by the proof of Theorem 9.3 of \cite{relationships}.

If $n$ is 402 or 620, the result is true by the proof of Theorem 9.2 of \cite{relationships}.

If $n$ is 290, 671, or 903, the result is true by the proof of Theorem 9.8 of \cite{relationships}.

If $n$ is 148, 149, or 150, the result is true by the proof of Theorem 9.9 of \cite{relationships}.

The proofs in \cite{relationships} show that the central quadrilateral is homothetic
to a parallelogram associated with the reference quadrilateral.

This covers all the cases.
\end{proof}


The following proposition is useful
when giving geometric proofs of results involving the Steiner point of an
orthodiagonal quadrilateral.

\begin{proposition}
\label{prop:spOrtho}
The Steiner point of an orthodiagonal quadrilateral coincides with the point
of intersection of the perpendicular bisectors of the diagonals.
\end{proposition}

\begin{figure}[h!t]
\centering
\includegraphics[width=0.5\linewidth]{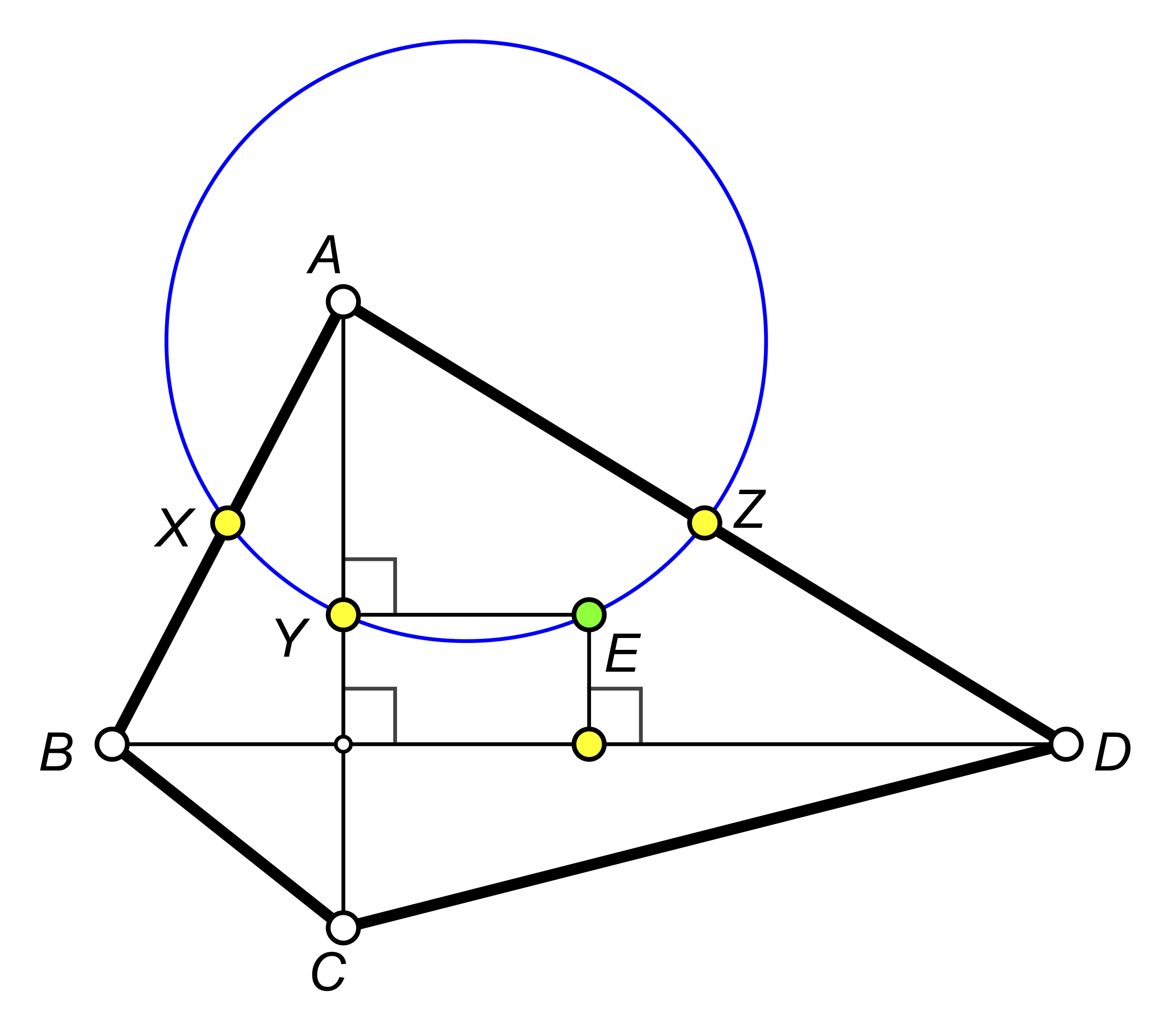}
\caption{Steiner point of an orthodiagonal quadrilateral}
\label{fig:spOrtho}
\end{figure}

\begin{proof}
This is Proposition 10.4 of \cite{relationships}.
\end{proof}

\newpage

We found some results for an orthodiagonal quadrilateral.
They are shown in the following table.

\bigskip
\begin{center}
\begin{tabular}{|l|p{3.3in}|}
\hline
\multicolumn{2}{|c|}{\textbf{\color{blue}\large \strut Central Quads of an Orthodiagonal  Quadrilateral}}\\ \hline
\textbf{Shape of central quad}&\textbf{centers}\\ \hline
\ru cyclic&3\\ \hline
\ru trapezoid&4\\ \hline
\end{tabular}
\end{center}


\begin{theorem}
\label{thm:odSteinX3}
Let $E$ be the Steiner point of orthodiagonal quadrilateral $ABCD$.
Let $F$, $G$, $H$, and $I$ be the circumcenters of triangles $\triangle ABE$, $\triangle BCE$, $\triangle CDE$,
and $\triangle DAE$, respectively (Figure~\ref{fig:odSteinX3}).
Then $FGHI$ is a cyclic quadrilateral.
\end{theorem}

\begin{figure}[h!t]
\centering
\includegraphics[width=0.3\linewidth]{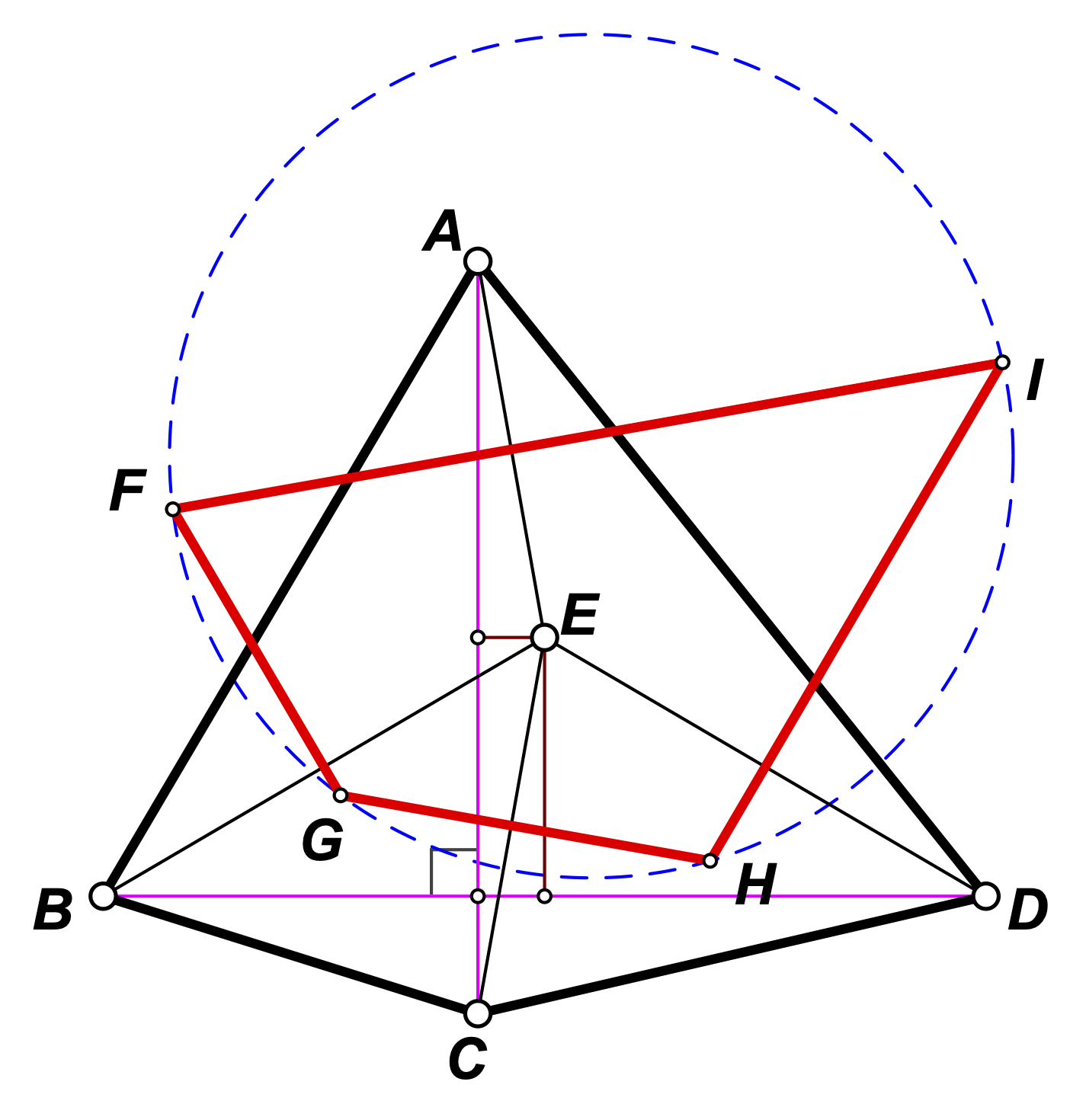}
\caption{orthodiagonal quadrilateral, $X_3$-points $\implies$ cyclic}
\label{fig:odSteinX3}
\end{figure}

\begin{open}
Is there a purely geometric proof of Theorem~\ref{thm:odSteinX3}?
\end{open}

\begin{open}
If the central quadrilateral is cyclic, must the center be the circumcenter?
\end{open}


\begin{theorem}
\label{thm:odSteinX4}
Let $E$ be the Steiner point of orthodiagonal quadrilateral $ABCD$.
Let $F$, $G$, $H$, and $I$ be the orthocenters of triangles $\triangle ABE$, $\triangle BCE$, $\triangle CDE$,
and $\triangle DAE$, respectively (Figure~\ref{fig:odSteinX3}).
Then $FHIG$ is a trapezoid with $FH\parallel GI$.
\end{theorem}

\begin{figure}[h!t]
\centering
\includegraphics[width=0.5\linewidth]{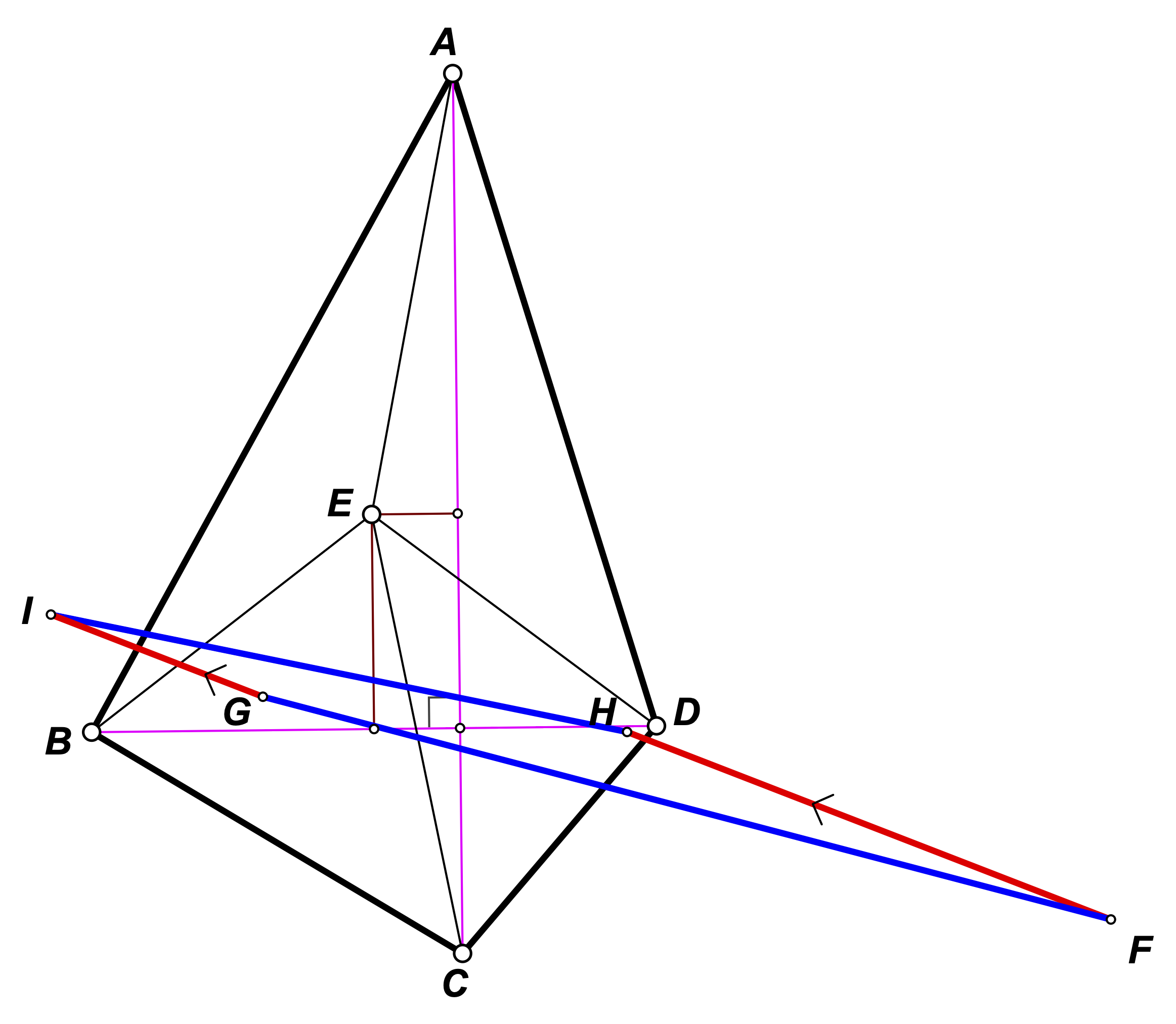}
\caption{orthodiagonal quadrilateral, $X_4$-points $\implies$ $FH\parallel GI$}
\label{fig:odSteinX4}
\end{figure}

\begin{open}
Is there a purely geometric proof of Theorem~\ref{thm:odSteinX4}?
\end{open}

\begin{open}
If the central quadrilateral is a trapezoid, must the center be the orthocenter?
\end{open}

Our computer study found some additional results for an equidiagonal orthodiagonal quadrilateral.
They are shown in the following table.

\bigskip
\begin{center}
\begin{tabular}{|l|p{3.3in}|}
\hline
\multicolumn{2}{|c|}{\textbf{\color{blue}\large \strut Central Quads of Equidiagonal Orthodiagonal  Quads}}\\ \hline
\textbf{Shape of central quad}&\textbf{centers}\\ \hline
\ru orthodiagonal&486, 487, 642 \\ \hline
\end{tabular}
\end{center}

\begin{theorem}
\label{thm:eoSteinX642}
Let $E$ be the Steiner point of equidiagonal orthodiagonal quadrilateral $ABCD$.
Let $n$ be 486, 487, or 642.
Let $F$, $G$, $H$, and $I$ be the $X_n$-points of triangles $\triangle ABE$, $\triangle BCE$, $\triangle CDE$,
and $\triangle DAE$, respectively.
Figure~\ref{fig:eoSteinX486} shows an example using the inner Vecten points ($n=486$).
Figure~\ref{fig:eoSteinX642} shows an example when $n=642$.
Then $FGHI$ is an orthodiagonal quadrilateral.
When $n=486$, the diagonals of $FGHI$ meet at point~$E$.
\end{theorem}

\begin{figure}[h!t]
\centering
\includegraphics[width=0.35\linewidth]{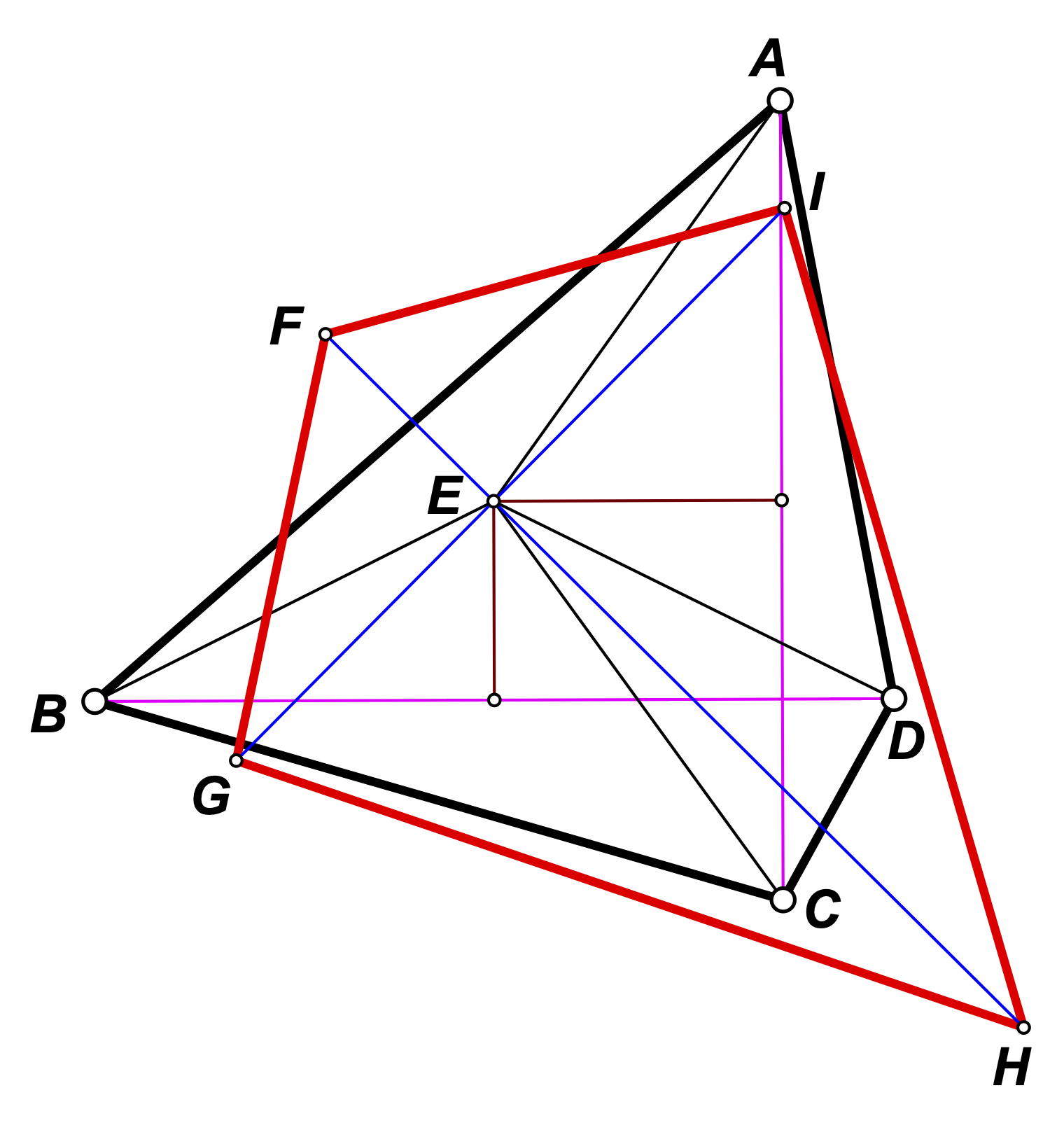}
\caption{equi-ortho-quad, $X_{486}$-points $\implies$ orthodiagonal}
\label{fig:eoSteinX486}
\end{figure}

\begin{figure}[h!t]
\centering
\includegraphics[width=0.35\linewidth]{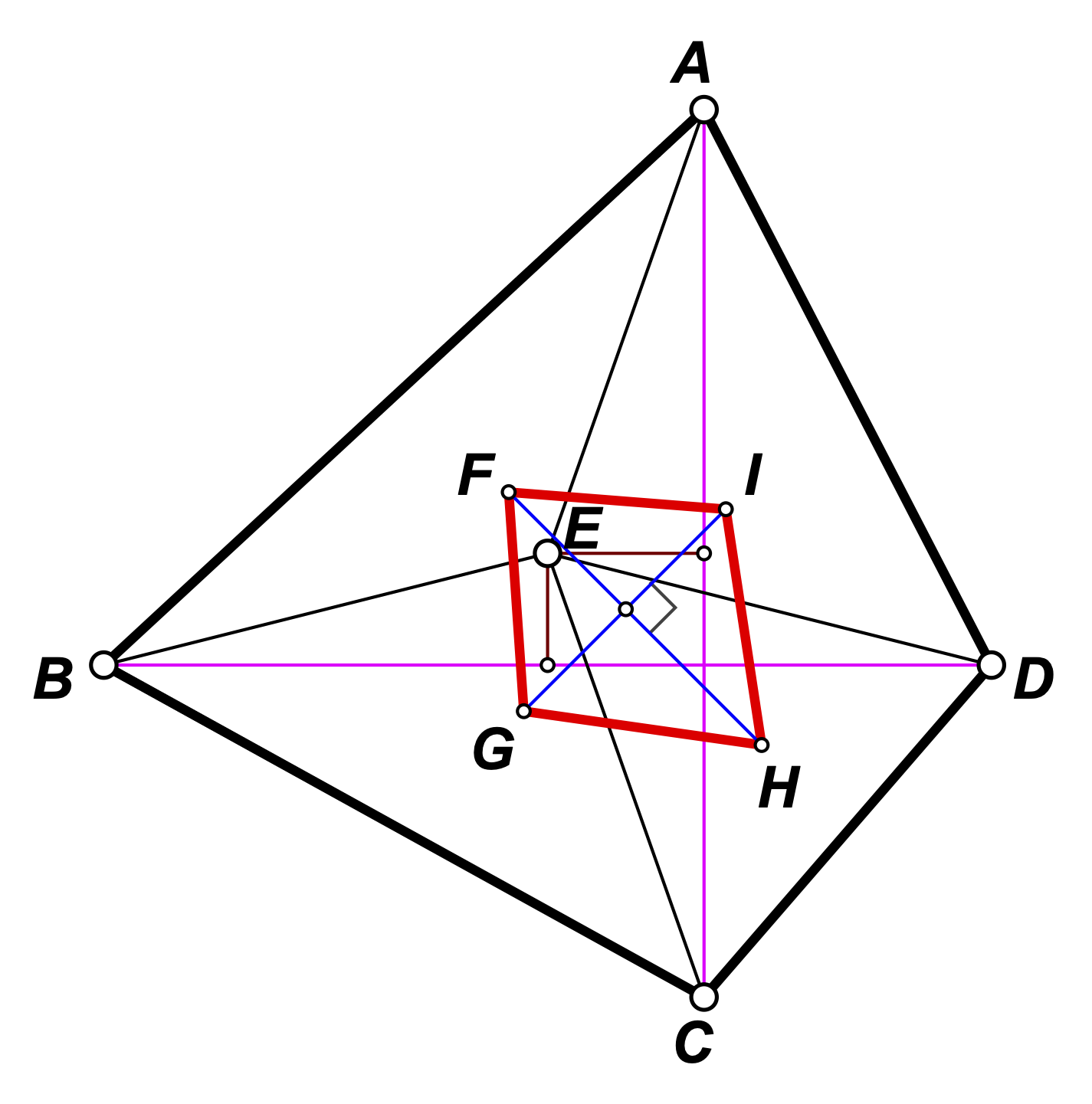}
\caption{equi-ortho-quad, $X_{642}$-points $\implies$ orthodiagonal}
\label{fig:eoSteinX642}
\end{figure}

\newpage
Our computer study found some additional results for bicentric quadrilaterals.
They are shown in the following table.

\bigskip
\begin{center}
\begin{tabular}{|l|p{3.3in}|}
\hline
\multicolumn{2}{|c|}{\textbf{\color{blue}\large \strut Central Quads of Bicentric  Quadrilaterals}}\\ \hline
\textbf{Shape of central quad}&\textbf{centers}\\ \hline
\ru cyclic&1, 165, 214\\ \hline
\end{tabular}
\end{center}

\begin{theorem}
\label{thm:biSteinX1}
Let $E$ be the Steiner point of bicentric quadrilateral $ABCD$.
Let $F$, $G$, $H$, and $I$ be the incenters of triangles $\triangle ABE$, $\triangle BCE$, $\triangle CDE$,
and $\triangle DAE$, respectively (Figure~\ref{fig:biSteinX1}).
Then $FGHI$ is a cyclic quadrilateral.
\end{theorem}

\begin{figure}[h!t]
\centering
\includegraphics[width=0.5\linewidth]{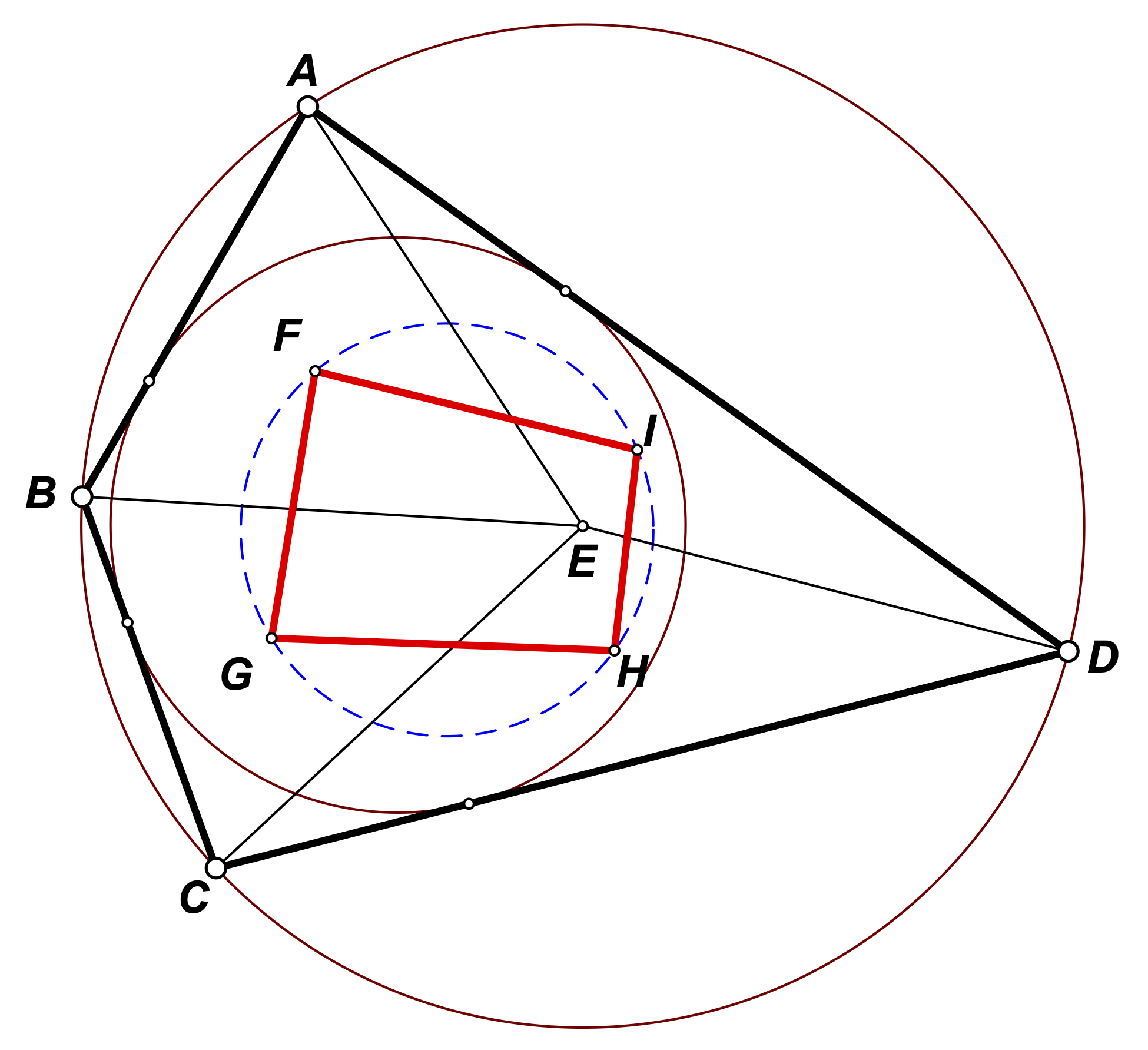}
\caption{bicentric quadrilateral, $X_1$-points $\implies$ cyclic}
\label{fig:biSteinX1}
\end{figure}

For more information about this result, see \cite{knot1}, \cite{knot2}, and \cite{Oai}.
A purely geometric proof can be found in \cite{Zaslavsky}.


\begin{lemma}
\label{lemma:X214}
The $X_{214}$-point of a triangle is the midpoint of $X_1$ and $X_{100}$.
\end{lemma}

\begin{proof}
See \cite{ETC214}.
\end{proof}

\void{
\begin{lemma}
\label{lemma:X100}
The $X_{100}$-point of an isosceles triangle coincides with the vertex,
\end{lemma}

Possible geometric proof:

\begin{proof}
From \cite{ETC100}, we know that $X_{100}$ lies on the circumcircle.
All triangle centers of an isosceles triangle lie on the median to the base.
Thus, \textcolor{red}{(sort of)} $X_{100}$ must coincide with the vertex.
\end{proof}

\msg{X(100) could also coincide with the reflection of A wrt the circumcenter}

Analytic proof:

\begin{proof}
From \cite{ETC100}, the barycentric coordinates for $X_{100}$ are $$\Bigl(a (a - b) (-a + c): (a - b) b (b - c):(b - c) c (-a + c)\Bigr).$$
When $b=c$, this has the form $(p:0:0)$ which is vertex $A$.
\end{proof}
}

\newpage

\begin{theorem}
\label{thm:biSteinX214}
Let $E$ be the Steiner point of bicentric quadrilateral $ABCD$.
Let $F$, $G$, $H$, and $I$ be the $X_{214}$-points of triangles $\triangle ABE$, $\triangle BCE$, $\triangle CDE$,
and $\triangle DAE$, respectively.
Then $FGHI$ is a cyclic quadrilateral.
\end{theorem}

\begin{figure}[h!t]
\centering
\includegraphics[width=0.4\linewidth]{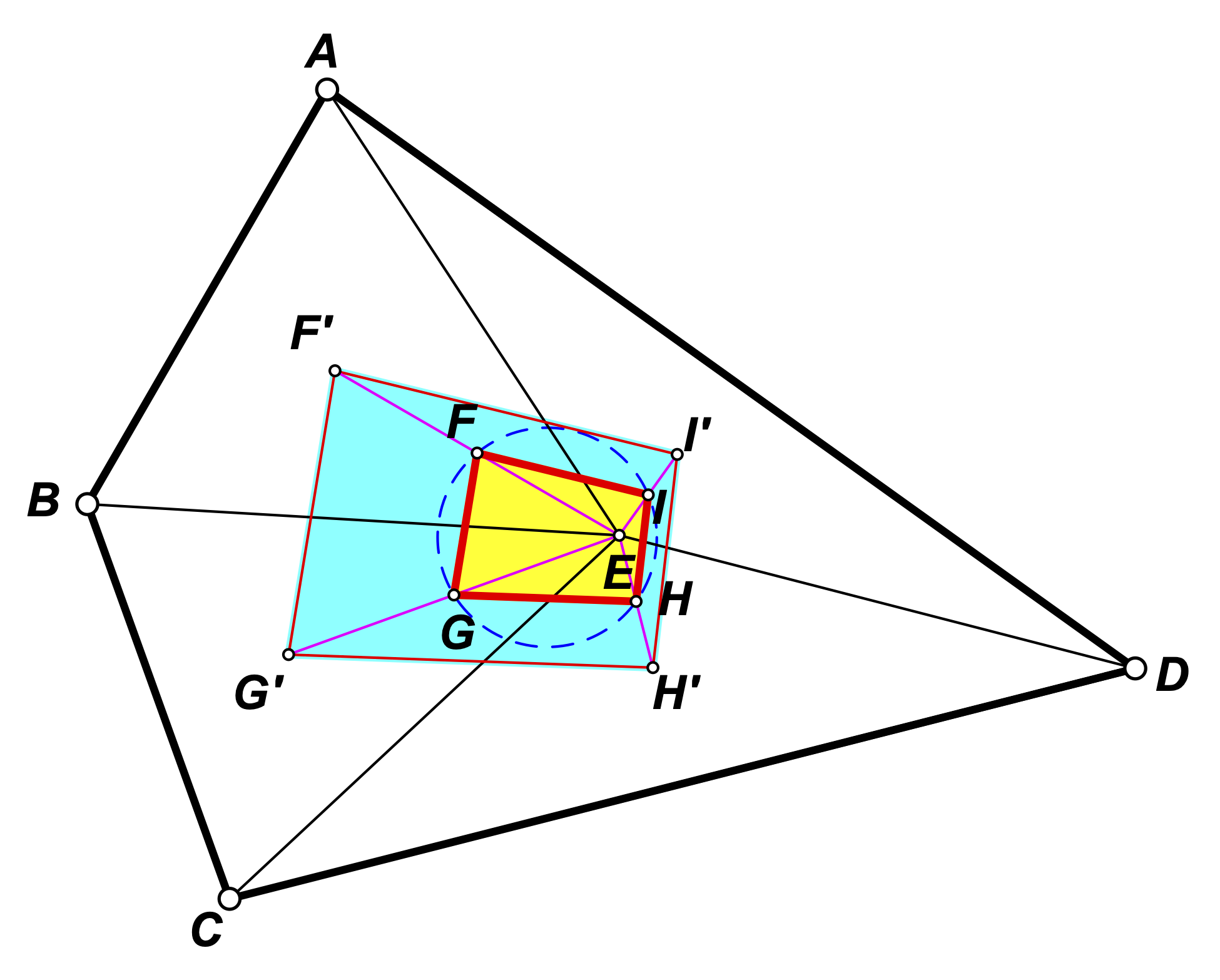}
\caption{bicentric quadrilateral, $X_{214}$-points $\implies$ cyclic}
\label{fig:biSteinX214}
\end{figure}

\begin{proof}
Let $F'$, $G'$, $H'$, and $I'$ be the $X_{1}$-points of triangles $\triangle ABE$, $\triangle BCE$, $\triangle CDE$,
and $\triangle DAE$, respectively.\\
Let $F$ be the midpoint of $EF'$.\\
Proposition~\ref{prop:circumcenter} $\implies$ $E$ is the circumcenter of $ABCD$.\\
$EA=EB$ $\implies$ $EAB$ isosceles triangle
$\implies$ $F'$ lies on  median to  base, $AB$.\\
Lemma~\ref{lemma:dpIsoscelesTriangleA} $\implies$ $E$ is the $X_{100}$-point of $\triangle EAB$.\\
Lemma~\ref{lemma:X214} $\implies$ $F$ is the $X_{214}$-point of triangle $\triangle EAB$.\\
A dilation of $\frac12$ about $E$ maps $F'$ to $F$.\\
Similarly, this dilation maps $F'G'H'I'$ to $FGHI$.\\
Theorem~\ref{thm:biSteinX1} $\implies$ $F'G'H'I'$ is cyclic.\\
Thus, $FGHI$ is homothetic to $F'G'H'I'$ and is therefore also cyclic. 
\end{proof}


\begin{lemma}
\label{lemma:X165}
The $X_{165}$-point of a triangle lies on the line $X_1X_3$ and $X_1X_3=3X_3X_{165}$.
\end{lemma}

\begin{proof}
From \cite{ETC165}, we have that $X_{165}$ lies on the line $X_1X_3$.
Using the barycentric coordinates for $X_1$, $X_3$, and $X_{165}$
along with the distance formula between two points, we find that
$$\overline{X_1X_3}^2=-\frac{a b c \left(a^3-a^2 b-a^2 c-a
   b^2+3 a b c-a c^2+b^3-b^2 c-b
   c^2+c^3\right)}{(a-b-c) (a+b-c)
   (a-b+c) (a+b+c)}$$
and
$$\overline{X_3X_{165}}^2=-\frac{a b c \left(a^3-a^2 b-a^2 c-a
   b^2+3 a b c-a c^2+b^3-b^2 c-b
   c^2+c^3\right)}{9 (a-b-c) (a+b-c)
   (a-b+c) (a+b+c)}.$$
Thus, $$\left(\frac{X_1X_3}{X_3X_{165}}\right)^2=9$$
and we are done.
\end{proof}

\begin{theorem}
\label{thm:biSteinX165}
Let $E$ be the Steiner point of bicentric quadrilateral $ABCD$.
Let $F$, $G$, $H$, and $I$ be the $X_{165}$-points of triangles $\triangle ABE$, $\triangle BCE$, $\triangle CDE$,
and $\triangle DAE$, respectively.
Then $FGHI$ is a cyclic quadrilateral.
\end{theorem}

\begin{figure}[h!t]
\centering
\includegraphics[width=0.4\linewidth]{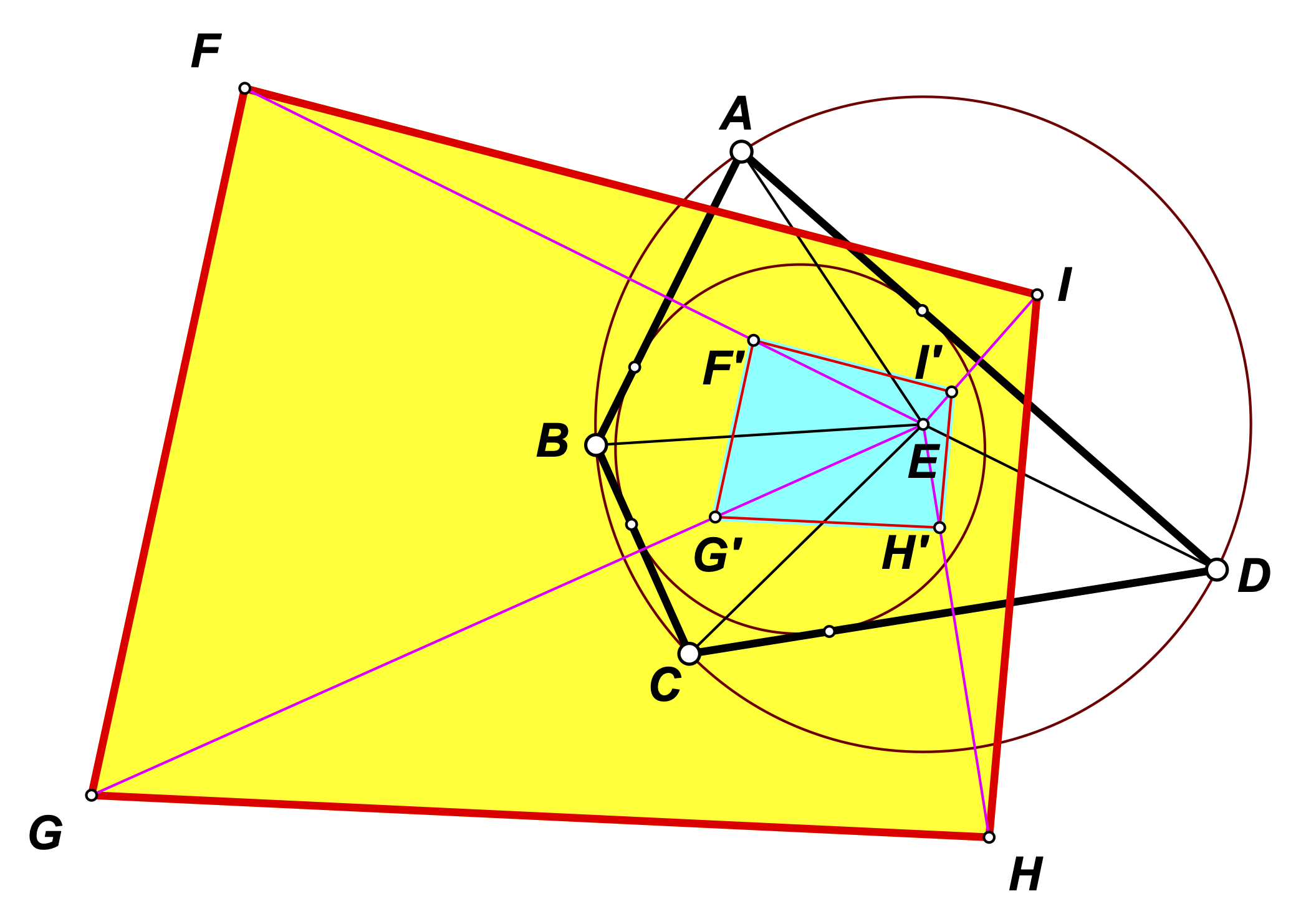}
\caption{bicentric quadrilateral, $X_{165}$-points $\implies$ cyclic}
\label{fig:biSteinX165}
\end{figure}

\begin{proof}
Let $F'$, $G'$, $H'$, and $I'$ be the $X_{1}$-points of triangles $\triangle ABE$, $\triangle BCE$, $\triangle CDE$,
and $\triangle DAE$, respectively.\\
Proposition~\ref{prop:circumcenter} $\implies$ $E$ is the circumcenter of $ABCD$.\\
$EA=EB$ $\implies$ $EAB$ isosceles triangle.\\
Lemma~\ref{lemma:X165} $\implies$ $F$ lies on $EF'$ and $EF=3EF'$.\\
A dilation of $3:1$ with center $E$ maps $F'$ to $F$.\\
Similarly, this dilation maps $F'G'H'I'$ to $FGHI$.\\
Theorem~\ref{thm:biSteinX1} $\implies$ $F'G'H'I'$ is cyclic.\\
Thus, $FGHI$ is homothetic to $F'G'H'I'$ and is therefore also cyclic. 
\end{proof}


Our computer study found a number of results for a cyclic orthodiagonal quadrilateral.
They are shown in the following table.

\bigskip
\begin{center}
\begin{tabular}{|l|p{3.3in}|}
\hline
\multicolumn{2}{|c|}{\textbf{\color{blue}\large \strut Central Quads of Cyclic Orthodiagonal Quadrilaterals}}\\ \hline
\textbf{Shape of central quad}&\textbf{centers}\\ \hline
\ru rectangle&$\mathbb{M}$, 2, 148--150, 402, 620, 671, 903\\ \hline
\ru bicentric&\no{$\mathbb{T}$}, 3, \no{399}\\ \hline
\ru equidiagonal&13, 14, 616--619\\ \hline
\ru isosceles trapezoid&154\\ \hline
\ru trapezoid&26, 139, 155--157\\ \hline
\end{tabular}
\end{center}


\begin{theorem}
\label{thm:coStein}
Let $E$ be the Steiner point of cyclic orthodiagonal quadrilateral $ABCD$.
Let $n$ be 2, 11, 115, 116, 122-125, 127, 130, 134--137, 139, 148--150, 244--247, 290, 338, 339, 402, 620, 671, 865--868, or 903.
Let $F$, $G$, $H$, and $I$ be the $X_n$-points of triangles $\triangle ABE$, $\triangle BCE$, $\triangle CDE$,
and $\triangle DAE$, respectively.
Then $FGHI$ is a rectangle.
Moreover, if $n=$148, 149, or 150, then $FGHI$ and $ABCD$ have the same diagonal point.
\end{theorem}

\void{
Note that by Theorem~\ref{thm:cqStein}, $FGHI$ is a parallelogram.

\msg{We would be done if we could state in Theorem~\ref{thm:cqStein} that
the sides of the parallelogram were parallel to the diagonals of the reference quadrilateral.}
}

\begin{theorem}
Let $E$ be the Steiner point of cyclic orthodiagonal quadrilateral $ABCD$.
Let $F$, $G$, $H$, and $I$ be the $X_{154}$-points of triangles $\triangle ABE$, $\triangle BCE$, $\triangle CDE$,
and $\triangle DAE$, respectively.
Then $FGHI$ is an isosceles trapezoid.
\end{theorem}

\begin{theorem}
Let $E$ be the Steiner point of cyclic orthodiagonal quadrilateral $ABCD$.
Let $n$ be 26, 139, 155, 156, or 157.
Let $F$, $G$, $H$, and $I$ be the $X_n$-points of triangles $\triangle ABE$, $\triangle BCE$, $\triangle CDE$,
and $\triangle DAE$, respectively.
Then $FGHI$ is a trapezoid.
\end{theorem}

\newpage

\begin{theorem}
\label{thm:coSteinX3}
Let $E$ be the Steiner point (circumcenter) of cyclic orthodiagonal quadrilateral $ABCD$.
Let $F$, $G$, $H$, and $I$ be the circumcenters of triangles $\triangle ABE$, $\triangle BCE$, $\triangle CDE$,
and $\triangle DAE$, respectively (Figure~\ref{fig:coSteinX3}).
Then $FGHI$ is a bicentric quadrilateral with incenter $E$.
\end{theorem}

\begin{figure}[h!t]
\centering
\includegraphics[width=0.4\linewidth]{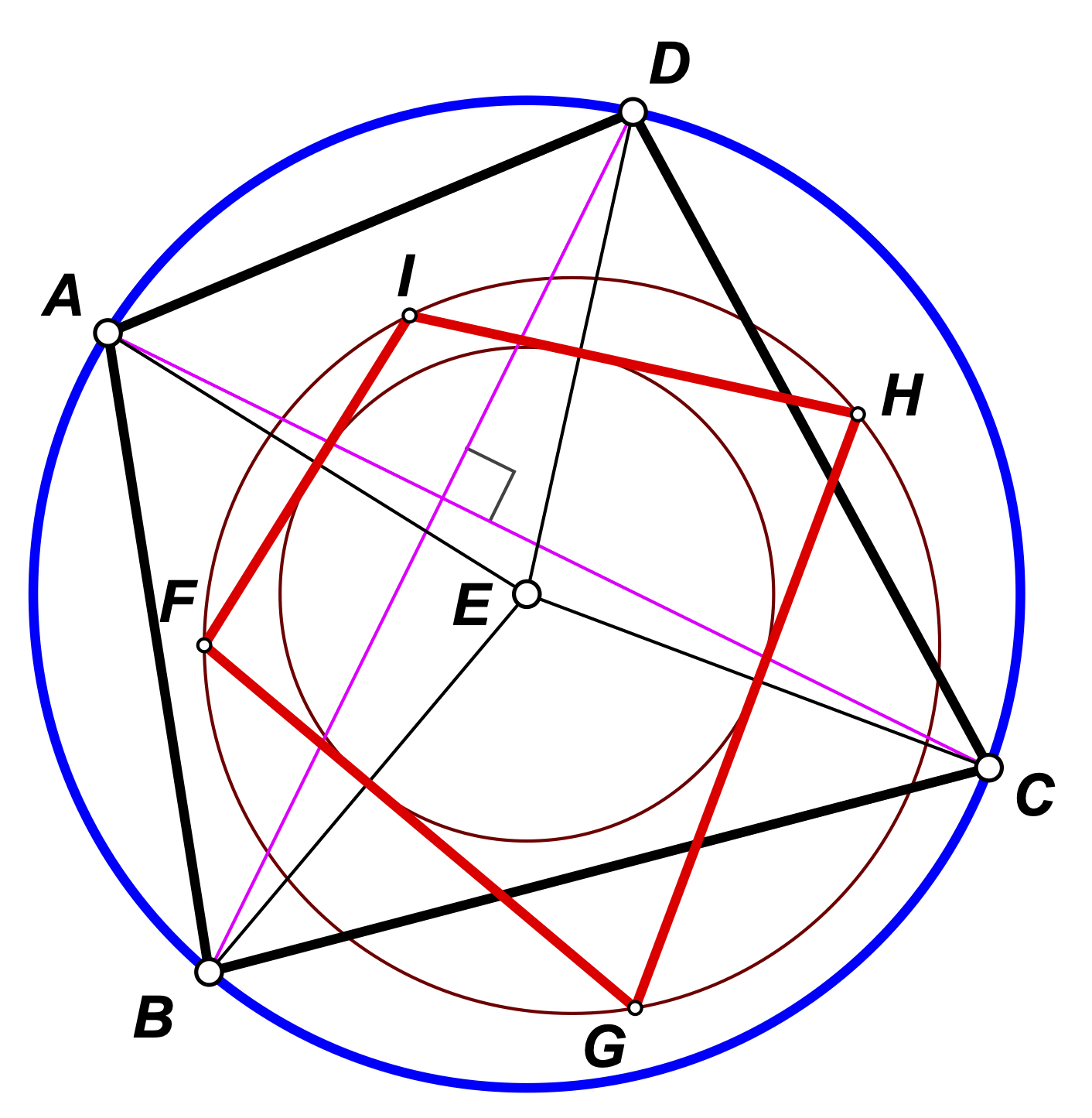}
\caption{cyclic orthodiagonal quadrilateral, $X_{3}$-points $\implies$ bicentric}
\label{fig:coSteinX3}
\end{figure}

\begin{proof}
Since $ABCD$ is cyclic, $FGHI$ is tangential by Theorem~\ref{thm:cqSteinX3}.
Since $ABCD$ is orthodiagonal, $FGHI$ is cyclic by Theorem~\ref{thm:odSteinX3}.
Thus, $FGHI$ is bicentric.
\end{proof}

\begin{open}
Is there a purely geometric proof of Theorem~\ref{thm:coSteinX3}?
\end{open}

\begin{theorem}
\label{thm:coSteinX14}
Let $E$ be the Steiner point of a cyclic orthodiagonal quadrilateral $ABCD$.
Let $n$ be 13, 14, 616, 617, 618, or 619.
Let $F$, $G$, $H$, and $I$ be the $X_{n}$-points of triangles $\triangle ABE$, $\triangle BCE$, $\triangle CDE$,
and $\triangle DAE$, respectively (Figure~\ref{fig:coSteinX14} shows the case when $n=14$).
Then $FGHI$ is equidiagonal.
\end{theorem}

\begin{figure}[h!t]
\centering
\includegraphics[width=0.4\linewidth]{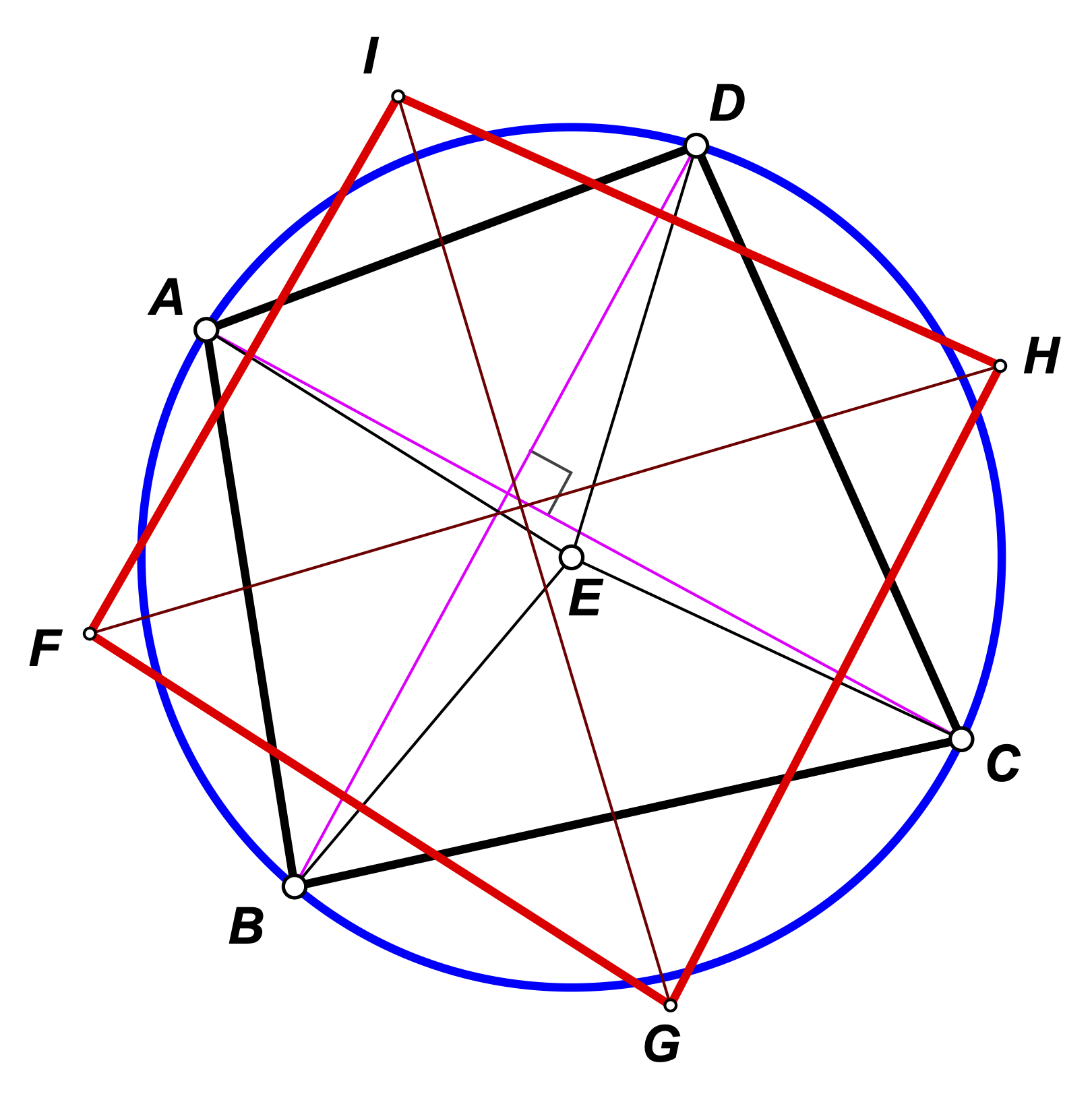}
\caption{cyclic orthodiagonal quad, $X_{14}$-points $\implies$ $FH=GI$}
\label{fig:coSteinX14}
\end{figure}

Note: In Figure~\ref{fig:coSteinX14}, $FH$ is not perpendicular to $GI$.

\begin{open}
Is there a purely geometric proof of Theorem~\ref{thm:coSteinX14} for the cases where $n$ is 13 or 14?
\end{open}

\newpage

Our computer study found a number of results for Equidiagonal Orhodiagonal Trapezoids.
They are shown in the following table.

\bigskip
\begin{center}
\begin{tabular}{|l|p{3.3in}|}
\hline
\multicolumn{2}{|c|}{\textbf{\color{blue}\large \strut Central Quads of Equidiagonal Orthodiagonal Trapezoids}}\\ \hline
\textbf{Shape of central quad}&\textbf{centers}\\ \hline
\ru square&$\mathbb{M}$, 2, 148--150, 290, 402, 620, 671, 903\\ \hline 
\ru equidiagonal kite&13, 14, 616--618\\ \hline
\ru right kite&$\mathbb{T}$, 3, 399\\ \hline
\end{tabular}
\end{center}


Note: An Equidiagonal Orthodiagonal Trapezoid is an isosceles trapezoid and is cyclic.
As an isosceles trapezoid, the central quadrilateral is always a kite.

\begin{theorem}
\label{thm:eotSteinX14}
Let $E$ be the Steiner point of an equidiagonal orthodiagonal trapezoid $ABCD$.
Let $n$ be 13, 14, 616, 617, or 618.
Let $F$, $G$, $H$, and $I$ be the $X_{n}$-points of triangles $\triangle ABE$, $\triangle BCE$, $\triangle CDE$,
and $\triangle DAE$, respectively (Figure~\ref{fig:eotSteinX14} shows the case when $n=14$).
Then $FGHI$ is an equidiagonal kite.
\end{theorem}

\begin{figure}[h!t]
\centering
\includegraphics[width=0.55\linewidth]{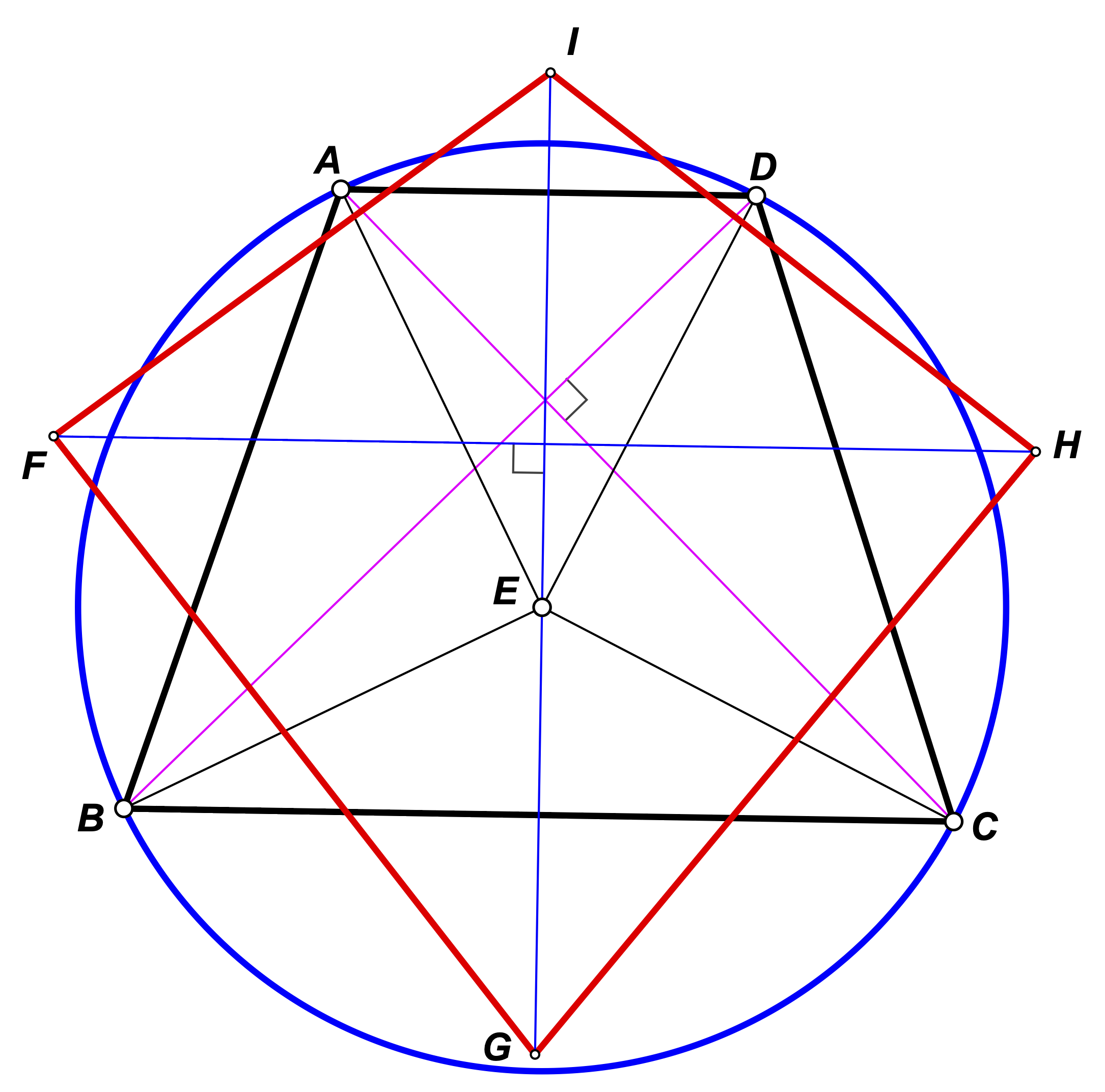}
\caption{equi ortho trap, $X_{14}$-points $\implies$ equidiagonal kite}
\label{fig:eotSteinX14}
\end{figure}

\begin{proof}
Quadrilateral $FGHI$ is a kite by Theorem~\ref{thm:itCen}.
Quadrilateral $FGHI$ is equidiagonal by Theorem~\ref{thm:coSteinX14}.
Thus  $FGHI$ is an equidiagonal kite.
\end{proof}

\begin{theorem}
Let $E$ be the Steiner point of an equidiagonal orthodiagonal trapezoid $ABCD$.
Let $n$ be 3 or 399 or in $\mathbb{T}$.
Let $F$, $G$, $H$, and $I$ be the $X_{n}$-points of triangles $\triangle ABE$, $\triangle BCE$, $\triangle CDE$,
and $\triangle DAE$, respectively.
Then $FGHI$ is a right kite.
\end{theorem}

\newpage

Our computer study also found some results for Equidiagonal Kites.
They are shown in the following table.

\bigskip
\begin{center}
\begin{tabular}{|l|p{3.3in}|}
\hline
\multicolumn{2}{|c|}{\textbf{\color{blue}\large \strut Central Quads of Equidiagonal Kites}}\\ \hline
\textbf{Shape of central quad}&\textbf{centers}\\ \hline
\ru square&486, 642\\ \hline
\ru rectangle&586\\ \hline
\end{tabular}
\end{center}

\begin{theorem}
\label{thm:ekSteinX486}
Let $E$ be the Steiner point of an equidiagonal kite $ABCD$.
Let $n$ be 486 or 642.
Let $F$, $G$, $H$, and $I$ be the $X_{n}$-points of triangles $\triangle ABE$, $\triangle BCE$, $\triangle CDE$,
and $\triangle DAE$, respectively.
Then $FGHI$ is a square.
\end{theorem}

\begin{proof}
The case $n=486$ follows from Theorem~10.9 of \cite{relationships}.\\
The case $n=642$ follows from Theorem~10.7 of \cite{relationships}.
\end{proof}

\begin{theorem}
\label{thm:ekSteinX586}
Let $E$ be the Steiner point of an equidiagonal kite $ABCD$.
Let $F$, $G$, $H$, and $I$ be the $X_{586}$-points of triangles $\triangle ABE$, $\triangle BCE$, $\triangle CDE$,
and $\triangle DAE$, respectively.
Then $FGHI$ is a rectangle.
\end{theorem}

\newpage

\section{Results Using the Poncelet Point}
\label{section:PonceletPoint}

The \emph{Poncelet point} (sometimes called the Euler-Poncelet point) of a quadrilateral
is the common point of the nine-point circles of the component triangles (half-triangles)
of the quadrilateral. A triangle formed from three vertices of a quadrilateral is called
a \emph{component triangle} of that quadrilateral.
The \emph{nine-point circle} of a triangle is the circle through the midpoints
of the sides of that triangle.

Figure \ref{fig:ppPonceletPoint} shows the Poncelet point of quadrilateral $ABCD$.
The yellow points represent the midpoints of the sides and diagonals of the quadrilateral.
The component triangles are $BCD$, $ACD$, $ABD$, and $ABC$.
The blue circles are the nine-point circles of these triangles.
The common point of the four circles is the Poncelet point (shown in green).

\begin{figure}[h!t]
\centering
\includegraphics[width=0.4\linewidth]{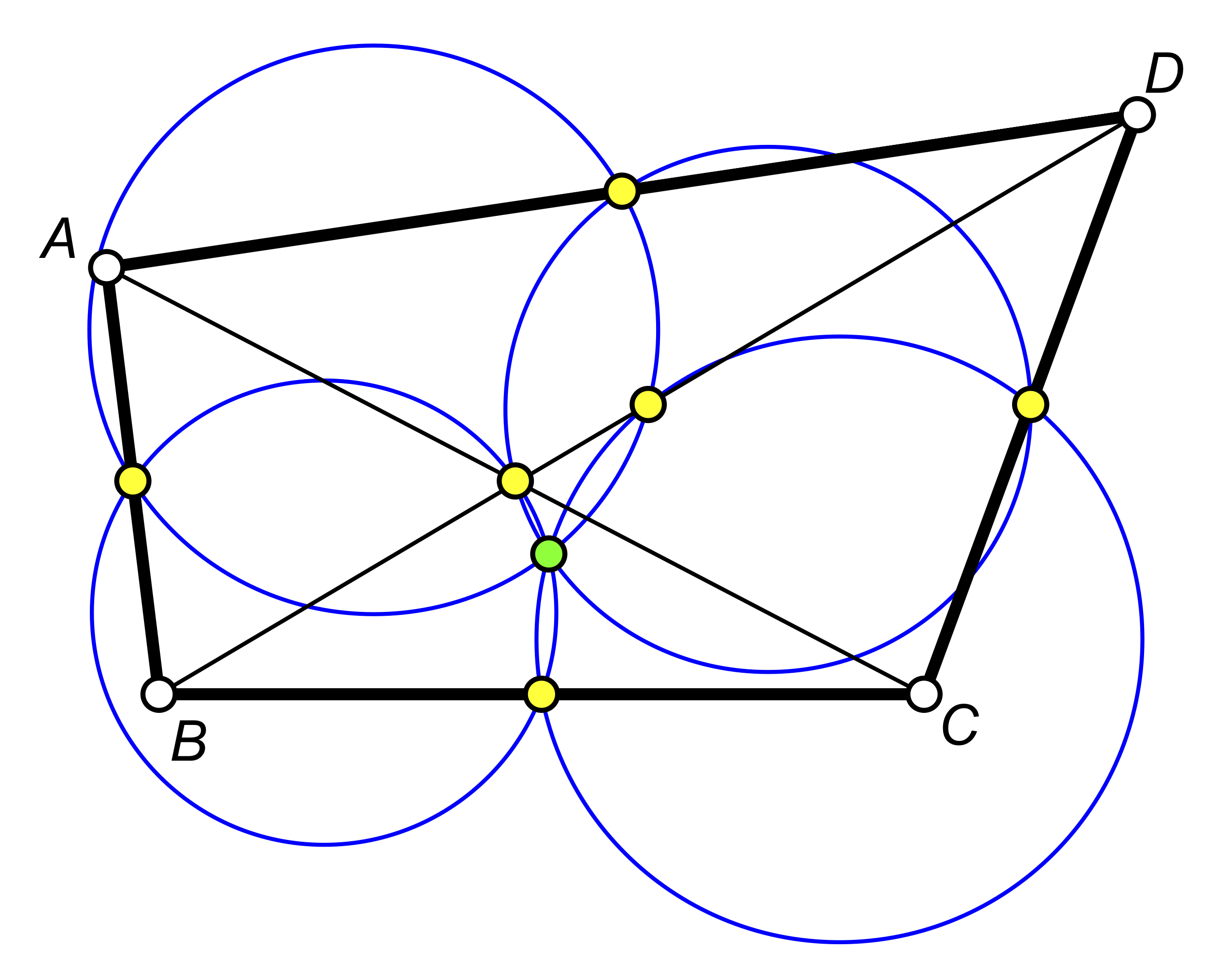}
\caption{The Poncelet point of quadrilateral $ABCD$}
\label{fig:ppPonceletPoint}
\end{figure}

In this section, we study the case where point $E$ is the Poncelet point of the quadrilateral.
Results that are true when point $E$ is arbitrary are omitted.

The following two propositions come from \cite{relationships}.

\begin{proposition}
\label{prop:ppParallelogram}
The Poncelet point of a parallelogram coincides with the diagonal point.
\end{proposition}

\begin{proposition}
\label{prop:ppOrtho}
The Poncelet point of an orthodiagonal quadrilateral coincides with the diagonal point.
\end{proposition}

Because of these propositions, we will not discuss parallelograms or orthodiagonal quadrilaterals in this section
because they have been covered in Section 5 of \cite{shapes}.

Our computer study found a few results when the quadrilateral is not a parallelogram or orthodiagonal.
They are shown in the following table.

\bigskip
\begin{center}
\begin{tabular}{|l|p{3.3in}|}
\hline
\multicolumn{2}{|c|}{\textbf{\color{blue}\large \strut Central Quads of Hjelmslev Quadrilaterals}}\\ \hline
\textbf{Shape of central quad}&\textbf{centers}\\ \hline
\ru trapezoid&3, 69\\ \hline
\ru tangential trapezoid&4\\ \hline
\end{tabular}
\end{center}

\begin{lemma}
\label{lemma-ppHjelmslevMidpoint}
Let $E$ be the Poncelet point of quadrilateral $ABCD$ that has right angles at $B$ and $D$.
Then $E$ is the midpoint of $BD$.
\end{lemma}

\begin{theorem}
\label{theorem:gqDiagX3}
Let $E$ be any point on diagonal $BD$ of quadrilateral $ABCD$.
Let $F$, $G$, $H$, and $I$ be the $X_{3}$ points of $\triangle EAB$, $\triangle EBC$, 
$\triangle ECD$, and $\triangle EDA$, respectively (Figure~\ref{fig:gqDiagX3}).
Then $FGHI$ is a trapezoid with $FG\parallel HI$.
\end{theorem}

\begin{figure}[h!t]
\centering
\includegraphics[width=0.4\linewidth]{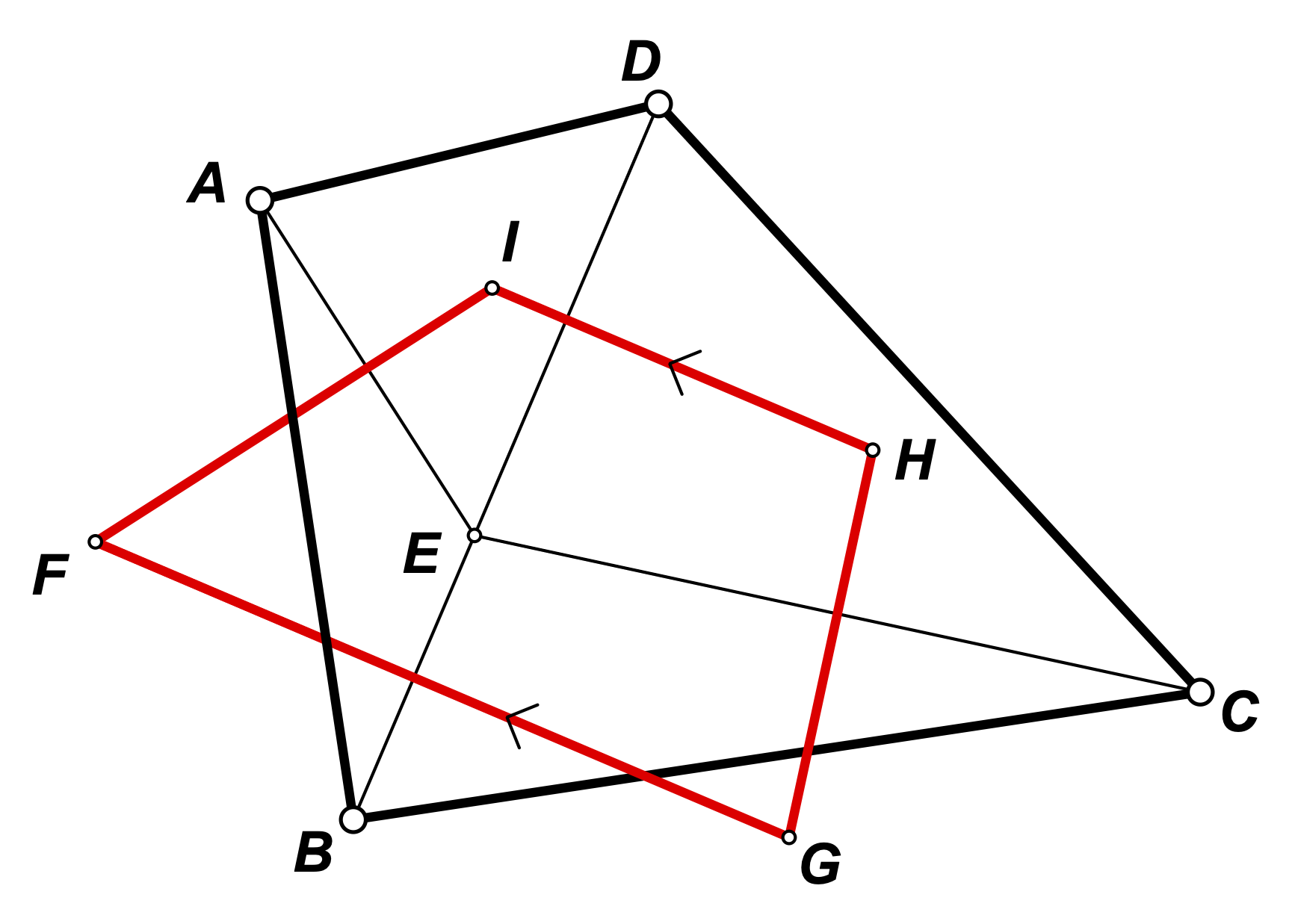}
\caption{$E\in BD$, $X_3$-points, $\implies FG\parallel HI$}
\label{fig:gqDiagX3}
\end{figure}

\begin{proof}
Point $F$ is the circumcenter of $\triangle ABE$, so $F$ lies on the perpendicular bisector of $BE$. Similarly $G$ lies on the perpendicular bisector of $BE$. Hence $FG\perp BE$. Similarly $IH\perp DE$. Since $E$ lies on $BD$, we have that $FG$ and $IH$ are perpendicular to $BD$.
Therefore $FG\parallel IH$. 
\end{proof}

\begin{theorem}
\label{theorem:hjPonX3}
Let $E$ be the Poncelet point of a Hjelmslev quadrilateral $ABCD$.
Let $F$, $G$, $H$, and $I$ be the $X_{3}$ points of $\triangle EAB$, $\triangle EBC$, 
$\triangle ECD$, and $\triangle EDA$, respectively.
Then $FGHI$ is a trapezoid with $FG\parallel HI$.
\end{theorem}

\begin{proof}
By Lemma~\ref{lemma-ppHjelmslevMidpoint}, $E$ is the midpoint of $BD$.
Then, by Theorem~\ref{theorem:gqDiagX3}, $FGHI$ is a trapezoid with $FG\parallel HI$.
\end{proof}

\begin{theorem}
\label{theorem:gqDiagX4}
Let $E$ be any point on diagonal $BD$ of quadrilateral $ABCD$.
Let $F$, $G$, $H$, and $I$ be the $X_{4}$ points of $\triangle EAB$, $\triangle EBC$, 
$\triangle ECD$, and $\triangle EDA$, respectively (Figure~\ref{fig:gqDiagX4}).
Then $FHGI$ is a trapezoid with $FI\parallel HG$.
\end{theorem}

\begin{figure}[h!t]
\centering
\includegraphics[width=0.4\linewidth]{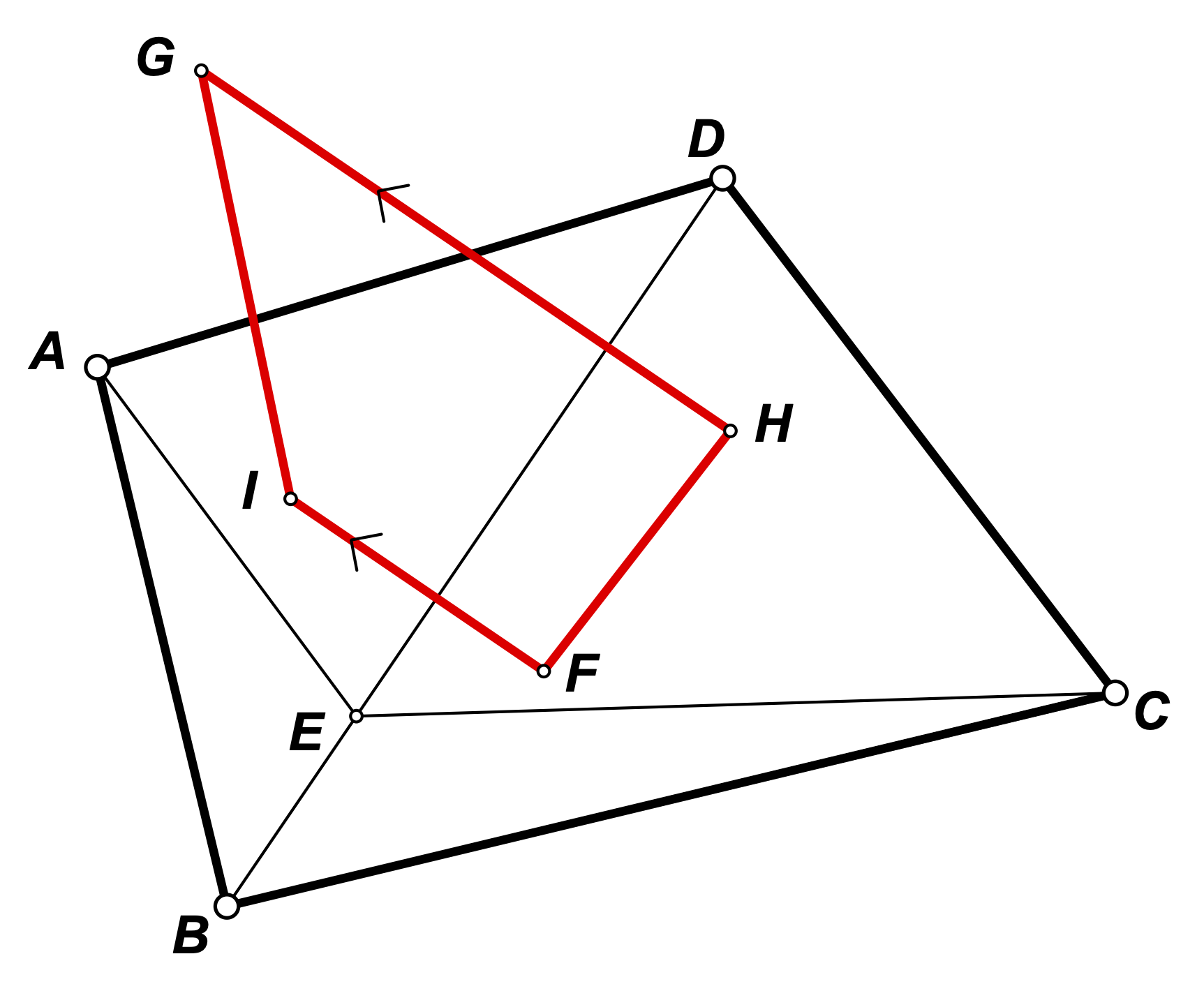}
\caption{$E\in BD$, $X_4$-points, $\implies FI\parallel HG$}
\label{fig:gqDiagX4}
\end{figure}

\begin{proof}
Point $F$ is the orthocenter of $\triangle ABE$, so $AF\perp BE$.
Point $I$ is the orthocenter of $\triangle ADE$, so $AI\perp DE$.
Since $E$ lies on $BD$, we have that $AF$ and $AI$ are both perpendicular to $BD$.
It follows that $FI\perp BD$ since $F$ and $I$ lie on the perpendicular from $A$ to $BD$.
In a similar way, it can be proved that $GH\perp BD$.
Therefore $FI \parallel GH$.
\end{proof}

\newpage

\begin{theorem}
\label{thm:hjPonX4}
Let $E$ be the Poncelet point of a Hjelmslev quadrilateral $ABCD$.
Let $F$, $G$, $H$, and $I$ be the $X_{4}$ points of $\triangle EAB$, $\triangle EBC$, 
$\triangle ECD$, and $\triangle EDA$, respectively (Figure~\ref{fig:hjPonX4}).
Then $FGHI$ is a tangential trapezoid with incenter $E$ and $FI\parallel GH$.
Also, diagonal $BD$ passes through the points of contact of the incircle with sides $GH$ and $FI$.
\end{theorem}

Note that $E$ is the midpoint of $BD$ by Lemma~\ref{lemma-ppHjelmslevMidpoint}.

\begin{figure}[h!t]
\centering
\includegraphics[width=0.5\linewidth]{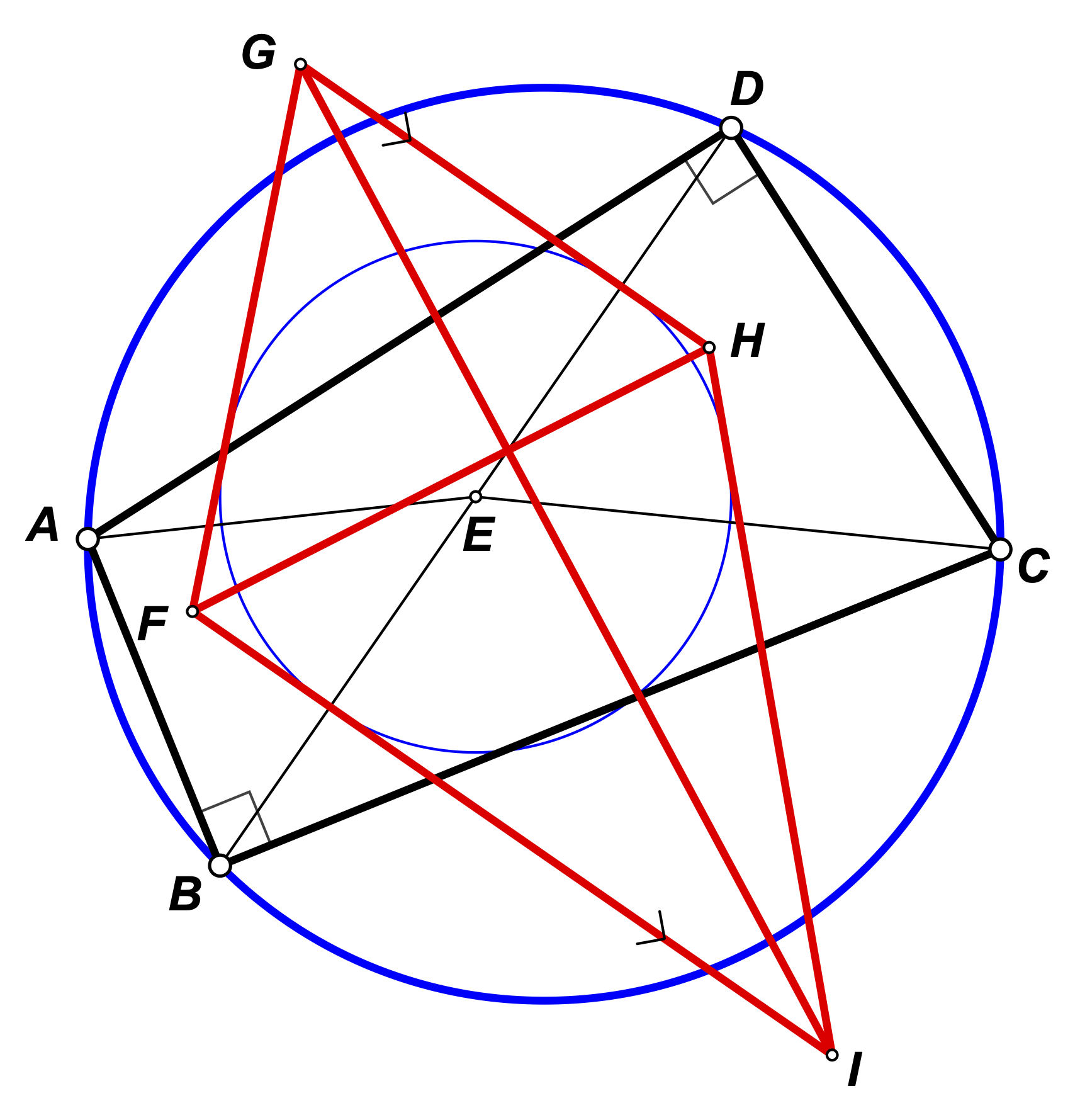}
\caption{$X_4$-points $\implies$ $FGHI$ is a tangential trapezoid}
\label{fig:hjPonX4}
\end{figure}

\begin{open}
Is there a purely geometric proof of Theorem~\ref{thm:hjPonX4}?
\end{open}

\begin{open}
If $E$ is the Poncelet point of a Hjelmslev quadrilateral $ABCD$, and the central quadrilateral
is a tangential trapezoid, must the center be the orthocenter?
\end{open}

\newpage


\section{Results Using the Area Centroid}
\label{section:areaCentroid}

The \emph{Area Centroid} (also called the \emph{quasi centroid} or 1st QG quasi centroid) of a convex quadrilateral
is the the center of mass of the quadrilateral when its surface is made of some evenly distributed material.

Geometrically, it is the intersection of the diagonals of the centroid quadrilateral of the given quadrilateral.

A triangle formed from three vertices of a quadrilateral is called
a \emph{component triangle} of that quadrilateral.

The quadrilateral whose vertices are the centroids of the four component triangles of
a quadrilateral is called the \emph{centroid quadrilateral} of that quadrilateral.

Figure \ref{fig:quasiCentroid} shows the area centroid of quadrilateral $ABCD$.
The yellow points represent the centroids of the component triangles of the quadrilateral.
The component triangles are $BCD$, $ACD$, $ABD$, and $ABC$.
The blue region is the centroid quadrilateral $G_AG_BG_CG_D$.
The red point is the area centroid.

\begin{figure}[h!t]
\centering
\includegraphics[width=0.5\linewidth]{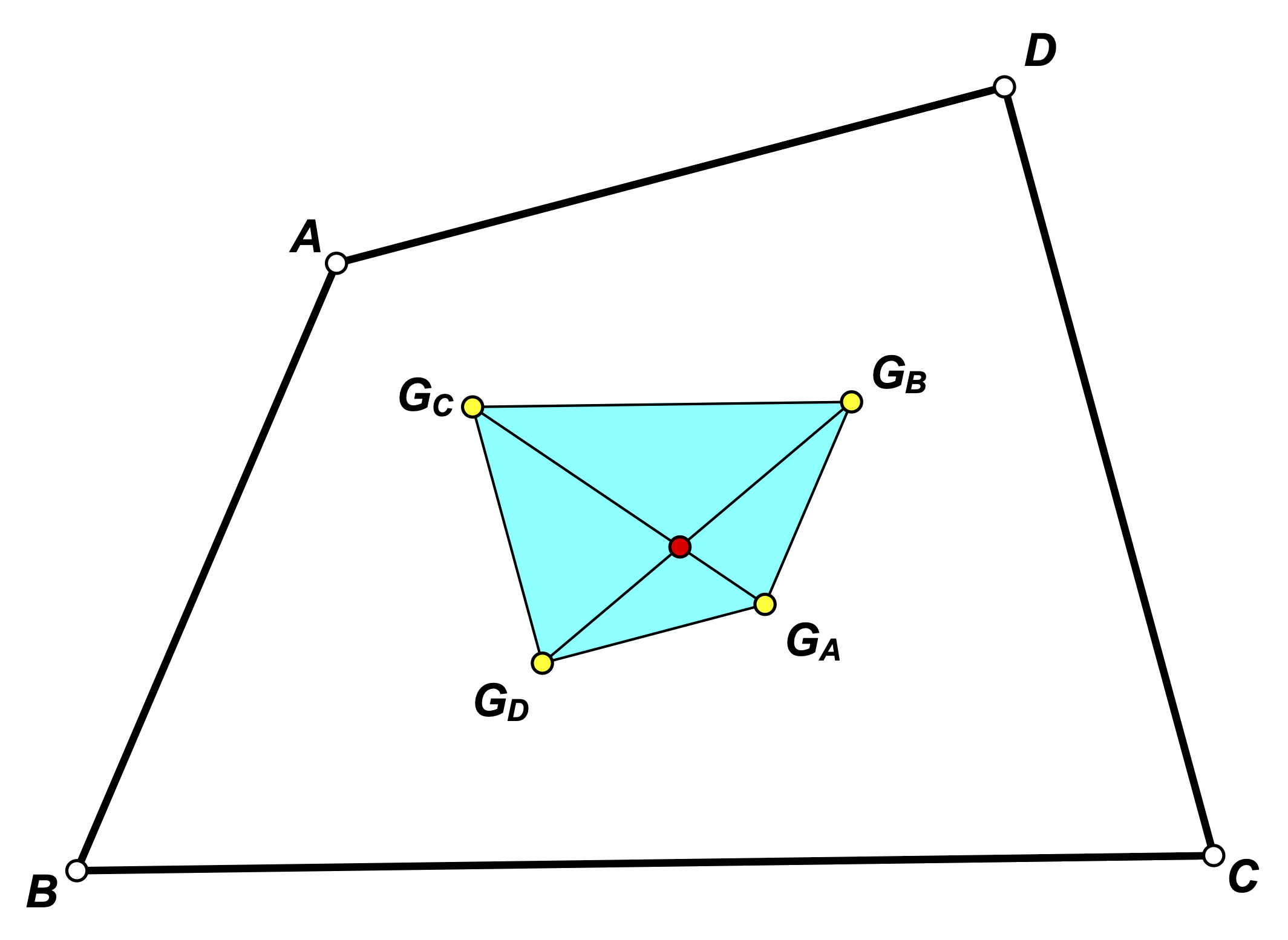}
\caption{The area centroid of quadrilateral $ABCD$}
\label{fig:quasiCentroid}
\end{figure}

In this section, we study the case where point $E$ is the area centroid of the quadrilateral.
Results that are true when point $E$ is arbitrary are omitted.

If $ABCD$ is a kite, with $AB=AD$ and $CB=CD$, then by symmetry,
the area centroid lies on diagonal $AC$. By Theorem~\ref{thm:kiCen},
the central quadrilateral of $ABCD$ is an isosceles trapezoid.

Our computer study did not find any new results, other than the ones that are
true when $E$ is an arbitrary point or when the quadrilateral is a kite.

\begin{center}
\begin{tabular}{|l|p{2.2in}|}
\hline
\multicolumn{2}{|c|}{\textbf{\color{blue}\large \strut Central Quadrilaterals of all Quadrilaterals}}\\ \hline
\multicolumn{2}{|c|}{No new relationships were found.}\\ \hline
\end{tabular}
\end{center}

\newpage

\section{Areas for Future Research}

There are many avenues for future investigation.

\subsection{Investigate other triangle centers}\ \\

In our study, we only investigated triangle centers $X_n$ for $n\leq 1000$.
Extend this study to larger values of $n$.

\subsection{Use other shape quadrilaterals}\ \\

In our investigation, we only studied 28 shapes of quadrilaterals as shown in Figure~\ref{fig:quadShapes}.
There are many other shapes of quadrilaterals. Study these other shapes.

For example, we did find some results associated with a right kite
when point $E$ is the area centroid of quadrilateral $ABCD$.
A \emph{right kite} is a kite in which two opposite angles are right angles.

\void{
\bigskip
\begin{center}
\begin{tabular}{|l|p{2.5in}|}
\hline
\multicolumn{2}{|c|}{\textbf{\color{blue}\large \strut Central Quadrilateral of a Right Kite}}\\ \hline
\textbf{Shape of central quadrilateral}&\textbf{center}\\ \hline
\ru rectangle&68, 317, 577\\ \hline
\end{tabular}
\end{center}
}

\begin{theorem}
\label{thm:rkQuasiX68}
Let $E$ be the area centroid of right kite $ABCD$.
Let $n$ be 68, 317, or 577.
Let $F$, $G$, $H$, and $I$ be the $X_{n}$-points of triangles $\triangle ABE$, $\triangle BCE$, $\triangle CDE$,
and $\triangle DAE$, respectively. (Figure~\ref{fig:rkQuasiX68} shows the case where $n=68$.)
Then $FGHI$ is a rectangle. The sides of the rectangle are parallel to the diagonals of the right kite.
\end{theorem}

\begin{figure}[h!t]
\centering
\includegraphics[width=0.35\linewidth]{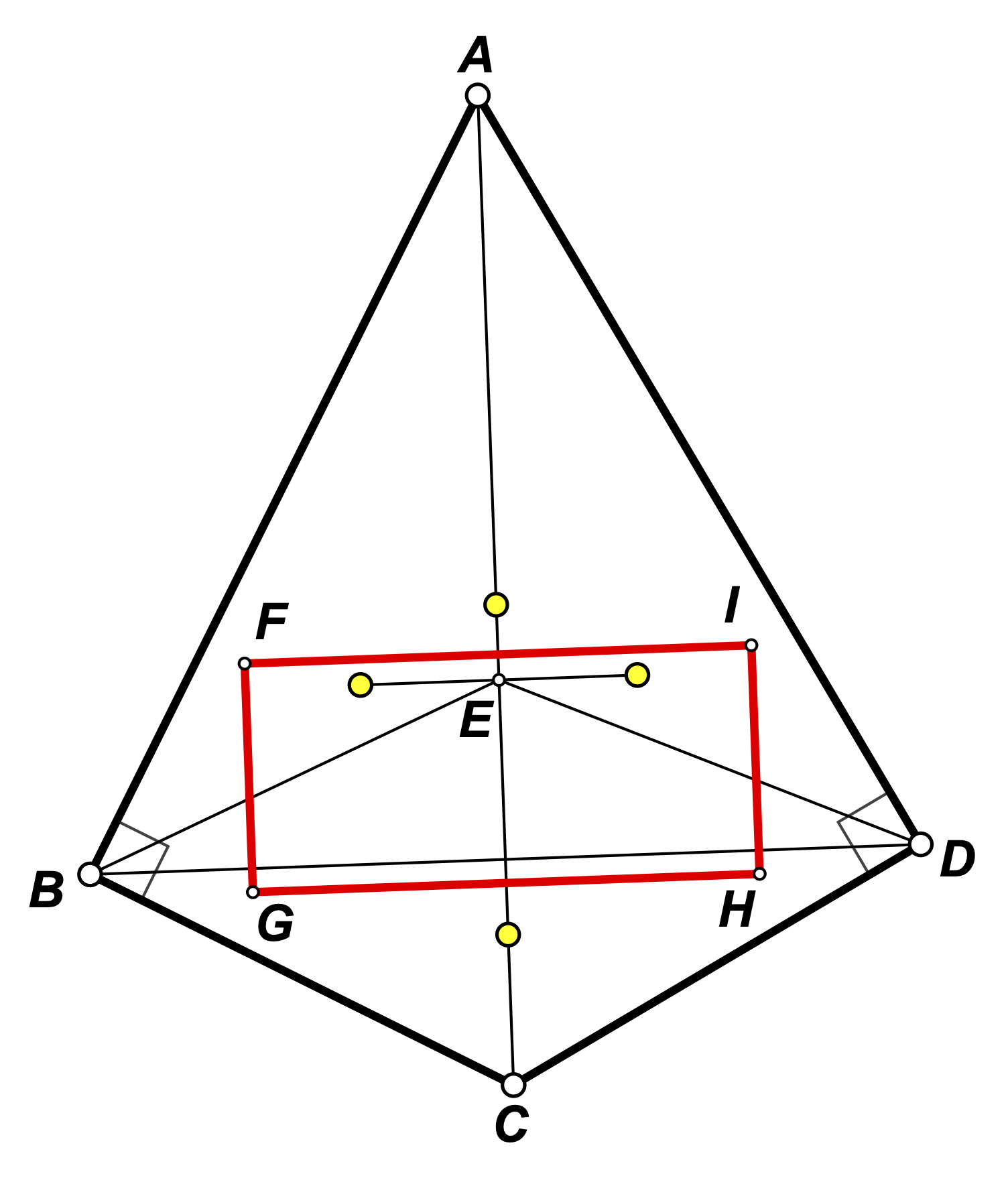}
\caption{Right Kite, $X_{68}$-points $\implies$ rectangle}
\label{fig:rkQuasiX68}
\end{figure}

\subsection{Use other quadrilateral centers}\ \\

In our study, we only used the common quadrilateral centers listed in Table~\ref{table:radiators}
as radiators.
Additional notable points associated with the reference quadrilateral  could be used, such as the Miquel point (QL-P1),
the Morley Point (QL-P2), the Newton Steiner point (QL-P7), and the various quasi points.
A list of notable points associated with a quadrilateral can be found in \cite{EQF}.

\subsection{Investigate radiators lying on quadrilateral lines}\ \\

In Section~\ref{section:misc}, we studied cases where the radiator was restricted to
lie on a certain line of symmetry of quadrilateral $ABCD$.
We could also look at cases where the radiator lies on some notable line associated with
the reference quadrilateral, 
such as a bimedian, the Newton line (QL-L1), the Steiner line (QL-L2), etc., or, in the case of cyclic quadrilaterals, the Euler line.
A list of notable lines associated with a quadrilateral can be found in \cite{EQF}.

\subsection{Ask about uniqueness}\ \\

Find an entry in one of our tables where there is only one center giving a particular relationship
for a certain type of quadrilateral. For example, in Section~\ref{section:vertexCentroid},
we found in Theorem~\ref{thm:equiCenX591} that for an equidiagonal quadrilateral, when the radiator
is the centroid of the reference quadrilateral, the central quadrilateral is orthodiagonal only when $n=591$.
Is this because we only searched the first 1000 values of $n$?
Expand the search and find other values of $n$ for which the central quadrilateral is orthodiagonal
or prove that $X_{591}$ is the unique center for which the central quadrilateral is orthodiagonal
when the radiator is the quadrilateral centroid.

\subsection{Form quadrilaterals with the Radiator}\ \\

In the current study, we formed four triangles using the radiator and two vertices of the reference quadrilateral.
Instead, form four quadrilaterals using the radiator and three vertices of the reference quadrilateral.
Then consider notable quadrilateral points in each of the four quadrilaterals formed and investigate the shape of the
quadrilateral determined by these four points.

The following results were found by computer.

\begin{theorem}
\label{thm:QgqArbCen}
Let $E$ be an arbitrary point in the plane of quadrilateral $ABCD$ (not on the boundary).
Let $F$, $G$, $H$, and $I$ be the vertex centroids of quadrilaterals $EBCD$, $EACD$, $EABD$
and $EABC$, respectively (Figure~\ref{fig:QgqArbCen}).
Then $FGHI$ is homothetic to $ABCD$.
If $P$ is the center of the homothety and $M$ is the vertex centroid of $ABCD$,
then $P$ lies on $EM$ and $PE/PM=4$.
\end{theorem}

\begin{figure}[h!t]
\centering
\includegraphics[width=0.51\linewidth]{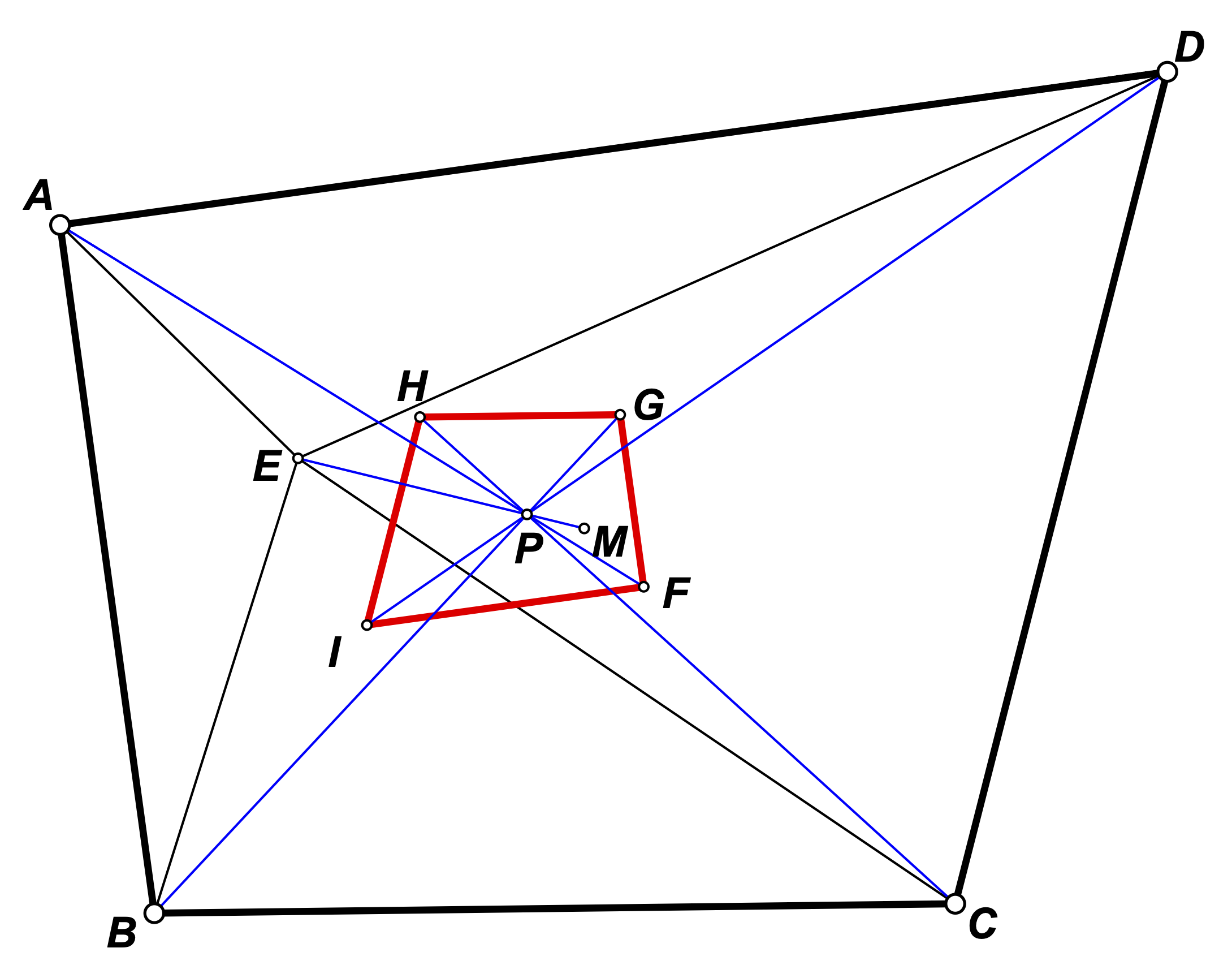}
\caption{centroids $\implies$ homothetic quadrilaterals}
\label{fig:QgqArbCen}
\end{figure}

\begin{theorem}
\label{thm:QcqArbStein}
Let $E$ be an arbitrary point in the plane of cyclic quadrilateral $ABCD$ (not on the boundary).
Let $F$, $G$, $H$, and $I$ be the Steiner points of quadrilaterals $EBCD$, $EACD$, $EABD$
and $EABC$, respectively (Figure~\ref{fig:QcqArbStein}).
Then $FGHI$ is cyclic.
If $P$ is the circumcenter of $FGHI$ and $O$ is the circumcenter of $ABCD$,
then $P$ is the midpoint of $EO$.
\end{theorem}

\begin{figure}[h!t]
\centering
\includegraphics[width=0.4\linewidth]{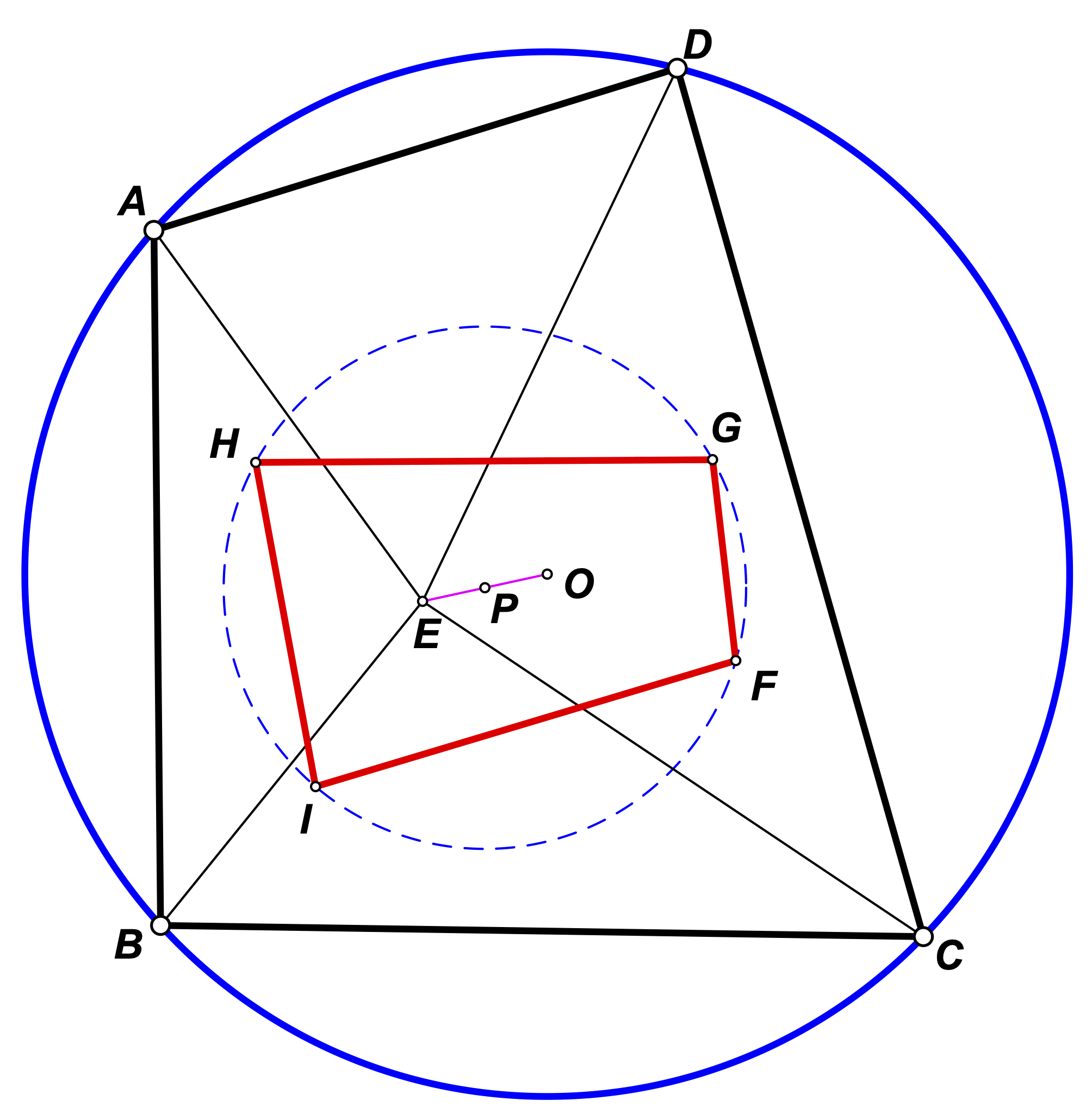}
\caption{cyclic, Steiner points $\implies$ cyclic}
\label{fig:QcqArbStein}
\end{figure}

\begin{theorem}
\label{thm:QcqArbPonce}
Let $E$ be an arbitrary point in the plane of cyclic quadrilateral $ABCD$ (not on the boundary).
Let $F$, $G$, $H$, and $I$ be the Poncelet points of quadrilaterals $EBCD$, $EACD$, $EABD$
and $EABC$, respectively (Figure~\ref{fig:QcqArbPonce}).
Then $FGHI$ is cyclic.
\end{theorem}

\begin{figure}[h!t]
\centering
\includegraphics[width=0.4\linewidth]{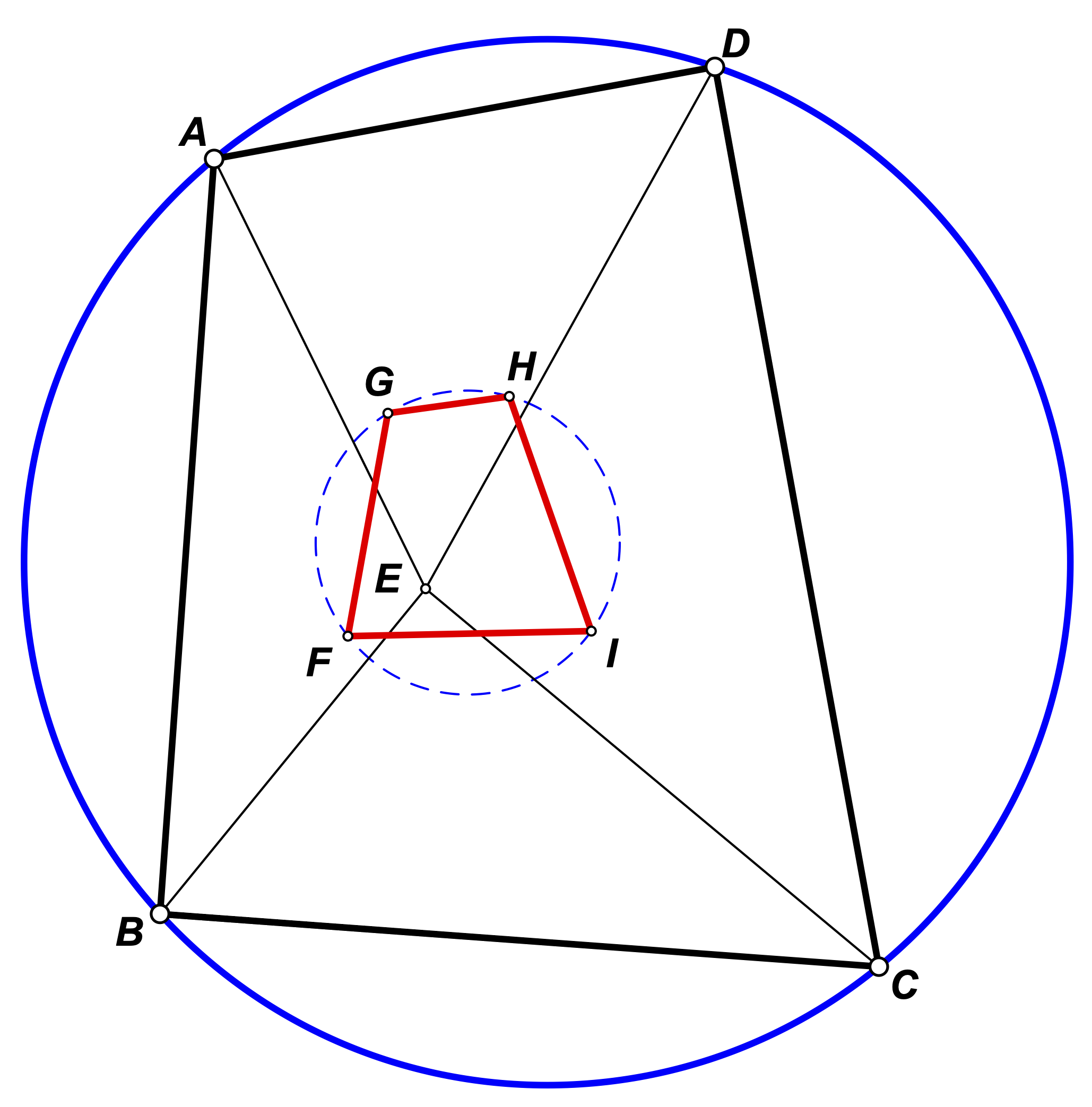}
\caption{cyclic, Poncelet points $\implies$ cyclic}
\label{fig:QcqArbPonce}
\end{figure}

\begin{open}
Are there purely geometric proofs of the previous 3 theorems?
\end{open}

\goodbreak
\subsection{Work in 3-space}\ \\

If point $D$ is moved off the plane of $\triangle ABC$, then the reference quadrilateral becomes a tetrahedron. Choose a point $E$ inside this reference tetrahedron and draw lines to each of the vertices
of the reference tetrahedron. This forms four tetrahedra with one vertex at $E$.
Locate tetrahedron centers (such as the centroid, circumcenter, or Monge point) in each
of these four tetrahedra. These centers form a new tetrahedron called the \emph{central tetrahedron}
of the given tetrahedron.
Investigate when the central tetrahedron has a special shape (such as being isodynamic, orthocentric,
or isosceles). A list of some special shape tetrahedra can be found in Section~3 of \cite{tetrahedron}.
The point $E$ can be an arbitrary point or it could be a notable point associated with
the reference tetrahedron. A list of notable points can be found in Section~4 of \cite{tetrahedron}.

The following result was discovered by computer.

\begin{theorem}
\label{thm:ZgqArbCen}
Let $E$ be any point not on the boundary of tetrahedron $ABCD$.
Let $F$, $G$, $H$, and $I$ be the centroids of tetrahedra $EBCD$, $EACD$, $EABD$
and $EABC$, respectively.
Then tetrahedron $FGHI$ is similar to tetrahedron $ABCD$.
\end{theorem}

The line from a vertex of a tetrahedron to the centroid of the opposite face is called a \emph{median}.
It is well known (Commandino's Theorem, \cite[p.~57]{Altshiller}) that the four medians of a tetrahedron concur at a point called the \emph{centroid} of that tetrahedron
and that the centroid divides each median in the ratio $1:3$.

The proof of Theorem~\ref{thm:ZgqArbCen} is similar to the proof of Theorem~\ref{thm:genArbX2}
and is omitted.



\goodbreak

\end{document}